\documentclass[preprint,1p]{elsarticle}

\usepackage{amssymb,amsmath,amsthm}
\usepackage[numbers]{natbib}
\usepackage{geometry}
\usepackage{fleqn}
\usepackage{graphicx}
\usepackage{txfonts}
\usepackage{hyperref}
\usepackage{enumitem}
\usepackage{adjustbox}
\usepackage{changepage}
\usepackage[toc,page]{appendix}
\usepackage{subcaption}
\usepackage{multirow}
\usepackage{float}

\usepackage{xspace}
\usepackage{xcolor} 

\usepackage{amssymb}
\usepackage{graphicx} 
\usepackage{xspace}
\usepackage{multirow}
\usepackage{bbm}
\usepackage{mydef}
\usepackage{remarks}
\usepackage{algorithm}
\usepackage{algpseudocode}
\usepackage[numbers]{natbib}
\usepackage{graphicx}
\usepackage{subcaption}
\usepackage{enumitem}
\usepackage{multirow}
\usepackage{adjustbox}
\usepackage{amsmath}
\usepackage{mathtools}

\let\cite=\citet
\usepackage{accents}

\begin{document}

\begin{frontmatter}
  \title{A nodal based high order nonlinear stabilization for finite element approximation of Magnetohydrodynamics}
  \tnotetext[t1]{This research is funded by Swedish Research Council (VR) under grant number 2021-04620.}
  \author[1]{Tuan Anh Dao}
  \ead{tuananh.dao@it.uu.se}
  \author[1]{Murtazo Nazarov\corref{cor1}}
  \ead{murtazo.nazarov@it.uu.se}
  \address[1]{Division of Scientific Computing, Department of Information Technology, Uppsala University, Sweden}
  \cortext[cor1]{Corresponding author}

  \begin{abstract}
    We present a novel high-order nodal artificial viscosity approach designed for solving Magnetohydrodynamics (MHD) equations. Unlike conventional methods, our approach eliminates the need for ad hoc parameters. The viscosity is mesh-dependent, yet explicit definition of the mesh size is unnecessary. Our method employs a \red{multimesh} strategy: the viscosity coefficient is constructed from a linear polynomial space constructed on the fine mesh, corresponding to the nodal values of the finite element approximation space. The residual of MHD is utilized to introduce high-order viscosity in a localized fashion near shocks and discontinuities. This approach is designed to precisely capture and resolve shocks. Then, high-order Runge-Kutta methods are employed to discretize the temporal domain. Through a comprehensive set of challenging test problems, we validate the robustness and high-order accuracy of our proposed approach for solving MHD equations.
  \end{abstract}


\end{frontmatter}

\section{Introduction}
In this article, we are interested in solving the following MHD system for \red{the conserved variables} $\bsfU:=(\rho, \bbm, E, \bB):  \Omega\times[0,\widehat t \,\,] \mapsto \polR^{2d+2}$, such that 
\begin{equation}\label{eq:mhd1}
  \p_t \bsfU + \DIV \bsfF_a(\bsfU) = 
  0,
\end{equation}
where the flux term is defined as 
\begin{equation*}
  \bsfF_a:=
  \big(
  \bbm,	
  \bbm \otimes \bu + p\polI - \bbetaa,
  \bu (E+p) -\bu \SCAL \bbetaa,
  \bB \otimes \bu - \bu \otimes \bB
  \big)^\top.
\end{equation*}
Here $\Omega\subset\mR^d$ is an open bounded domain, $d$ is the space dimension and $\widehat t>0$ is the final time, and $\bu:=\bbm / \rho$ is the velocity field. The term $\bbetaa$ is the Maxwell stress tensor,
\[
  \bbetaa := \big( -\frac12(\bB\SCAL\bB)\polI + \bB\otimes\bB \big).
\]
The thermodynamic pressure $p$ is computed from the equation of state 
% \begin{equation}
\[
  p := (\gamma-1) \l(E - \frac12 \rho (\bu \SCAL \bu) - \frac12 (\bB \SCAL \bB)\r), 
\]
% \end{equation} 
where $\gamma>0$ is the adiabatic gas constant.

The MHD system, designed to simulate fusion processes, has attracted significant attention from researchers over the last decade. This increased attention is explained to the recent successful developments in Tokamak reactor technology. State-of-the-art numerical methods for solving MHD rely on approximate Riemann solvers within Godunov schemes. Notable references for finite difference, finite volume and discontinuous Galerkin (DG) schemes include works such as \citep{Brio_Wu_1988, Powell_et_al_1991, Dai_Woodward_1998, Balsara_et_al_1999, Bouchut_et_al_2007, Balsara_2010, Balsara_Dubser_Abgrall_2014, Warburton_Karniadakis_1999, Li_Shu_2005, Dumbser_2016}.  

While the Riemann-solver approach has demonstrated high effectiveness in simulating various challenging MHD problems, it is worth acknowledging the existing difficulty in computing Riemann solutions. For instance, \cite{Torrilhon_2002} demonstrated the existence of non-unique solutions for specific Riemann problems within the MHD system due to its non-strictly hyperbolic nature and non-convex flux. Therefore, the Riemann-solver based methods become less effective in such scenarios. An alternative to the Riemann-solver-based approach is the utilization of central schemes, as demonstrated in works of, \eg \citep{Tadmor_mhd1_2004, Li_Xu_Yakovlev_2011, Cheng_et_al_2013}. Central schemes are obtained by approximating the flux term by simple quadrature formulas applied to the cells. This approach eliminates the necessity for exact or approximate Riemann solvers and was used in the finite volume and DG frameworks. 
An alternative method to avoid relying on a Riemann solver is the residual distribution approach outlined in \citep{Abgrall_2009}. This technique involves integrating the MHD system over each element $K$ of the mesh to obtain the residual. Subsequently, the residual is distributed to each node on element $K$. As nodes may be shared by multiple elements, the final step entails collecting all contributing residuals to the node and solving the discrete system.

The success of finite volume and DG methods has not been extended to finite element approximations, and one reason for this limitation is the challenge of extending Riemann solvers to finite elements. Another contributing factor could be the absence of stable, robust, and highly accurate stabilization techniques tailored for solving compressible Magnetohydrodynamics (MHD) equations. Given that the finite element approximation of the flux term is akin to central schemes, regularization of the scheme becomes necessary. A noteworthy observation in the literature review is the scarcity of references dedicated to the approximation of compressible MHD equations using the finite element method. The earliest works we found were by \citep{Shadid_2010, Shadid_et_al_2016} and \citep{Nkonga_hal_2016}. In these works, the Variational Multiscale Method based on Galerkin-Least-Square stabilization (GLS) was employed to simulate resistive MHD equations. It is worth noting that GLS-based methods are typically implicit in time, necessitating robust nonlinear iteration solvers for each time step, resulting in computational expenses. Moreover, making least-squares methods higher-order in both time and space poses challenges due to the complex nature of the stabilization terms and the implementation of time-dependent test spaces. Furthermore, while the least-square stabilizations prove highly efficient in mitigating minor spurious oscillations within smooth regions, they fall short in suppressing the Gibbs phenomenon induced by shocks and discontinuities. Addressing this limitation necessitates the incorporation of supplementary shock-capturing terms. 

Over the last decade, there has been active research focused on stabilizing finite element schemes using only nonlinear shock-capturing terms. For instance, the method discussed in \citep{Guermond_2008, Guermond_pasquetti_popov_JCP_2011, GuermondNaPo11, Nazarov_2013, Nazarov_Hoffman_2013, Nazarov_Larcher_2017} constructs shock-capturing or artificial viscosity coefficients that are proportionate to the residual of the partial differential equation (PDE) or the corresponding entropy inequality of the system. This approach has been thoroughly tested across various systems of conservation laws, including the compressible Euler equations. What makes this stabilization method particularly interesting is its potential to design schemes that achieve very high orders in both time and space.

The recent works by \citep{Dao2022a, Dao2022b} advance the application of the residual-based artificial viscosity method to solve the compressible MHD system. In these studies, the ideal MHD system regularized by introducing an artificial elliptic term of the form $-\DIV (\epsilon \GRAD \bsfU_h)$ to the system \eqref{eq:mhd1}. Here, the artificial viscosity coefficient takes the form $\epsilon \sim C h^2 |R(\bsfU_h)|$, where $R(\bsfU_h)$ denotes the MHD residual, and $\bsfU_h$ is an approximation of $\bsfU$.

While the authors demonstrate successful fourth-order convergence rates using explicit Runge-Kutta methods in time and up to third-order polynomials in space, challenges persist in the field of residual- or entropy-based stabilization. The main drawbacks of the outlined artificial viscosity construction include:
\begin{enumerate}
\item[{\em (i)}] Uncertain tunable constant: The size of the tunable constant $C$ remains unknown a priori, making it challenging to determine an appropriate value.
\item[{\em (ii)}] Unclear definition of mesh size: In the context of finite element approximations, unstructured meshes are commonly employed. However, defining the mesh size parameter $h$ in these meshes becomes ambiguous.
\end{enumerate}
Addressing these challenges will significantly enhance the robustness and applicability of the proposed method.

Recent papers by \citep{Basting_2017, Kuzmin_2020, Mabuza_2020, Dao2023} offer some solutions to the limitations mentioned, focusing on the ideal MHD equations through an edge-based artificial viscosity method. In the weak form, the viscous term in the finite element approximation takes the form $(\epsilon \GRAD \bsfU_h, \GRAD \bvarphi_i)$, where $\bvarphi_i$ (with $i \in \polN$) is the basis function of the finite element space. Introducing the matrix $\e_{ij} := (\epsilon \GRAD \bvarphi_j, \GRAD \bvarphi_i)$ for $i, j \in \polN$, the edge viscosity is expressed as
$
(\epsilon \GRAD\bsfU_h , \GRAD \bvarphi_i ) =
\sum_{j\in \polN} \e_{ij} (\bsfU_j - \bsfU_i),
$
where $\bsfU_j$ and $\bsfU_i$ represent the nodal values of $\bsfU_h$. In the cited references, the authors use continuous linear element in space that resulted to a second order scheme.

However, it is important to note that the edge-based viscosity requires assembling the viscosity matrix $\e_{ij}$ at each time step. In addition, for higher-order polynomial spaces, this matrix becomes denser, introducing an unnessesary level of diffusion in the approximation. As a consequence, the required CFL number for stability decreases rapidly as the polynomial space increases, as discussed in \citep{Guermond2024} and related works.

The objective of this paper is to develop a high-order artificial viscosity method without the need for tunable coefficients and explicit mesh-size definition. The methodology builds upon the work of \citep{Guermond_Nazarov_2014}, where the concept of reorienting meshes for artificial viscosity construction was initially introduced. In this paper, we introduce a nodal-based viscosity distinct from that presented in \citep{Guermond_Nazarov_2014}. Our approach involves utilizing two meshes to construct viscosity coefficients for higher polynomial spaces, thereby adding sufficient viscosity without compromising the hyperbolic CFL condition. Importantly, the proposed method can be easily integrated into existing nodal-based algorithms.

The work is organized as follows. In Section~\ref{sec:prelim}, we provide an overview of the finite element spaces and the necessary tools for our analysis and approximation. Section~\ref{sec:nv} introduces the finite element approximation of MHD and presents the nodal artificial viscosity method. Then, Section~\ref{sec:num} tests the proposed method on well-known MHD benchmarks. In Section~\ref{Sec:scalar}, we revisit scalar conservation laws and establish the discrete maximum principle for the nodal artificial viscosity introduced in this work. 

\section{Preliminaries}\label{sec:prelim}
In this section, we present the triangulation of the domain of interest, the construction of finite element spaces, and some definitions that will be needed in our further analysis.

\subsection{Finite element approximations}
Let us denote by $\calT_h$ a subdivision of $\Omega$ into finite number of disjoint elements $K$ such that $\overline\Omega=\cup_{k\in \calT_h} \overline K$, where $\overline \Omega$ and $\overline K$ denotes the closure of $\Omega$ and $K$, respectively. We consider a family of shape-regular and conforming meshes $\{\calT_h\}_{h>0}$, where $h$ denotes the smallest diameter of all triangles of $\calT_h$. We denote by $\bg_K: \widehat K \mapsto K$ the affine mapping that maps the reference element $\widehat K$ to $K$. 
% To facilitate further discussions, we assume that $K$ is a strictly acute triangle.

We define finite element spaces that we use. For each mesh $\calT_h$ we associate the following continuous approximation space:
\begin{equation}\label{eq:Xh}
  \calX_h := \{v_h: v_h \in \calC^0(\overline \Omega); \, \forall K\in \calT_h, \,
  v_h|_K \circ \bg_K \in \polP_k \},
\end{equation}
where $\polP_k$ is the set of multivariate polynomials of total degree at most $k\ge1$ defined over $\widehat K$. We also denote by $I$ the total Lagrange nodes in the mesh $\calT_h$, and $\{ \bN_1, \ldots, \bN_I \}$ is the collection of all the Lagrange nodes and $\{\varphi_1, \cdots, \varphi_I\}$ is the set of corresponding scalar-valued shape functions, and $S_i$ is the support of $\varphi_i$, and $S_{ij} = S_i \bigcap S_j$ is the intersection of the supports of $\varphi_i$ and $\varphi_j$. We often refer to the index set of the all degrees of freedoms by $\calV:=\{1,2,\ldots, I\}$. With $\Nel(S_i)$ we denote the number of elements in $S_i$. We denote by $\calI(S_i)$ and $\calI(K)$ the set of all indices of the shape functions living at $S_i$ and cell $K$, respectively.

We denote by $m_{ij} := \int_\Omega \varphi_j(\bx) \varphi_i(\bx) \ud \bx$ the consistent mass matrix, and by $m_i := \sum_{j\in \calV} m_{ij} = \int_\Omega \varphi_i(\bx) \ud \bx$ the lumped mass matrix. Note that, we used here the partition of unity property of the test functions, \ie $\sum_{j\in\calV} \varphi_j = 1$. 

Let us denote by $\widetilde K$ a reference equilateral triangle (or tetrahedron in 3D) whose edges equal to 1. For the given physical element $K$, we denote by $\Phi_K: \widetilde K \mapsto K$ the affine mapping to transform $\widetilde K$ to $K$, and the Jacobian matrix of this transformation by $\polJ_K$. Then by the chain rule, we get:
\[
  \GRAD(q \circ \Phi_K) = \polJ_K^\top(\GRAD q) (\Phi_K),
\]
where $\polJ_K^\top$ is the transpose of $\polJ_K$. 

Next, we denote by $\calT_h^{\text{fine}}$ the mesh whose vertices coincide with the nodes of the finite element space $\calX_h$. For example, $\calT_h^{\text{fine}} = \calT_h$, when the polynomial space is $k=1$, \ie $\polP_1$. Examples of patches $\calT_h^{\text{fine}}$ and $\calT_h$ for the case of $\polP_2$ and $\polP_3$ are shown in Figure~\ref{fig:meshes}.

In addition, we need a continuous piecewise linear finite element space on the fine mesh $\calTh^{\text{fine}}$ that will be useful in constructing the first-order viscosity used in this paper:
\begin{equation}\label{eq:Xhfine}
  \calX^{\polP_1, \text{fine}}_h := \{v_h: v_h \in \calC^0(\overline \Omega); \, \forall K\in \calT^{\text{fine}}_h, \,
  v_h|_K \circ \bg_K \in \polP_1 \},
\end{equation}
and let us denote the shape function of $\calX^{\polP_1, \text{fine}}_h$ by $\varphi_i^{\polP_1, \text{fine}}$. Therefore, the lumped mass matrix corresponding to this fine space is 
$	
m^{\polP_1, \textrm{fine}}_i := \int_{S^{\polP_1, \textrm{fine}}_i}\varphi_i^{\polP_1, \textrm{fine}} \ud \bx,
$
where $S_i^{\polP_1, \textrm{fine}} = \text{supp }(\varphi_i^{\polP_1, \text{fine}})$.

\subsection{Eigenvalues of the nonlinear system}\label{sec:eigenvalues:mhd}
For the construction of the first-order viscosity, we need to determine the largest eigenvalue of the MHD system to approximate the maximum local wave speed. Let $\be \in \mR^d$ be a direction. The eight eigenvalues corresponding to the elementary waves of the ideal MHD equations are given by
%, see \eg \citep{Powell_et_al_1991, Barth_1999, Rossmanith_2006}:
\[
  \lambda_{1,8} \coloneqq \bu\SCAL\be \mp c_{f}, \quad
  \lambda_{2,7} \coloneqq \bu\SCAL\be \mp b, \quad 
  \lambda_{3,6} \coloneqq \bu\SCAL\be \mp c_{s}, \quad
  \lambda_{4,5} \coloneqq \bu\SCAL\be,
\]
where
\[
  c_{f,s}^2 \coloneqq \frac12
  \Big(
  a^2 + \frac{\bB \SCAL \bB}{\rho}
  \Big) \pm 
  \frac12 
  \Bigg(
  \Big(
  a^2 + \frac{\bB \SCAL \bB}{\rho}
  \Big)^2 
  - 4a^2 b^2
  \Bigg)^\frac12,
\]
\[
  a:=\gamma p/\rho,\quad\quad b := \bB\SCAL\be/\rho^\frac12.
\]
The eigenvalue of the biggest magnitude is $\max_{i=1,\ldots,8}|\lambda_i|$ $=$ $\max$($|\lambda_1|$,$|\lambda_8|$). We will use this value to approximate the maximum local wave speed at a given point in space.

\subsection{Projection method \citep{Brackbill_Barnes_1980} to clean the divergence error}\label{sec:div_cleaning}
We apply a post-process to the magnetic field after each explicit solve in time. The obtained magnetic solution $\bB'$ by solving the linear system is corrected as
\[
  \bB=\bB'-\nabla \Psi_p,
\]
where $\Psi_p$ is obtained by solving the Poisson equation $\DIV(\GRAD\Psi_p) - \DIV\bB'=0$. Once $\Psi_p$ is obtained, the finite element representation of $\GRAD\Psi_p$ is determined by projection. Due to numerical error from this projection, the resulting magnetic field does not satisfy $\DIV\bB = 0$ exactly in both strong and weak senses. However, it is enough to keep the divergence error small so that the numerical runs remain stable. After each correction of the magnetic field, the dependent variables: pressure $p$, temperature $T$, energy $e$ and the entropy functions are updated accordingly to ensure consistency of the discrete solution. 
%See \citep{Toth_2000} for a more detailed accuracy analysis of the projection method.

\section{Nonlinear viscosity method for MHD}
\label{sec:nv}
We discretize the regularized ideal MHD system \eqref{eq:mhd1} using continuous finite element approximations. First, we define the following vector-valued function spaces:
\[
  \bcalX_h := [\calX_h]^{d}, \quad \bcalW_h := \calX_h \CROSS \bcalX_h \CROSS \calX_h \CROSS \bcalX_h, 
\]
where $\calX_h$ is defined in \eqref{eq:Xh}. Then, we formulate the finite element approximation of the MHD system \eqref{eq:mhd1} as follows: 
find $\bsfU_h(t) \in \calC^1([0,\widehat t \,\,]; \bcalW_h)$ such that
\[
  \begin{aligned}
    \Big(\p_t \bsfU_h \,,\, \bsfV_h \Big) 
    + & \Big( \DIV \bsfF_a(\bsfU_h) \,,\, \bsfV_h \Big)
        = 0,
        \quad \forall\,\, \bsfV_h \in \bcalW_h,
  \end{aligned}
\] 
where $\bsfU_h:=(\rho_h, \bbm_h, E_h, \bB_h)^\top$ of which components are finite element discretizations of the conserved variables $\rho, \bbm, E$, and $\bB$, respectively. Here, the inner product is computed as
\[
\Big(\bv \,,\, \bw\Big) := \sum_{K \subset \calT_h} \Big(\bv \,,\, \bw\Big)_K :=  \sum_{K \subset \calT_h} \int_K \bv \SCAL \bw \ud \bx.
\]

It is well known that finite element approximation of first order hyperbolic problems including the MHD system is unstable, see \eg \citep[Chapter~5]{Ern_Guermond_2004}. We aim to construct a stabilization technique to make the method stable. We regularize the MHD system by adding vanishing viscosity elliptic terms. 
The viscosity solution $\bsfU^{\e}_h(t) \in \calC^1([0,\widehat t \,\,]; \bcalW_h)$ is obtained by solving the following problem: find $\bsfU^{\e}_h(t)$ such that
\begin{equation}\label{eq:ODE}
  \begin{aligned}
    \Big(\p_t \bsfU^{\e}_h \,,\, \bsfV_h \Big)
    & + \Big( \DIV \bsfF_a(\bsfU^{\e}_h) \,,\, \bsfV_h \Big)
      + b(\bsfU^{\e}_h \,,\, \bsfV_h) = 0, 
      \quad \forall\,\, \bsfV_h \in \bcalW_h,
  \end{aligned}
\end{equation}
where the viscous bilinear form is defined as follows
\begin{equation}\label{eq:b}
  b(\bsfU^{\e}_h \,,\, \bsfV_h) := \sum_{K\subset \calT_h} 
  \int_K  \e_h \polJ_K \polJ_K^\top \, \GRAD \bsfU^{\e}_h \SCAL \GRAD \bsfV_h \ud \bx, 
  \quad \forall \bsfU_h, \bsfV_h \in \bcalW_h.
\end{equation}
Here \blue{$\e_h \equiv\e_h(\bsfU^{\e}_h ) \in \calC^1([0, \widehat t \,]; \calX_h)$} is a nodal vanishing viscosity function to be defined in the following sections. The product $\polJ_K \polJ_K^\top$ comes from reorienting the physical triangle to the equilateral triangle, which is convenient to prove positivity of the scheme, see \eg \citep[Remark 3.1]{Guermond_Nazarov_2014}. It is interesting to notice that the same artificial viscosity is used for all the components. Apart from the simplicity of having only one viscosity coefficient $\e_h$, several advantages of using it have been shown in \citep{Dao2022a}. The velocity field and pressure are computed at the nodal points, \eg $\bu_h(\bN_i,t) = \bbm_h(\bN_i,t)/\rho_h(\bN_i,t)$, $\forall i\in \calV$ and $t \in [0, \widehat t \,\,]$. 

One of the main difficulties in solving \eqref{eq:mhd1} is to preserve the solenoidal nature of the magnetic field, $\DIV \bB = 0$. However, it is not in the focus of this paper. To keep the divergence error of the discrete solution small, we use the simple projection method described in Section~\ref{sec:div_cleaning}.

\subsection{Construction of the stabilization term}
We want to construct $\e_h$ in \eqref{eq:b} to be minimum of first order viscosity that adds sufficient stabilization close to the shock areas without deteriorating the time-step restriction for the explicit schemes, and high order viscosity that vanishes in the smooth region. 

\blue{
Since the viscosity function will be computed at every time level, let us start by discretizing the temporal domain. Split $[0, \widehat t \,\,]$ into $N$ intervals of variable length. Let $t^n$ be the current time for $n=0,\ldots,N$ and the next time $t^{n+1}$ is computed using the time step $\tau_n$, \ie $t^{n+1} = t^n + \tau_n$. Let us denote by $\bsfU^n_h := \sum_{j\in \calV} \bsfU^n_j \varphi_j(\bx)$ the finite element approximation of the solution $\bsfU(\bx, t^n)$ at time $t^n>0$ with the nodal values $\bsfU_j^n$, and $\e_h^n := \sum_{j\in \calV} \e^n_j \varphi_j(\bx)$ the viscosity function defined at time $t^n$, where its nodal values $\e_j^n$ are computed in this section.
}

\subsubsection{First order viscosity}
We construct the first order viscosity on the local patches $S^{\polP_1, \textrm{fine}}_i$ for every node $N_i, i\in\calV$. We start by defining the local patch indicator 
\begin{equation}\label{eq:Phi}
  \Phi^{\polP_1, \textrm{fine}}_i := \max_{ i\not=j \in \calI(S^{\polP_1, \textrm{fine}}_i)}|\GRAD \varphi^{\polP_1, \textrm{fine}}_j|,
\end{equation}
\blue{
and local maximum wave speed for every $t^n$:
\begin{equation}\label{eq:vloc_max}
  \begin{aligned}
    \lambda_{\max,i}(\bsfU_h^n) :=&\, \max_{ i\not=j \in \calI(S^{\polP_1, \textrm{fine}}_i)}\Big(\lambda_1(\bsfU_h^n),\, \lambda_8(\bsfU_h^n) \Big)  \\
    = & 
        \max_{ i\not=j \in \calI(S^{\polP_1, \textrm{fine}}_i)}
        \Big(
        |\bu_h(\bN_j,t^n))\SCAL\be - c_{f}(\bN_j,t^n))| \,,\, |\bu_h(\bN_j,t^n))\SCAL\be + c_{f}(\bN_j,t^n))|
        \Big),
  \end{aligned}
\end{equation}
}

\begin{definition}[Nodal artificial viscosity]\label{def:nodal}
  Let $\Nel(S_i)$ be the number of elements in $S_i$. Then, for every node $\bN_i$, $i\in\calV$, the nodal artificial viscosity is 
  \begin{equation}\label{eq:eps}
  \blue{
    \e^{\text{L}}_{i}(\bsfU_h^{\,n}) := C_i m^{\polP_1, \textrm{fine}}_i \lambda_{\max, i}(\bsfU_h^{\,n}) \, \Phi^{\polP_1, \textrm{fine}}_i,
    }
  \end{equation}
where 
  \red{
  \begin{equation}\label{eq:Ci}
    C_i := \frac{d+1}{2}
    \frac{1}{\Nel(S_i)}
      \max_{K\in S_i}
      |K|
      ^{-1}.
  \end{equation}
  }

\end{definition}

When $\bcalW_h$ is the space of continuous piecewise linear polynomials, \ie $\polP_1$, the meshes $\calT_h$ and $\calT_h^{\polP_1, \textrm{fine}}$ are equal, and the lumped mass matrix $m^{\polP_1, \textrm{fine}}_i$ has support on the patch depicted in Figure~\ref{fig:meshes}(a). For higher-order polynomial spaces, the fine mesh $\calT_h^{\polP_1, \textrm{fine}}$ coincides with the nodal points of the space $\bcalW_h$, and therefore the support of $m^{\polP_1, \textrm{fine}}_i$ contains the closest nodes to the node $\bN_i$. The corresponding supports of the first order viscosity in \eqref{eq:eps} for the node $\bN_i$ are depicted in the dashed triangles in Figure~\ref{fig:meshes}(b) and Figure~\ref{fig:meshes}(c) for $\polP_2$ and $\polP_3$ spaces.

\begin{figure}[h!]
  \centering
  \begin{subfigure}{0.3\textwidth}
    \centering
    \includegraphics[width=\textwidth]{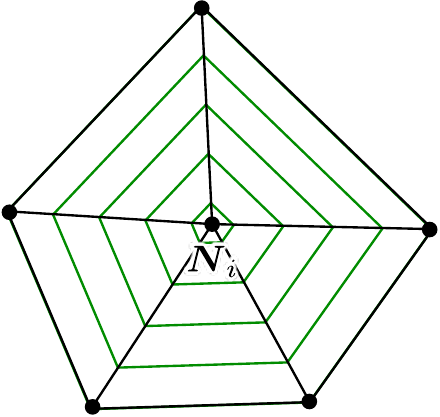}
    \caption*{$\polP_1$}
  \end{subfigure}
  \hfill
  \begin{subfigure}{0.3\textwidth}
    \centering
    \includegraphics[width=\textwidth]{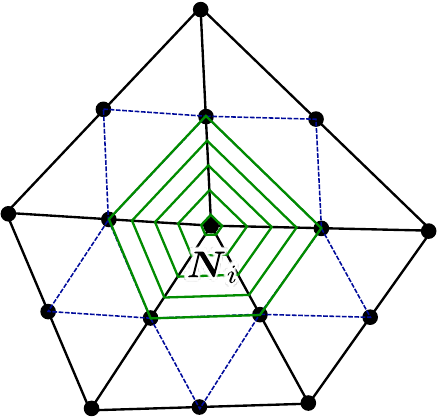}
    \caption*{$\polP_2$}
  \end{subfigure}
  \hfill
  \begin{subfigure}{0.3\textwidth}
    \centering
    \includegraphics[width=\textwidth]{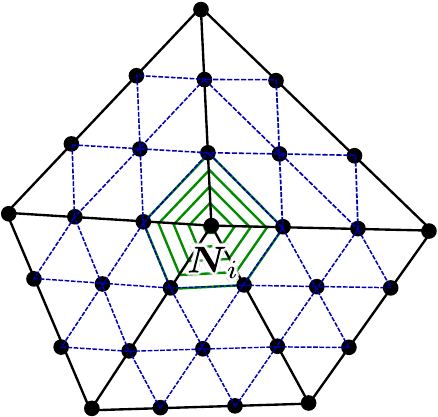}
    \caption*{$\polP_3$}
  \end{subfigure}
  \caption{Nodal distribution and sub-meshes for different polynomial spaces. The green dashed area is the support of $\bN_i$ in the corresponding $\polP_1$ sub-mesh.}
  \label{fig:meshes}
\end{figure}

It is important to keep in mind that, although we calculate the artificial viscosity on the nodes of the $\polP_1$ fine meshes, we compute the bilinear viscosity operator \eqref{eq:b} in the corresponding high-order polynomial basis functions. This step is crucial for maintaining conservation. 

\begin{remark}[Relation to upwind schemes]
\red{ For linear finite element spaces in 1D and uniform meshes, we have $C_i = \frac{1}{2h}$, $m_i = h$, $\Phi_i = \frac{1}{h}$ and $\polJ_K = h$. 
  Then, the viscosity coefficient is
  $
  \e^{\text{L}}_{i}(\bsfU_h^{\,n}) \polJ_K \polJ_K^\top = \frac{1}{2}  h \lambda_{\max, i}(\bsfU_h^{\,n}),
  $
  for every nodal point $N_i$, which, in turn, coincides with the Lax-Friedrichs viscous flux coefficient. For uniform meshes and linear finite elements in 2D and 3D, 
we have again $\Phi_i = \frac{1}{h}$, where $h$ is the mesh size of the element $K$. Then, by noting that $\int_K \varphi_i \ud \bx = \frac{1}{d+1} |K|$ which implies $m_i = \Nel(S_i)\frac{1}{d+1}|K|$, and the Jacobian in the mapping $\polJ_K$ scales like $h$, we get
  \[
  \e^{\text{L}}_{i}(\bsfU_h^{\,n}) = \frac{d+1}{2} \frac{1}{\Nel(S_i)} \frac{1}{|K|}
  \, \Nel(S_i)\frac{1}{d+1}|K| \lambda_{\max, i}(\bsfU_h^{\,n}) \frac{1}{h}
  = \frac12\lambda_{\max, i}(\bsfU_h^{\,n}) \frac{1}{h}.
  \]
  And again, 
    $
  \e^{\text{L}}_{i}(\bsfU_h^{\,n}) \polJ_K \polJ_K^\top \sim \frac{1}{2}  h \lambda_{\max, i}(\bsfU_h^{\,n}),
  $
  which has a correct unit of the Lax-Friedrichs viscous flux coefficient. Note, that in 2D and 3D the artificial viscosity is tensor-valued. 

}
\end{remark}

\blue{

}

\subsubsection{Residual viscosity}
\label{sec:rv}
\blue{
For given solution $\bsfU_h^n$ at time $t^n$ we compute the finite element residual of the MHD system by solving the following projection problem: Find $\bsfR(\bsfU_h^n) \in \bcalW_h$ such that  
\[
  \Big(\bsfR(\bsfU^n_h)\, ,\, \bsfV_h \Big) + \sum_{K \subset \calT_h} \Big(\frac{|K|^{\frac{2}{d}}}{p}\, \GRAD \bsfR(\bsfU^n_h) \,,\, \GRAD \bsfV_h \Big)_K  := \Big(\big| D_{\tau} \bsfU^n_h + \DIV \bsfF_a(\bsfU^n_h) \big| \, ,\, \bsfV_h \Big), \quad \forall \bsfV_h \in \bcalW_h,
\]
where $D_{\tau} \bsfU^n_h$ is an approximation of the time derivative of the solution $\bsfU_h$ at time $t^n$, which can be obtained using the backward differentiation formula (BDF). In practice, any approximation above $\calO(\tau^{p-1})$ does not deteriorate the convergence rate of the Galerkin method. In all computational tests in this paper, we use the second-order BDF method to approximate the time derivative of the solution. The residual is usually highly oscillatory. The second term in the above projection is an elliptical smoothing, which smooths out small fluctuations in the residual without suppressing its jump at discontinuities.
 
Let us denote the components of the residual as 
$
\bsfR(\bsfU_h)_\rho,
\bsfR(\bsfU_h)_{\bbm},
\bsfR(\bsfU_h)_E,
\bsfR(\bsfU_h)_{\bB}
$
. We construct the residual-based artificial viscosity term $\e^\textrm{RV}_{h}(\bsfU_h^n) \in \calX_h$ with the nodal points $\bN_i$, $i\in\calV$, at time $t^n$ as 
\begin{equation}\label{eq:eps_rv}
  \label{eq:rv:mhd}
  \begin{aligned}
    \e^\textrm{RV}_{h,i}(\bsfU_h^n) = C_i 
    \min
    \Big(
    \lambda_{\max,i}(\bsfU_h^n)\, \Phi^{\polP_1, \textrm{fine}}_i \,,\, 
    \max_{\mathsf{x} = \{\rho, \red{E}, \bbm, \bB\}}  \frac{|\bsfR (\bsfU_{h,i}^n)_{\mathsf{x}}|}{\Psi_i(\mathsf{x}^n_h)}
    \Big) \, m^{\polP_1, \textrm{fine}}_i, 
  \end{aligned}
\end{equation}
where $\Phi^{\polP_1, \textrm{fine}}_i$ is defined in \eqref{eq:Phi} and $C_i$ is the artificial viscosity constant defined in \eqref{eq:Ci}, the maximum wave speed is defined in \eqref{eq:vloc_max}, and $\Psi_i(\mathsf{x})$ is a piecewise constant normalization function calculated as follows:
\begin{equation}\label{eq:psi_i}
	\Psi_i(\mathsf{x}) := 
	\frac14
	\|\mathsf{x} - \overline{\mathsf{x}}\, \|_{L^{\infty}(\Omega)}
	 \Bigg(
	 1 - 
	\frac{
	\max_{j\in\calI(S_i)}\mathsf{x}_j(t) - 
	\min_{j\in\calI(S_i)}\mathsf{x}_j(t) 
	}
	{
	\max_{j\in\calV}\mathsf{x}_j(t) - 
	\min_{j\in\calV}\mathsf{x}_j(t) 
	}
	\Bigg) + \epsilon \|\mathsf{x}\|_{L^{\infty}(\Omega)} ,
\end{equation}
where 
$
\overline{\mathsf{x}} := \frac{1}{|\Omega|} \int_{\Omega} \mathsf{x} \ud \bx
$
and $\epsilon = 10^{-8}$ is a small safety factor that is used to avoid division by zero, and 
\[
\theta_i:=\frac{
	\max_{j\in\calI(S_i)}\mathsf{x}_j(t) - 
	\min_{j\in\calI(S_i)}\mathsf{x}_j(t) 
	}
	{
	\max_{j\in\calV}\mathsf{x}_j(t) - 
	\min_{j\in\calV}\mathsf{x}_j(t) 
	}
\]
is a smoothness indicator, $\theta_i\in[0,1]$, that reduces the normalization in the region of shocks and discontinuities, which leads to an increase in the resulting scaled residual.  

\begin{remark}[Scaling of the residual]
Let us consider a linear advection equation in 1D with the constant speed $\beta>0$: $\p_t \rho + \beta \p_x \rho=0$. Then, the residual of an approximate solution $\rho_h^n$ at time $t^n$ and uniform grid with the mesh-size $h$ at the nodal point $x_i$ is:
\[
R_i(\rho_h^n):= (D_\tau \rho_h^n)_i + \beta \frac{\rho^n_{i+1} - \rho^n_{i-1}}{2h}.
\]
In addition, $\lambda_{\max,i}\equiv \beta$ and $\Phi_i\equiv \frac{1}{h}$. In addition, to ease the discussion let us take $\epsilon=0$ and assume that the time-derivative of $\rho_h^n$ is zero. Then,
\[
\begin{aligned}
\frac{|R_i(\rho_h^n)|}{\Psi_i(\rho_h^n)}
= &\
\frac{1}{
\frac{1}{4}\|\rho_h^n - \overline{\rho_h^n}\|_{L^{\infty}(\Omega)}(1-\theta_i)}
 \frac{\beta}{2h} | \rho^n_{i+1} - \rho^n_{i-1} | \\
= &\
\frac{4}{(1-\theta_i) \max_{i\in \calV}|\rho_i^n - \overline{\rho_h^n}|} 
\frac{\beta}{2h} \Big| (\rho^n_{i+1}-\overline{\rho_h^n}) - (\rho^n_{i-1}-\overline{\rho_h^n}) \Big| \\
\le &\
\frac{4}{(1-\theta_i) \max_{i\in \calV}|\rho_h^n - \overline{\rho_h^n}|} 
\frac{\beta}{2h} \, 2\,  \max_{i\in \calV}| |\rho_i^n - \overline{\rho_h^n}| \\
= &\
\frac{4}{1-\theta_i}\, \frac{\beta}{h} \\
= &\
\frac{4}{1-\theta_i}\, \lambda_{\max,i} \Phi_i,
\end{aligned}
\]
which shows that 1. the scaling of the residual is the same as $ \lambda_{\max,i} \Phi_i$; 2. when solution has discontinuity around $x_i$, the smoothness indicator  $\theta_i$ tends to 1 and scaled residual goes to infinity; 3. when solution is smooth, $\theta_i$ approaches to 0 and the scaled residual is bounded from above with $4 \lambda_{\max,i} \Phi_i$. When the solution is smooth, the residual has a convergence rate of the scheme which is $\calO(h^2)$ in this case.  

\end{remark}
}

\subsection{Time-stepping}
Once the MHD system is discretized in space using the continuous finite element method we obtain the system of ODEs \eqref{eq:ODE}. Let us denote this system as
\[
  \polM\, \p_t \bsfU_h(t) = \calF(\bsfU_h(t), \e^{\textrm{RV}}_h(\bsfU_h(t))),
\]
where $\polM \in \mR^{(2d+2)I\times (2d+2)I}$ is a \blue{consistent} mass matrix, $\calF(\bsfU_h(t), \e^{\textrm{RV}}_h(t))$ is the right-hand-side function of the system which depends on the solution $\bsfU_h(t)$ and viscosity $\e^{\textrm{RV}}_h(\bsfU_h(t))$. 

\blue{
Recall that $\bsfU^n_h$ is the finite element approximation of the solution $\bsfU(t^n, \bx)$ at time $t^n>0$ with the nodal values $\bsfU_j^n$. 
}
Next, we discretize this system in time using explicit $r$-stage Runge-Kutta methods:
\begin{equation}\label{eq:ho:mhd}
  \bsfU^{n+1}_h = \bsfU^{n}_h + \tau_n (b_1 \bsfK_1 + \ldots + b_r \bsfK_r),
\end{equation}
where $b_i$, $i=1,\ldots,r$ are coefficients obtained from the Butcher tableau, and the stage variables $\bsfK_i$ are computed as follows: for the given solution $\bsfU_h^n$ and the viscosity $\e_h^{\textrm{RV},n}:=\e^{\textrm{RV}}_h(\bsfU_h^n)$ at time level $t^n$, set $\bsfW_0:= \bsfU_h^n$, then let $\bsfW_l$ be the solution at the $l$-th stage of the Runge-Kutta method, then compute $\bsfK_l$ by solving the following system:
\[
  \polM\, \bsfK_l = \bsfF(\bsfW_l,\, \e^{\textrm{RV},n}_h),
\]
for all $l=0,\ldots, r$. Note that the viscosity coefficients can also be computed on the fly at every Runge-Kutta stage. However, in this work, the viscosity coefficient is constructed from the previous time level and does not change within the stages.

The time-step $\tau_n$ is computed using the following CFL condition:
\begin{equation}\label{eq:cfl}
  \tau_n = \text{CFL } \max_{i\in \calV} \lambda_{max,i}(\bsfU^n_h) \, \Phi_i^{\polP_1, \text{fine}}.
\end{equation}

\subsection{Boundary conditions}\label{sec:boundary_conditions:mhd}
In the following numerical examples, we use the following boundary conditions: Dirichlet, periodic, and slip or impermeability boundary conditions. The periodic boundary condition is implemented by mapping node points on two opposite boundaries of rectangular regions.

The slip boundary condition is implemented by replacing
$
\bbm_j^{n+1}
$
by
$
\bbm_j^{n+1} - (\bbm_j^{n+1} \SCAL \bn_j) \bn_j
$
 for any nodes $\bN_j$ on the sliding boundary, where $\bn_j$ is the outward pointing unit normal vector to the boundary. Note that the slip and Dirichlet conditions are strictly imposed as a correction step after each Runge-Kutta solution. More precisely, to calculate the solution $\bsfU_h^{n+1}$ at time $t^{n+1}$, the slip and Dirichlet boundary conditions are imposed by specifying the values of the solution at the boundary nodes $\bN_j $ according to its boundary data.

\subsection{Summary and the algorithm for solving MHD}
We conclude this section by presenting a concise summary of the time-stepping approach used to solve the MHD equations, as outlined in the following algorithm:
\begin{algorithm}[H]
  \renewcommand{\algorithmicrequire}{\textbf{Input:}}
  \renewcommand{\algorithmicensure}{\textbf{Output:}}
  \caption{High-order stabilization for the MHD equations}
  \label{alg:rv:mhd}
  \begin{algorithmic}[1]
    \Require $\bsfU_h^{0}$, $\calT_h$
    \Ensure $\bsfU_h^N$
    \State Construct the space $\polP_k$ on $\calT_h$
    \State Construct a fine mesh $\calT_h^{\textrm{fine}}$ with vertices on the nodes of $\polP_k$
    and construct the space $\polP_1$ on $\calT_h^{\textrm{fine}}$
    \State Compute $\Phi^{\polP_1, \textrm{fine}}_i$ and $m^{\polP_1, \textrm{fine}}_i$ for every $i\in \calV$
    \For {$n\gets 0$ to $N$} 
    \State Compute the residual $\bsfR(\bsfU^n)$ for every $i\in \calV$
    \State Compute $\e^{\textrm{RV},n}_{i}$ from \eqref{eq:rv:mhd} for every $i\in \calV$
    \State Solve \eqref{eq:ho:mhd} to get $\bsfU_h^{n+1}$
    \State Clean the divergence of $\bB_h$ as in Section~\ref{sec:div_cleaning} and update dependent variables
    \State Apply the Dirichlet or slip boundary conditions strongly as in Section~\ref{sec:boundary_conditions:mhd}
    \State Determine $\lambda_{\max,i}$ using \eqref{eq:vloc_max} for every $i\in\calV$
    \State Determine the next time step size $\tau_n$ using \eqref{eq:cfl}    \EndFor
  \end{algorithmic}
\end{algorithm}

\section{Numerical examples}
\label{sec:num}
In this section, we demonstrate the efficiency of our proposed stabilization method by solving several well-known benchmark problems. We start by solving a smooth problem, where the initial data and final time are chosen such that the exact solution is readily available. Our primary objective is to investigate the preservation of high-order accuracy across various polynomial degrees. Then, we continue our discussion by solving more challenging benchmarks related to ideal MHD, covering scenarios with strong shocks and discontinuities.

All error norms presented below are relative norms. The time-stepping in this section is performed using the classical fourth order Runge-Kutta method.

\subsection{Accuracy test: Smooth vortex problem from \cite{Wu_et_al_2018}}
We use this benchmark to confirm high-order accuracy of the proposed method. For this problem, we define the computational domain as a square $\Omega = [-10, 10] \times [-10, 10]$. The reference solution is a stationary flow with a vortex perturbation $(\rho(t), \bu(t), p(t), \bB(t))$ $=(\rho_0, \bu_0+\delta\bu, p_0+\delta p, \bB_0+\delta\bB)$, where $\rho_0 = 1$,  $\bu_0 = (1,1)$, $\delta u = \frac{\mu}{\pi\sqrt 2}e^{(1-r^2)/2}(-r_2,r_1)$, $p_0=0$, $\delta p = -\frac{\mu^2(1+r^2)e^{1-r^2}}{8\pi^2}$, $\bB_0=(0.1,0.1)$, $\delta B = \frac{\mu e^{(1-r^2)/2}}{2\pi}(-r_2,r_1)$, the vortex radius $r=\sqrt{r_1^2+r_2^2}$, $(r_1,r_2)=(x,y)-\bu_0t$, and the vortex strength $\mu=1.0$. The adiabatic constant is  $\gamma=\frac{5}{3}$.

We proceed to solve the problem using three different polynomial spaces: $\polP_1$, $\polP_2$, and $\polP_3$ on \blue{unstructured} meshes and present the solution's convergence in Table~\ref{table:convergence_vortex_P1}. The CFL number for all polynomial spaces is 0.1. The errors measured at final time $t=0.05$. Our analysis reveals optimal convergence rates for $\polP_1$ and $\polP_3$, corresponding to the second and fourth orders in L$^1$-norm, respectively. However, it is worth noting that the convergence rate for the second-order polynomial is suboptimal, a characteristic often associated with even-order finite elements, as discussed in references such as \citep{Nazarov_Larcher_2017}. 

\blue{The rates of convergence in L$^2$ and L$^{\infty}$ norms are somewhat suboptimal for $\polP_3$. This may be due to the use of unstructured meshes and floating error in the numerical integration of the L$^2$-norm. To calculate the L$^2$-norm, we interpolate the error into the space $\polP_5$, so that the resulting integrand becomes a function from the space $\polP_{10}$ which requires many quadrature points for exact integration.}

\begin{table}[h!]
  \centering
  \caption{Smooth vortex problem. Errors in $\bu_h$ and $\bB_h$ at time $t=0.05$}
  \label{table:convergence_vortex_P1}

  \vspace{0.1in}
  \centering{$\polP_1$ solution}

  \begin{adjustbox}{max width=\textwidth}
    \begin{tabular}{|c||c|c|c|c|c|c||c|c|c|c|c|c|}
      \hline
      \multirow{2}{*}{\#DOFs} &  \multicolumn{6}{c||}{$\bu_h$} &  \multicolumn{6}{c|}{$\bB_h$}  \\ \cline{2-13}
     {}   &       L$^1$   &    Rate   &       L$^2$   &    Rate & L$^{\infty}$   &    Rate &       L$^1$   &    Rate   &       L$^2$   &    Rate & L$^{\infty}$   &    Rate   \\ \hline
       8428 &  8.40E-05 &     -- &   4.72E-04 &     -- &   5.84E-03 &     -- & 1.81E-02 &     -- &   2.03E-02 &     -- &   3.21E-02 &     -- \\ 
      18770 &  3.63E-05 &   2.10 &   2.03E-04 &   2.10 &   2.56E-03 &   2.06 & 7.80E-03 &   2.12 &   8.75E-03 &   2.12 &   1.41E-02 &   2.08 \\ 
      42182 &  1.58E-05 &   2.05 &   8.88E-05 &   2.05 &   1.13E-03 &   2.03 & 3.40E-03 &   2.07 &   3.82E-03 &   2.06 &   6.16E-03 &   2.06 \\ 
      95374 &  6.91E-06 &   2.03 &   3.88E-05 &   2.03 &   4.91E-04 &   2.04 & 1.49E-03 &   2.04 &   1.67E-03 &   2.04 &   2.68E-03 &   2.05 \\
     213672 &  3.06E-06 &   2.02 &   1.72E-05 &   2.02 &   2.17E-04 &   2.02 & 6.57E-04 &   2.03 &   7.38E-04 &   2.03 &   1.19E-03 &   2.03 \\
     480976 &  1.35E-06 &   2.01 &   7.58E-06 &   2.01 &   9.60E-05 &   2.01 & 2.90E-04 &   2.02 &   3.26E-04 &   2.02 &   5.23E-04 &   2.02 \\ \hline
    \end{tabular}
  \end{adjustbox}
  % \end{table}
  \\
  % \begin{table}[h!]
  \centering
  \vspace{0.1in}
  $\polP_2$ solution
  \begin{adjustbox}{max width=\textwidth}
    \begin{tabular}{|c||c|c|c|c|c|c||c|c|c|c|c|c|}
      \hline
      \multirow{2}{*}{\#DOFs} &  \multicolumn{6}{c||}{$\bu_h$} &  \multicolumn{6}{c|}{$\bB_h$}  \\ \cline{2-13}
     {}   &       L$^1$   &    Rate   &       L$^2$   &    Rate & L$^{\infty}$   &    Rate &       L$^1$   &    Rate   &       L$^2$   &    Rate & L$^{\infty}$   &    Rate   \\ \hline
       8520 &  2.13E-05 &     -- &   1.23E-04 &     -- &   2.28E-03 &     -- & 5.18E-03 &     -- &   5.34E-03 &     -- &   1.17E-02 &     -- \\ 
      18872 &  7.26E-06 &   2.71 &   4.09E-05 &   2.76 &   7.16E-04 &   2.91 & 1.80E-03 &   2.69 &   1.81E-03 &   2.75 &   3.68E-03 &   2.94 \\ 
      43240 &  2.76E-06 &   2.33 &   1.49E-05 &   2.44 &   3.36E-04 &   1.82 & 6.70E-04 &   2.40 &   6.63E-04 &   2.45 &   1.64E-03 &   1.96 \\ 
      96568 &  1.16E-06 &   2.16 &   5.99E-06 &   2.26 &   1.62E-04 &   1.82 & 2.63E-04 &   2.34 &   2.62E-04 &   2.33 &   7.62E-04 &   1.92 \\
     214400 &  5.09E-07 &   2.06 &   2.57E-06 &   2.12 &   7.59E-05 &   1.90 & 1.12E-04 &   2.16 &   1.11E-04 &   2.17 &   3.47E-04 &   1.98 \\
     482032 &  2.24E-07 &   2.02 &   1.12E-06 &   2.06 &   3.42E-05 &   1.97 & 4.85E-05 &   2.06 &   4.78E-05 &   2.08 &   1.56E-04 &   1.98 \\ \hline
    \end{tabular}
  \end{adjustbox}
  %%%%
  \vspace{0.1in}
  $\polP_3$ solution
  \\
  \begin{adjustbox}{max width=\textwidth}
    \begin{tabular}{|c||c|c|c|c|c|c||c|c|c|c|c|c|}
      \hline
      \multirow{2}{*}{\#DOFs} &  \multicolumn{6}{c||}{$\bu_h$} &  \multicolumn{6}{c|}{$\bB_h$}  \\ \cline{2-13}
     {}   &       L$^1$   &    Rate   &       L$^2$   &    Rate & L$^{\infty}$   &    Rate &       L$^1$   &    Rate   &       L$^2$   &    Rate & L$^{\infty}$   &    Rate   \\ \hline
      8460  &  1.33E-05 &     -- &   1.16E-04 &     -- &   3.70E-03 &     -- & 2.92E-03 &     -- &   4.96E-03 &     -- &   2.05E-02 &     -- \\ 
      19170 &  1.73E-06 &   4.99 &   8.31E-06 &   6.44 &   1.85E-04 &   7.32 & 3.73E-04 &   5.09 &   3.52E-04 &   6.54 &   8.21E-04 &   7.96 \\ 
      42462 &  3.32E-07 &   4.14 &   1.69E-06 &   4.00 &   4.15E-05 &   3.76 & 7.17E-05 &   4.18 &   7.15E-05 &   4.04 &   2.00E-04 &   3.58 \\ 
      97290 &  6.06E-08 &   4.11 &   3.44E-07 &   3.85 &   8.80E-06 &   3.74 & 1.30E-05 &   4.13 &   1.45E-05 &   3.87 &   4.79E-05 &   3.47 \\
     217278 &  1.12E-08 &   4.20 &   7.60E-08 &   3.76 &   2.18E-06 &   3.47 & 2.38E-06 &   4.25 &   3.19E-06 &   3.78 &   1.11E-05 &   3.66 \\
     482400 &  2.21E-09 &   4.07 &   1.89E-08 &   3.49 &   5.84E-07 &   3.31 & 4.67E-07 &   4.09 &   7.93E-07 &   3.50 &   2.76E-06 &   3.48 \\ \hline
    \end{tabular}
  \end{adjustbox}
\end{table}

%%%%%%%%%%%%%%%%%%%%%%%%%%%%%%%%%%%%%%%%%%%% 
\subsection{Brio-Wu MHD shock tube problem \citep{Brio_Wu_1988}}
This benchmark is a popular one dimensional Riemann problem for ideal MHD. This test is a typical way to verify if a numerical method can resolve different nonlinear waves of the MHD system. The domain is $\Omega = [0, 1]$. The adiabatic constant is $\gamma=2$. At the initial state, the left profile is given by
\[
  (\rho, u, p, B_x) = (1,0,1,0.75), B_y = 1.
\]
The right profile is given by
\[
  (\rho, u, p, B_x) = (0.125,0,0.1,0.75), B_y = -1.
\]
These profiles imply the Dirichlet boundary conditions at $x=0$ and $x=1$ for all time $t > 0$, which are imposed strongly in every time step. The CFL number for all polynomial degrees used in this problem is 0.3.
Note that, even though the problem setting is in 1D, we solve it in 2D by setting the $y$ coordinate to change in $\in[0,h]$, here $h$ is the size of the interval in the $x$-direction. We then impose a periodic boundary condition on the $y$-direction. We run the simulations until the final time $t = 0.1$.

\red{
First, we numerically investigate the first-order viscosity constructed in Definition~\ref{def:nodal}. The first-order solutions using $\polP_1$ and $\polP_3$ elements under different resolutions are plotted in Figure~\ref{fig:briowu_firstorder}(a) and the corresponding amounts of artificial viscosity are plotted in Figure~\ref{fig:briowu_firstorder}(b). Under the same number of computational nodes, the first-order $\polP_1$ and $\polP_3$ solutions and added amounts of viscosity are nearly identical. Similar convergence behaviors between $\polP_1$ and $\polP_3$ when first-order viscosity is used can be seen in Table~\ref{table:convergence_briowu_firstorder}. Results using high-order residual-based viscosity are shown in Figure~\ref{fig:briowu}. Again, we can see very similar $\polP_1$ and $\polP_3$ solutions under the same number of nodes. The added amounts of artificial viscosity are also close in terms of magnitude and locality. This observation shows that the viscosity construction delivers excellent localization for the higher-order polynomial $\polP_3$: the viscosity is not spreaded out albeit wider in stencils. The upper bound $\varepsilon^L_h$ is hit at the first time steps due to the sharp jump. In Figure~\ref{fig:briowu_firstorder} and \ref{fig:briowu}, the solutions are compared with a fine reference solution given by the Athena code \citep{Stone_2008}. From the plots, one can see that the compound waves are finely resolved. Our solutions agree well with the reference solution. The $\polP_3$ solutions admit more oscillatory effects which is expected due to the high-order nature. Convergence rates can be seen in Table~\ref{table:convergence_briowu}. The $\polP_3$ solutions render higher errors for this test case because of the vanishing oscillations visible in the zoomed-in plots. We will see that this is not always the case. In the following benchmarks, $\polP_3$ outperforms $\polP_1$. These oscillations on the $\polP_3$ solution could be reduced by projecting the initial condition to the finite element space by some elliptic smoothing. In this work, the usual finite element interpolation is used to interpolate the initial condition into $\polP_3$.
}

\begin{figure}[!ht]
  \centering
  \begin{subfigure}{0.49\textwidth}
    \centering
    \includegraphics[width=\textwidth]{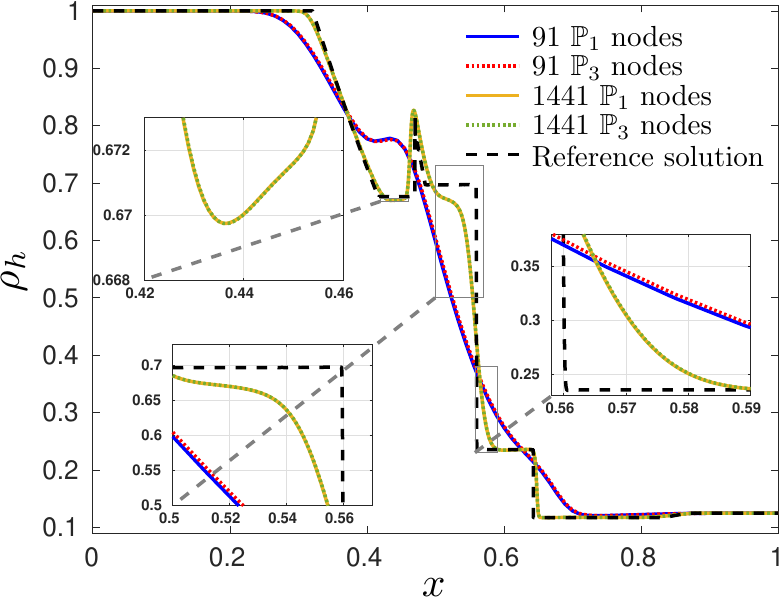}
    \caption*{$\polP_1$, $\polP_3$ density solutions}
  \end{subfigure}
  \hfill
  \begin{subfigure}{0.49\textwidth}
    \centering
    \includegraphics[width=\textwidth]{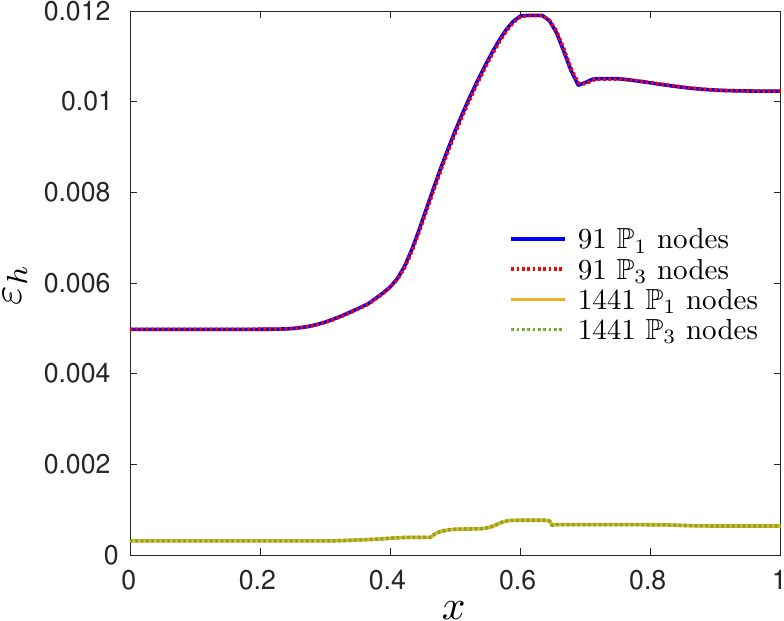}
    \caption*{$\polP_1$, $\polP_3$ first-order viscosity}
  \end{subfigure}
  \caption{First-order solutions to the Brio-Wu problem and the corresponding amounts of artificial viscosity. The reference solution is produced by the Athena code \citep{Stone_2008} using 10001 grid points.}
  \label{fig:briowu_firstorder}
\end{figure}

\begin{table}[h!]
  \centering
  \caption{Errors in density solution $\rho_h$ in the Brio-Wu problem using first-order viscosity. The errors are measured at time $t=0.1$}
  \label{table:convergence_briowu_firstorder}

  \begin{adjustbox}{max width=\textwidth}
    \begin{tabular}{|c||c|c|c|c||c|c|c|c|}
      \hline
      \multirow{2}{*}{\#nodes} &  \multicolumn{4}{c||}{$\polP_1$ solution} &  \multicolumn{4}{c|}{$\polP_3$ solution}  \\ \cline{2-9}
     {}   &       L$^1$   &    Rate   &       L$^2$   &    Rate  &       L$^1$   &    Rate   &       L$^2$ & Rate  \\ \hline
       91 &  6.01E-02 &     -- &   1.01E-01 &     --  & 5.96E-02 &     -- &   9.59E-02 &     -- \\ 
      181 &  4.33E-02 &   0.48 &   8.27E-02 &   0.29  & 4.13E-02 &   0.53 &   7.88E-02 &   0.29 \\ 
      361 &  3.22E-02 &   0.43 &   6.90E-02 &   0.26  & 3.14E-02 &   0.40 &   6.64E-02 &   0.25 \\ 
      721 &  2.30E-02 &   0.49 &   5.55E-02 &   0.32  & 2.33E-02 &   0.43 &   5.61E-02 &   0.25 \\
     1441 &  1.64E-02 &   0.49 &   4.49E-02 &   0.31  & 1.65E-02 &   0.49 &   4.51E-02 &   0.32 \\ \hline
    \end{tabular}
  \end{adjustbox}
\end{table}

\begin{figure}[!ht]
  \centering
  \begin{subfigure}{0.49\textwidth}
    \centering
    \includegraphics[width=\textwidth]{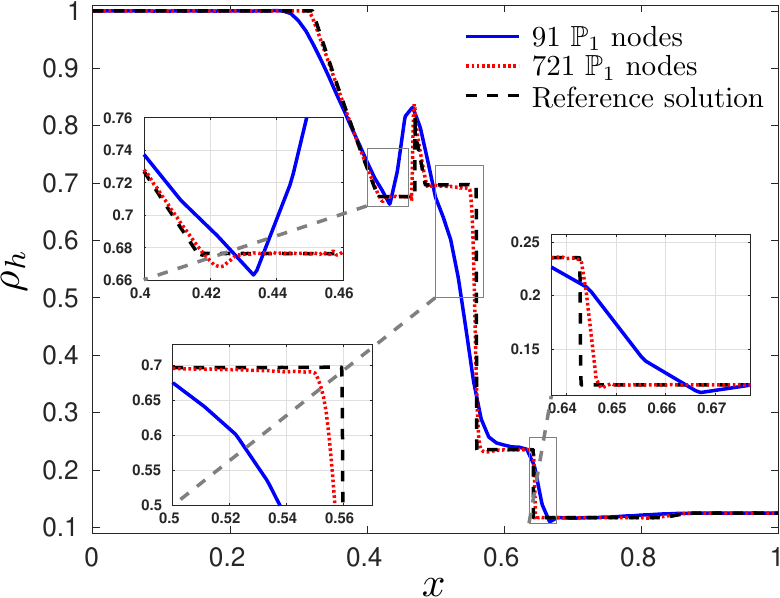}
    \caption*{$\polP_1$ solutions}
  \end{subfigure}
  \hfill
  \begin{subfigure}{0.49\textwidth}
    \centering
    \includegraphics[width=\textwidth]{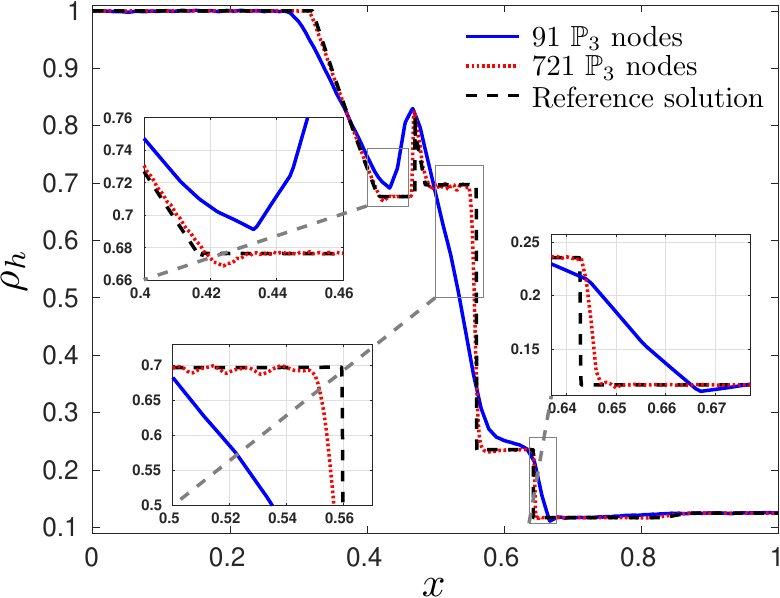}
    \caption*{$\polP_3$ solutions}
  \end{subfigure}
  \caption{Convergence of the numerical solutions to the Brio-Wu problem. The high-order solution. The reference solution is produced by the Athena code \citep{Stone_2008} using 10001 grid points.}
  \label{fig:briowu}
\end{figure}

\begin{figure}[!ht]
  \centering
  \begin{subfigure}{0.49\textwidth}
    \centering
    \includegraphics[width=\textwidth]{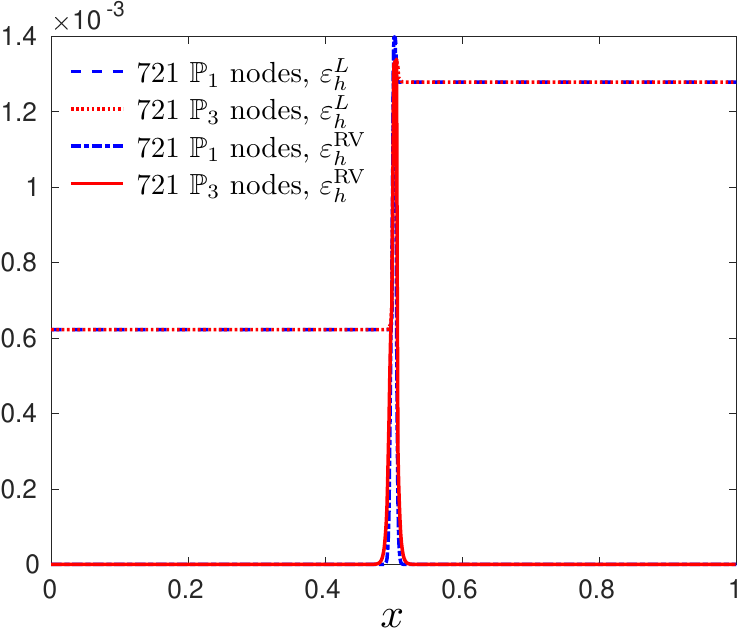}
    \caption*{Artificial viscosity at the first time step}
  \end{subfigure}
  \hfill
  \begin{subfigure}{0.49\textwidth}
    \centering
    \includegraphics[width=\textwidth]{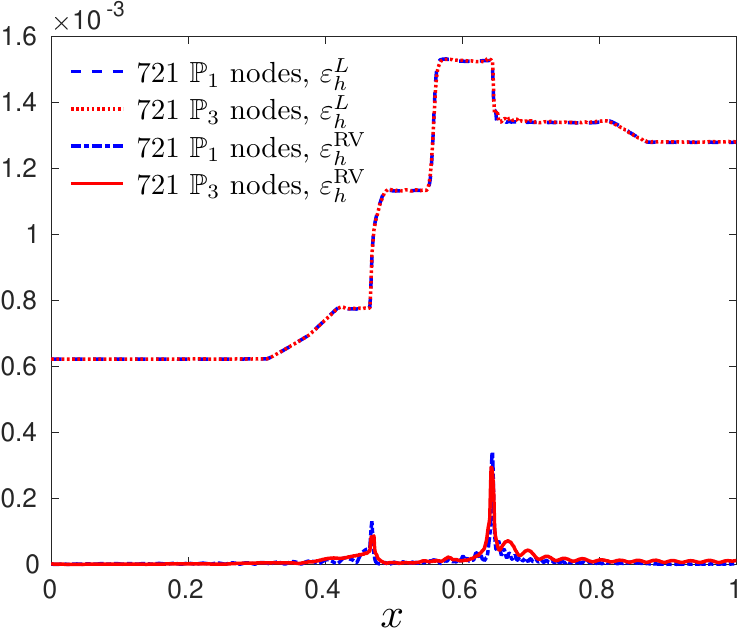}
    \caption{Artificial viscosity at the final time $t=0.1$}
  \end{subfigure}
  \caption{The amount of artificial viscosity added to the numerical solutions in the Brio-Wu problem. The high-order viscosity $\varepsilon^{\text{RV}}_h$ is plotted together with its first-order upper bound $\varepsilon^L_h$.}
  \label{fig:briowu_mu}
\end{figure}

\begin{table}[h!]
  \centering
  \caption{Errors in density solution $\rho_h$ in the Brio-Wu problem using high-order residual viscosity. The errors are measured at time $t=0.1$}
  \label{table:convergence_briowu}

  \begin{adjustbox}{max width=\textwidth}
    \begin{tabular}{|c||c|c|c|c||c|c|c|c|}
      \hline
      \multirow{2}{*}{\#nodes} &  \multicolumn{4}{c||}{$\polP_1$ solution} &  \multicolumn{4}{c|}{$\polP_3$ solution}  \\ \cline{2-9}
     {}   &       L$^1$   &    Rate   &       L$^2$   &    Rate  &       L$^1$   &    Rate   &       L$^2$ & Rate  \\ \hline
       91 &  4.10E-02 &     -- &   7.36E-02 &     --  & 4.26E-02 &     -- &   7.13E-02 &     -- \\ 
      181 &  2.16E-02 &   0.93 &   4.96E-02 &   0.57  & 2.09E-02 &   1.03 &   4.71E-02 &   0.60 \\ 
      361 &  1.17E-02 &   0.89 &   3.58E-02 &   0.47  & 1.13E-02 &   0.89 &   3.27E-02 &   0.53 \\ 
      721 &  5.79E-03 &   1.02 &   2.37E-02 &   0.60  & 6.59E-03 &   0.78 &   2.56E-02 &   0.35 \\
     1441 &  2.98E-03 &   0.96 &   1.60E-02 &   0.56  & 3.58E-03 &   0.88 &   1.74E-02 &   0.56 \\ \hline
    \end{tabular}
  \end{adjustbox}
\end{table}

%%%%%%%%%%%%%%%%%%%%%%%%%%%%%%%%%%%%%%%%%%%% 
\subsection{Orszag-Tang problem \citep{Orszag_Tang_1979}}
The considered domain is the unit square, $\Omega = [0,1] \times [0,1]$. This example is a highly-recognizable benchmark for ideal MHD. The solution is initialized as a smooth profile,
\[
  (\rho_0, \bu_0, p_0, \bB_0)= \left(\frac{25}{36\pi},(-\sin(2\pi y),\sin(2\pi x)),\frac{5}{12\pi},\left(-\frac{\sin(2\pi y)}{\sqrt{4\pi}},\frac{\sin(4\pi x)}{\sqrt{4\pi}}\right)\right).
\]
The adiabatic constant is $\gamma=\frac{5}{3}$. Periodic boundary conditions are used in all boundaries. 

The initially smooth data in this simulation evolves, and due to the nonlinearity of the MHD system, strong shocks and discontinuities develop. These sharp layers contribute to an increase in divergence error. Previous reports in the literature, such as \citep{Li_Shu_2005}, indicate that schemes with large divergence errors struggle to solve this benchmark for longer than $t=0.5$. In this paper, the simple divergence cleaning technique presented in Section~\ref{sec:div_cleaning} performs exceptionally well for all tested polynomial degrees. Although our code was capable of running for $t>1$ and exhibited turbulence behavior in the solution, we limited our presentation to shorter times to align with existing references.

The problem is solved using $200\times200$ nodal points with $\polP_1$ and $\polP_3$ polynomial spaces. The solutions at times $t=0.5$ and $t=1$, along with the corresponding artificial viscosity coefficients, are depicted in Figure~\ref{fig:OT_t05} and \ref{fig:OT_t1}, respectively. The viscosity effectively tracks the shock, adding a small amount in the smooth region. The $\polP_3$ residual exhibits more noise, inherent to high-order polynomials. Although the $\polP_1$ and $\polP_3$ solutions are nearly identical at time $t=0.5$, a significant difference emerges at the later time $t=1.0$. At this point, the $\polP_3$ solution captures more accurate structures, including the elliptical upsurge in the middle of the domain. Further discussion on the solution behavior at different times can be found in \citep{Dao2022b}. The CFL number is 0.3 for both polynomial spaces employed in these simulations.

%%%%% 200x200
\begin{figure}[h!]
  \centering
  \begin{subfigure}{0.4\textwidth}
    \centering
    \includegraphics[width=\textwidth]{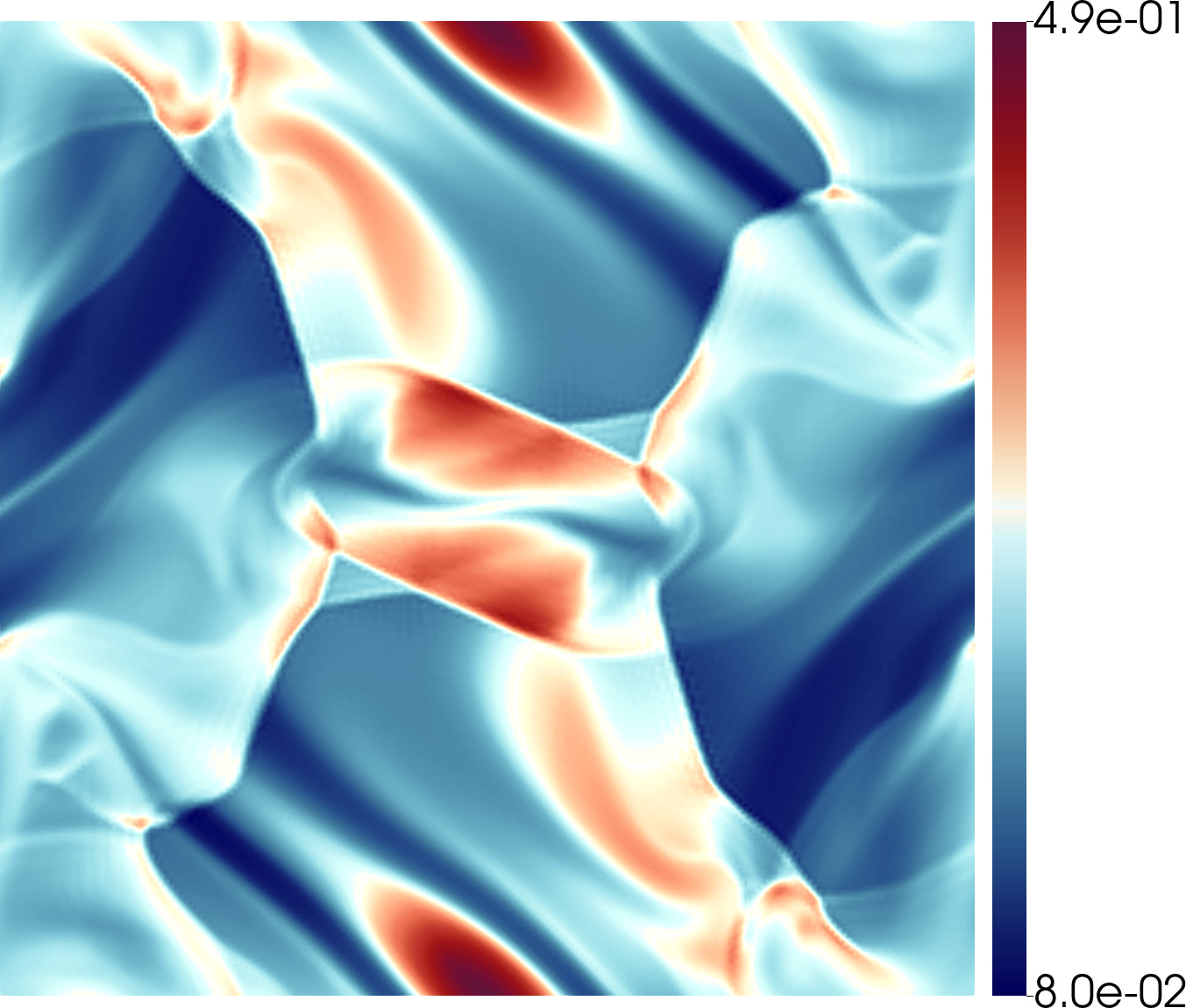}
    \caption*{Density $\rho_h$, $\polP_1$}
  \end{subfigure}
  \hspace{0.2in}
  \begin{subfigure}{0.4\textwidth}
    \centering
    \includegraphics[width=\textwidth]{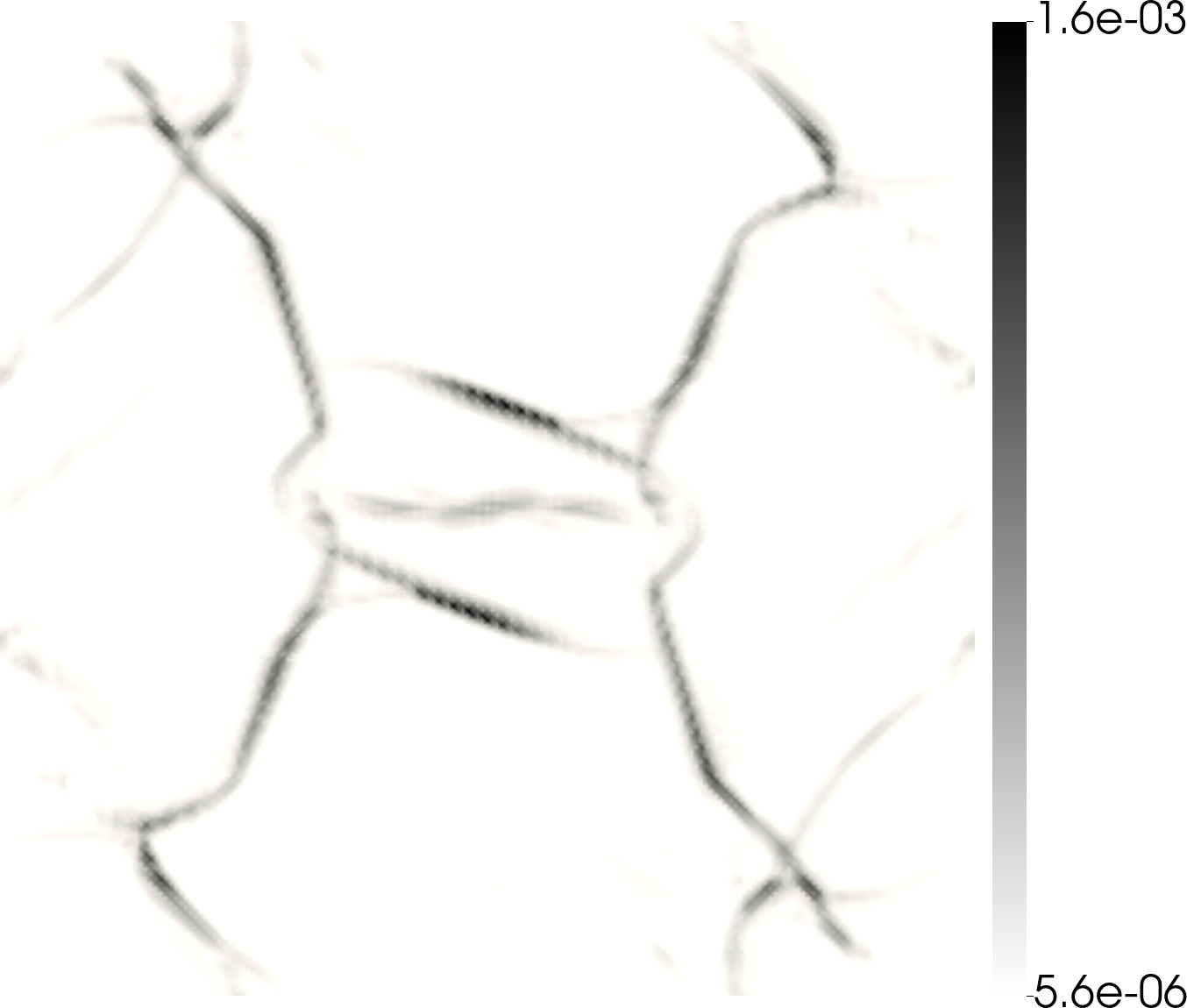}
    \caption*{Artificial viscosity $\e_h$, $\polP_1$}
  \end{subfigure}
  \\
  \vspace{0.2in}
  \begin{subfigure}{0.4\textwidth}
    \centering
    \includegraphics[width=\textwidth]{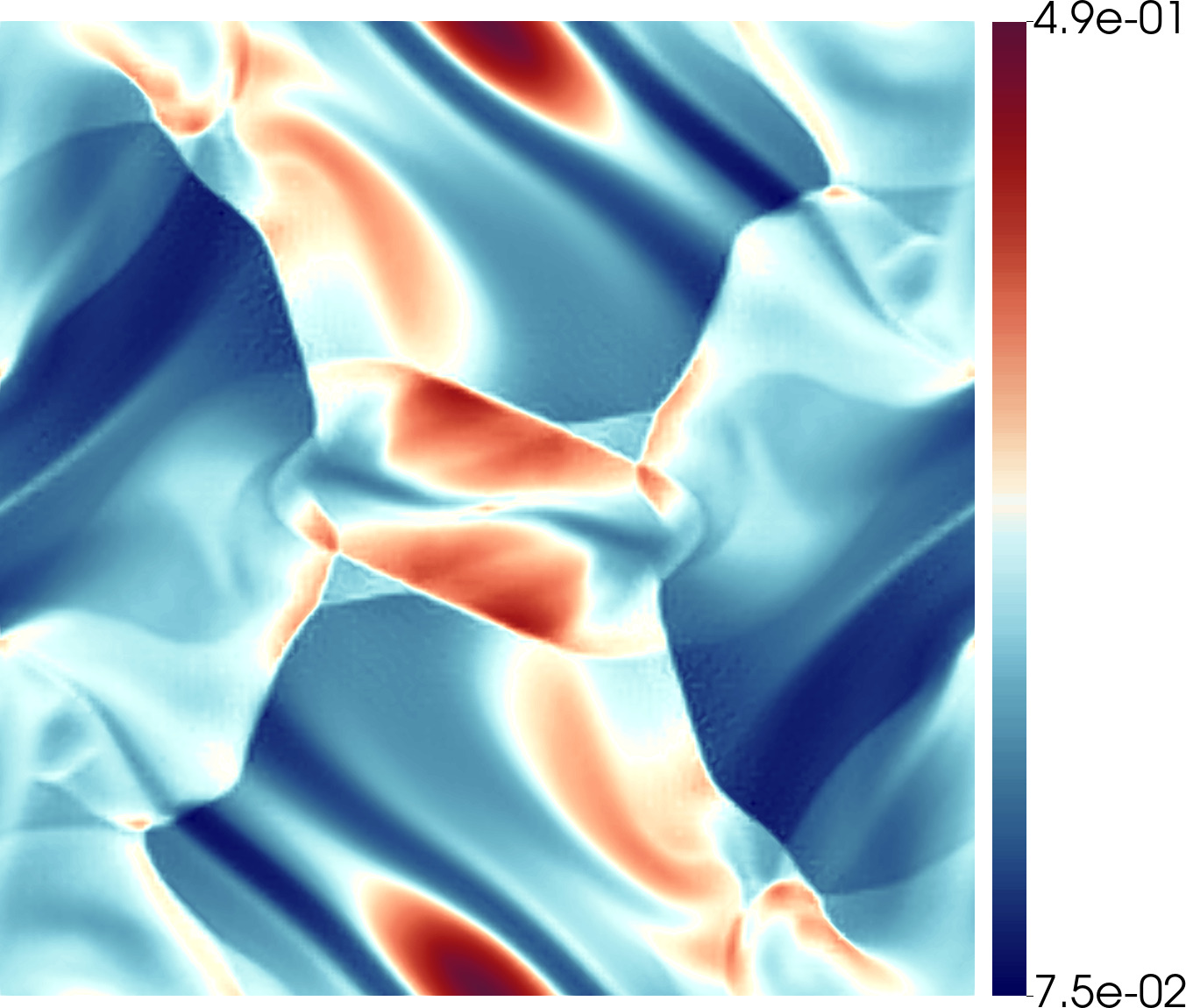}
    \caption*{Density $\rho_h$, $\polP_3$}
  \end{subfigure}
  \hspace{0.2in}
  \begin{subfigure}{0.4\textwidth}
    \centering
    \includegraphics[width=\textwidth]{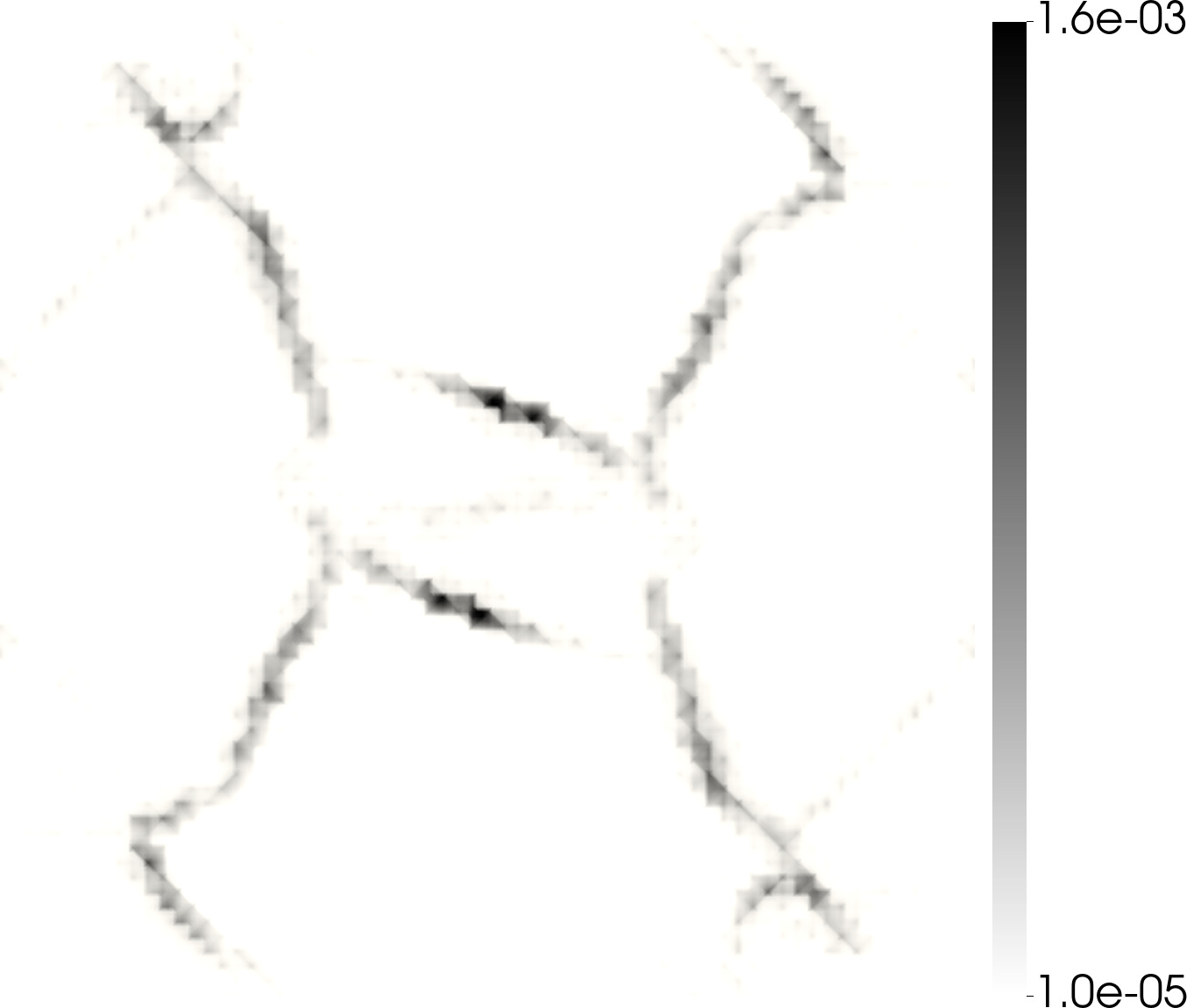}
    \caption*{Artificial viscosity $\e_h$, $\polP_3$}
  \end{subfigure}  
  \hfill
  \caption{$\polP_3$ solutions to Orszag-Tang problem at time $t=0.5$ on $200\times 200$ $\polP_1$ and $\polP_3$ nodes.}
  \label{fig:OT_t05}
\end{figure}

\begin{figure}[h!]
  \centering
  \begin{subfigure}{0.4\textwidth}
    \centering
    \includegraphics[width=\textwidth]{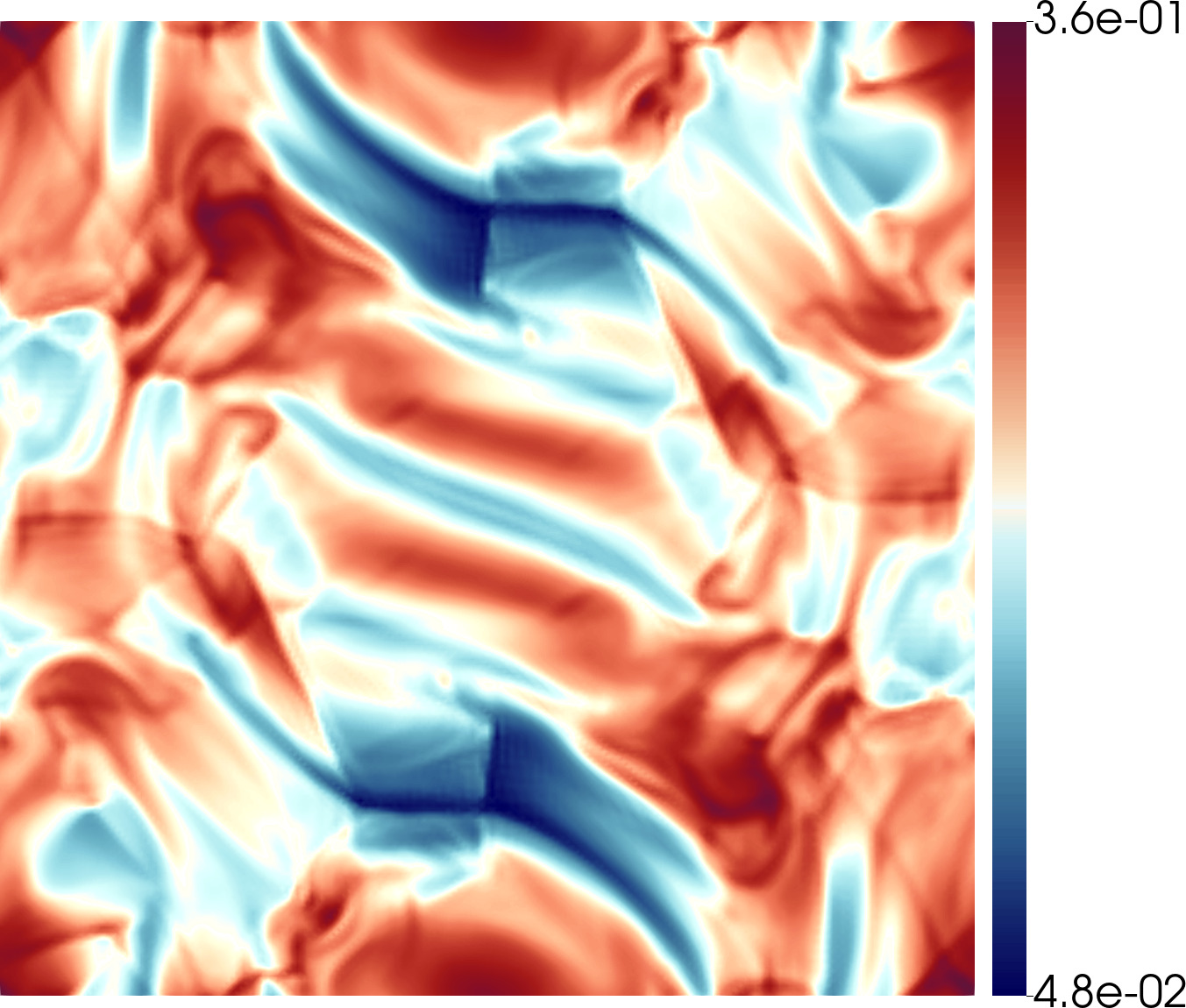}
    \caption*{Density $\rho_h$, $\polP_1$}
  \end{subfigure}
  \hspace{0.2in}
  \begin{subfigure}{0.4\textwidth}
    \centering
    \includegraphics[width=\textwidth]{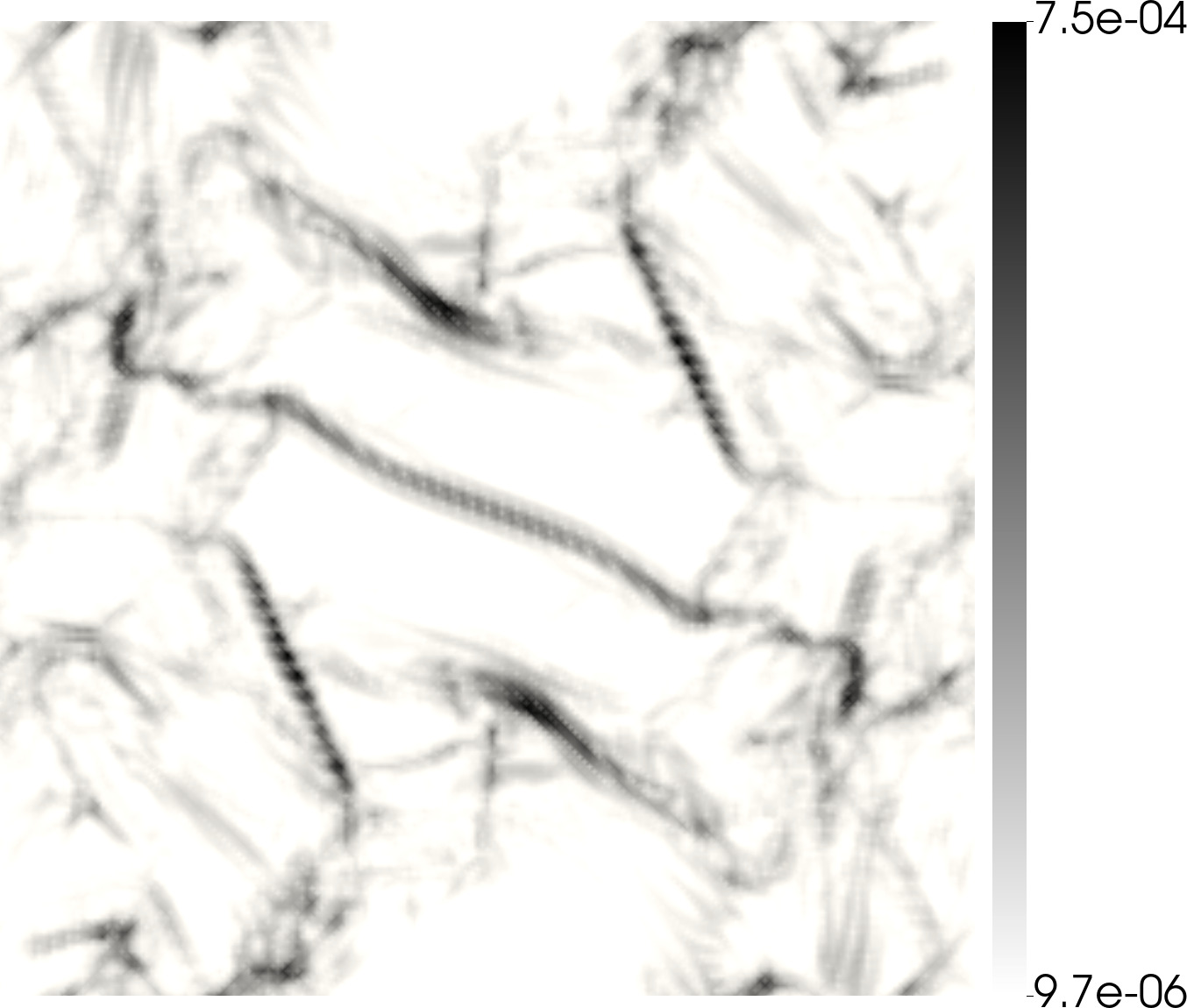}
    \caption*{Artificial viscosity $\e_h$, $\polP_1$}
  \end{subfigure}
  \\
  \vspace{0.2in}
  \begin{subfigure}{0.4\textwidth}
    \centering
    \includegraphics[width=\textwidth]{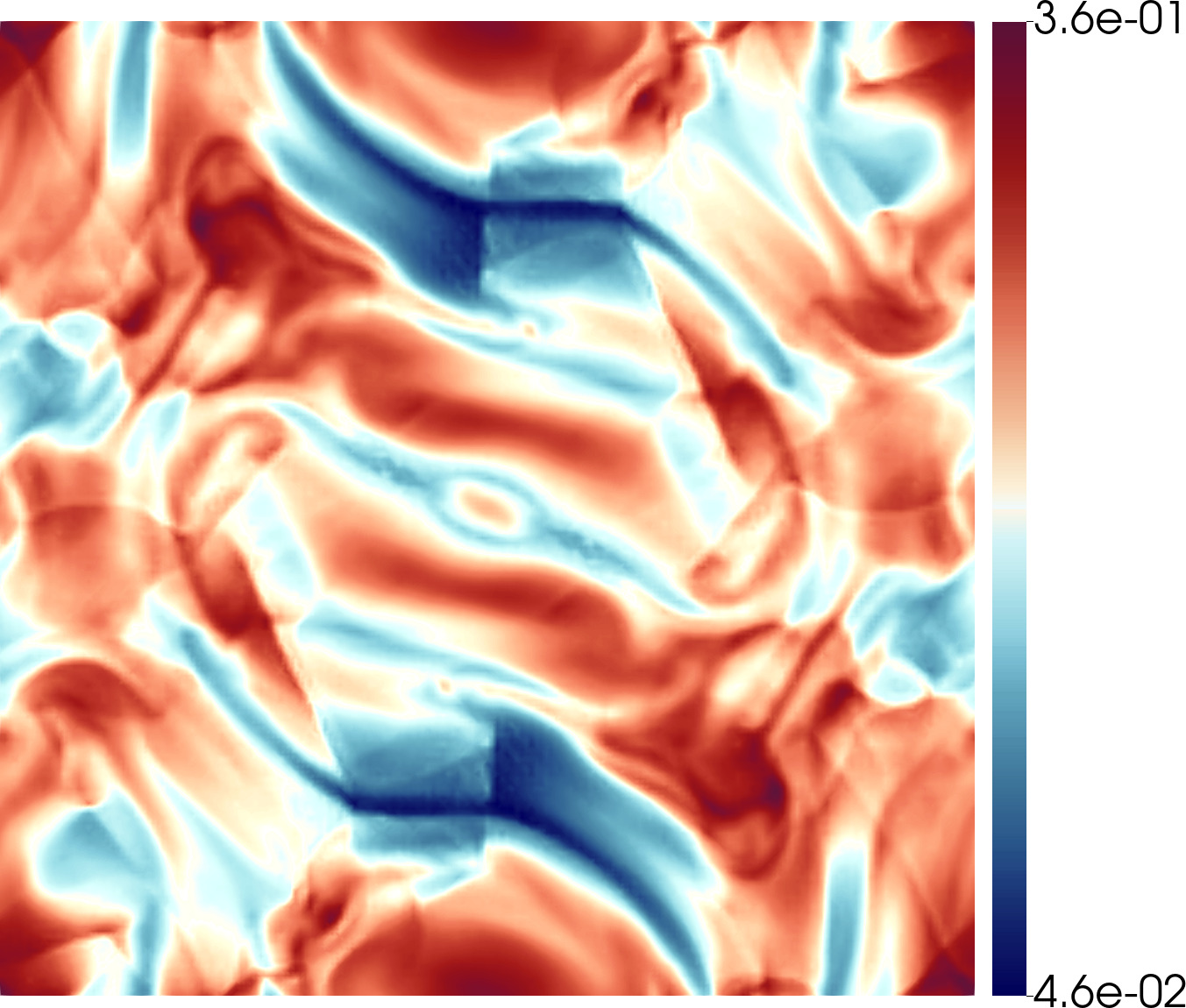}
    \caption*{Density $\rho_h$, $\polP_3$}
  \end{subfigure}
  \hspace{0.2in}
  \begin{subfigure}{0.4\textwidth}
    \centering
    \includegraphics[width=\textwidth]{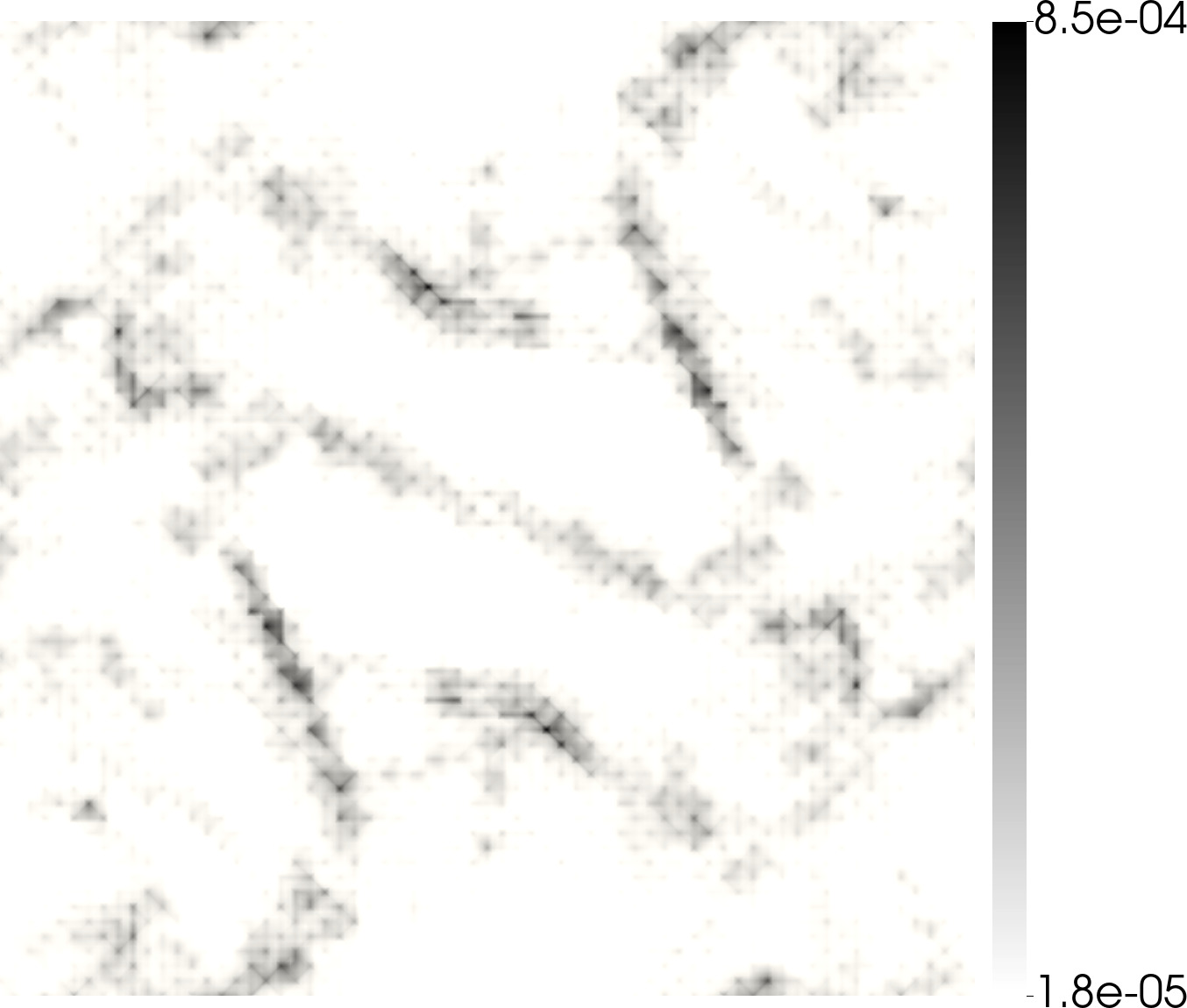}
    \caption*{Artificial viscosity $\e_h$, $\polP_3$}
  \end{subfigure}  
  \hfill
  \caption{$\polP_3$ solutions to Orszag-Tang problem at time $t=1$ on $200\times 200$ $\polP_1$ and $\polP_3$ nodes.}
  \label{fig:OT_t1}
\end{figure}

\red{We investigate the behavior of the divergence error of the numerical solutions in this test. For that purpose, the divergence error is calculated by
%\[
%\frac{L^2(\text{div}\bB_h)}{L^2((\nabla{\times}\bB_h)_z)},
%\]
\begin{equation}\label{eq:div_error_norm}
  \delta(t) := \frac{ \|\text{div}\bB_h\|_{L^2(\Omega)}}{\|\nabla{\times}\bB_h)_z\|_{L^2(\Omega)}},
\end{equation}
where $\text{div}\bB_h\in\calX_h$, that is computed by solving the projection problem $(\text{div}\bB_h,\varphi_h)=\int_{\Omega}(\DIV\bB_h)\varphi_h\ud\bx\;$ 
$\forall\varphi_h\in\calX_h$. The result is plotted in Figure~\ref{fig:OT_div}. The divergence error of the $\polP_3$ solutions are slightly higher than that of the $\polP_1$ solutions under the same number of computational nodes. However, the divergence error remains controlled over time for both polynomials, and seems to be improved under mesh refinement.}

\begin{figure}[h!]
  \centering
  \begin{subfigure}{0.4\textwidth}
    \centering
    \includegraphics[width=\textwidth]{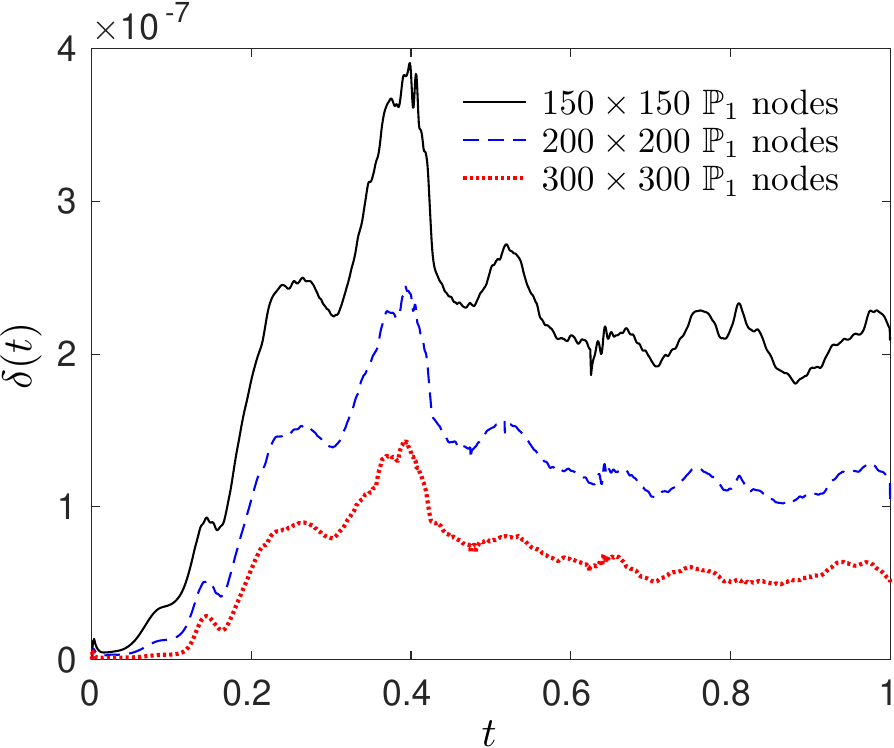}
    \caption*{$\polP_1$ solutions}
  \end{subfigure}
  \hspace{0.2in}
  \begin{subfigure}{0.4\textwidth}
    \centering
    \includegraphics[width=\textwidth]{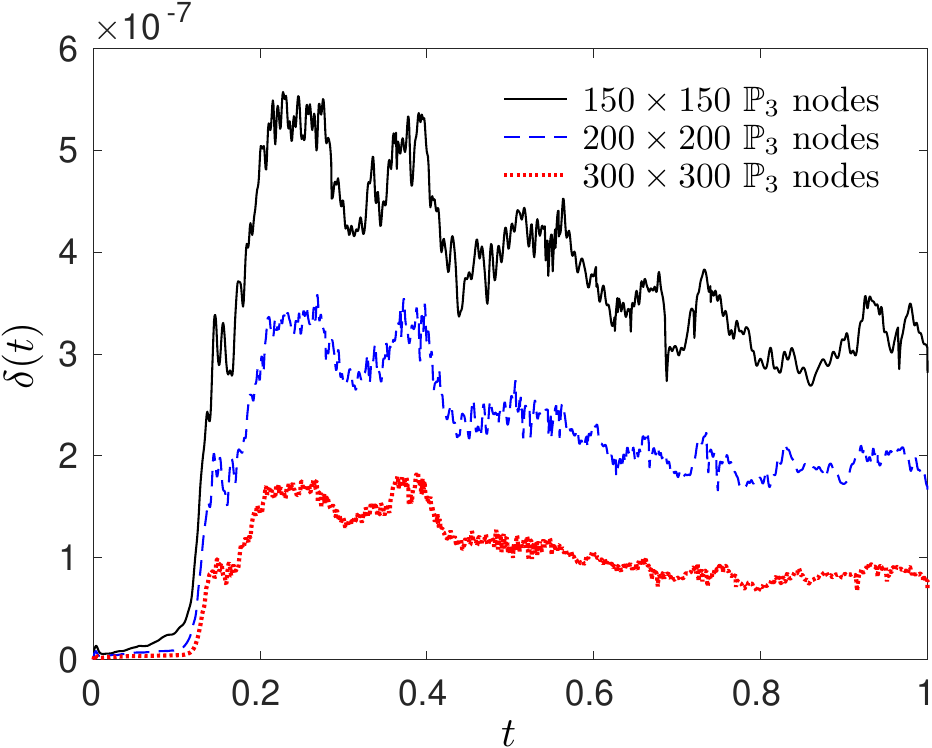}
    \caption*{$\polP_3$ solutions}
  \end{subfigure}
  \hfill
  \caption{Divergence error defined in \eqref{eq:div_error_norm} of the Orszag-Tang solutions on different mesh sizes and polynomials}
  \label{fig:OT_div}
\end{figure}

%\end{document}

%%%%%%%%%%%%%%%%%%%%%%%%%%%%%%%%%%%%%%%%%%%% 
\red{
\subsection{Kelvin-Helmholtz instability}
We solve the so-called Kelvin-Helmholtz instability problem in a square domain $\Omega = [-0.5, 0.5]\times [-0.5, 0.5]$. The boundary condition is periodic in all directions. The initial condition for this problem is given as
\[
(\rho_0, \bu_0, p_0, \bB_0) = 
\begin{cases}
\Big(2, \big( &\!\!\!\!\!0.5, 0), 2.5, \big(b_x, 0\big) \Big), \quad \textrm{if } |y| \le 0.25,\\
\Big(1, \big(-&\!\!\!\!\!0.5, 0), 2.5, \big(b_x, 0\big) \Big), \quad \textrm{if } |y| > 0.25,\\
\end{cases}
\]
where $b_x$ is the magnetic field in the $x$ direction. In the numerical simulation below, we use two values: $b_x=0$ corresponds to hydrodynamic flow, and $b_x=0.2$ corresponds to magnetohydrodynamic flow.
In addition, for each nodal point of the finite elements, we add random numbers from a discrete uniform distribution in the interval -0.005 to 0.005 to the initial velocity field in both $x$ and $y$ components. Due to different values in density, the appearance of velocity shear, and noise in the initial velocity, instabilities in the solution develop over time. This instability is usually referred to as the Kelvin-Helmholtz instability. To capture and resolve this instability in time, high-order accuracy is required for the numerical methods. 

We solve the problem on two meshes with $85\times 85$ and $170\times 170$ grid points using the space of cubic polynomials $\polP_3$. The corresponding number of nodal points for these two meshes are $255\times255$ and $510\times510$ respectively. We use the classical fourth-order Runge-Kutta method with $\textrm{CFL}=0.4$.

Figure \ref{fig:KH:b0} shows the density profile for the numerical simulation for the hydrodynamic case, \ie. $b_x=0$, at time levels $t=1,2,3,4,5,6$ . We see that the small-scale features of the solution are well-resolved even for coarser mesh, which confirms the high-order accuracy of the proposed method in this paper. Figure \ref{fig:KH:b02} depicts the solution of the full magnetohydrodynamics case, where $b_x=0.2$. Here we observe that in contrast to the hydrodynamic case, the presence of the magnetic field stabilizes the flow and suppresses small-scale fluctuation. 

\begin{figure}[h!]
 \centering
  $255 \times 255$ $\polP_3$ nodes  
  \\
  \begin{subfigure}{0.32\textwidth}
    \caption*{$t=1$}
    \includegraphics[width=\textwidth]{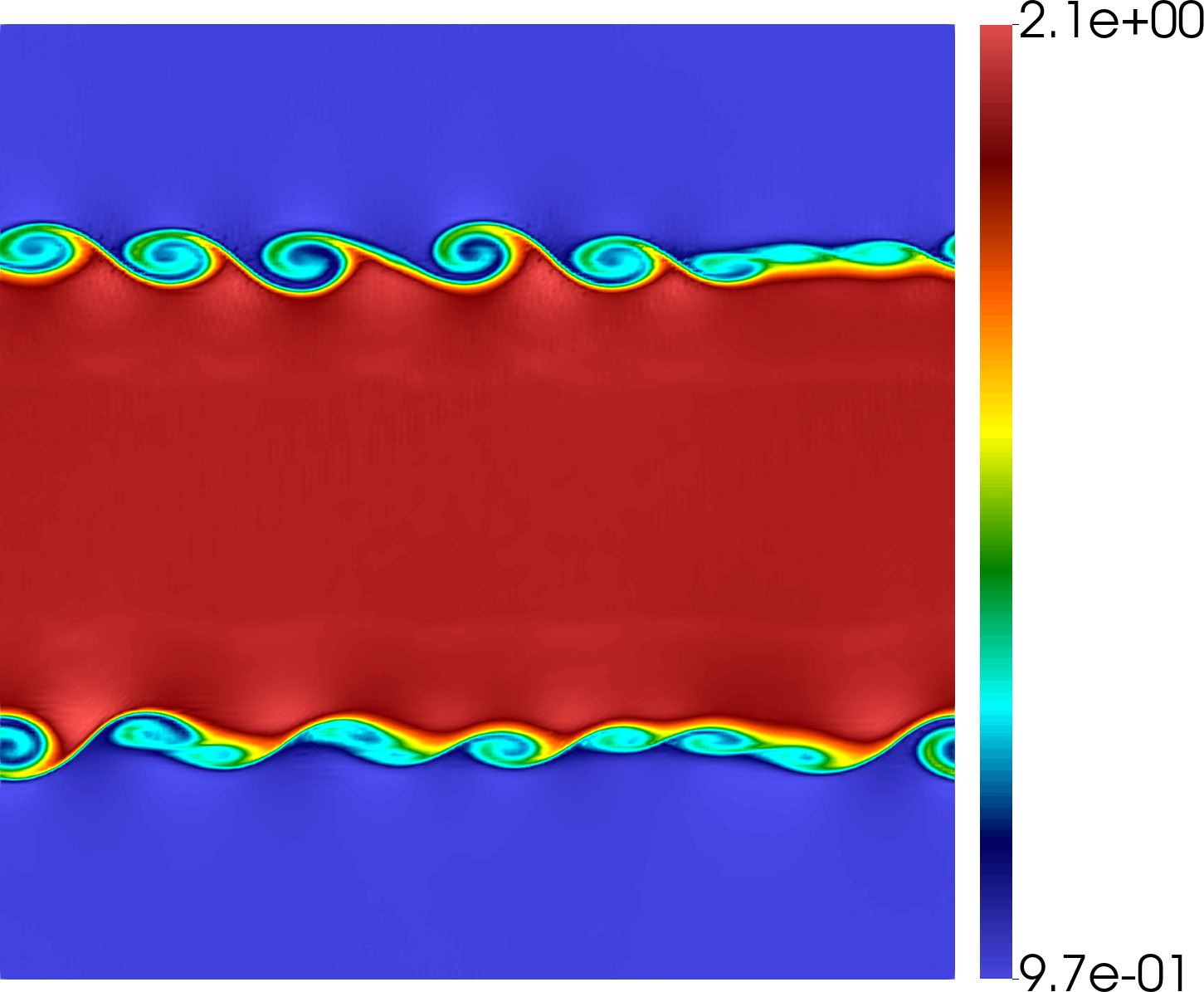}
  \end{subfigure}
  \begin{subfigure}{0.32\textwidth}
    \caption*{$t=2$}
    \includegraphics[width=\textwidth]{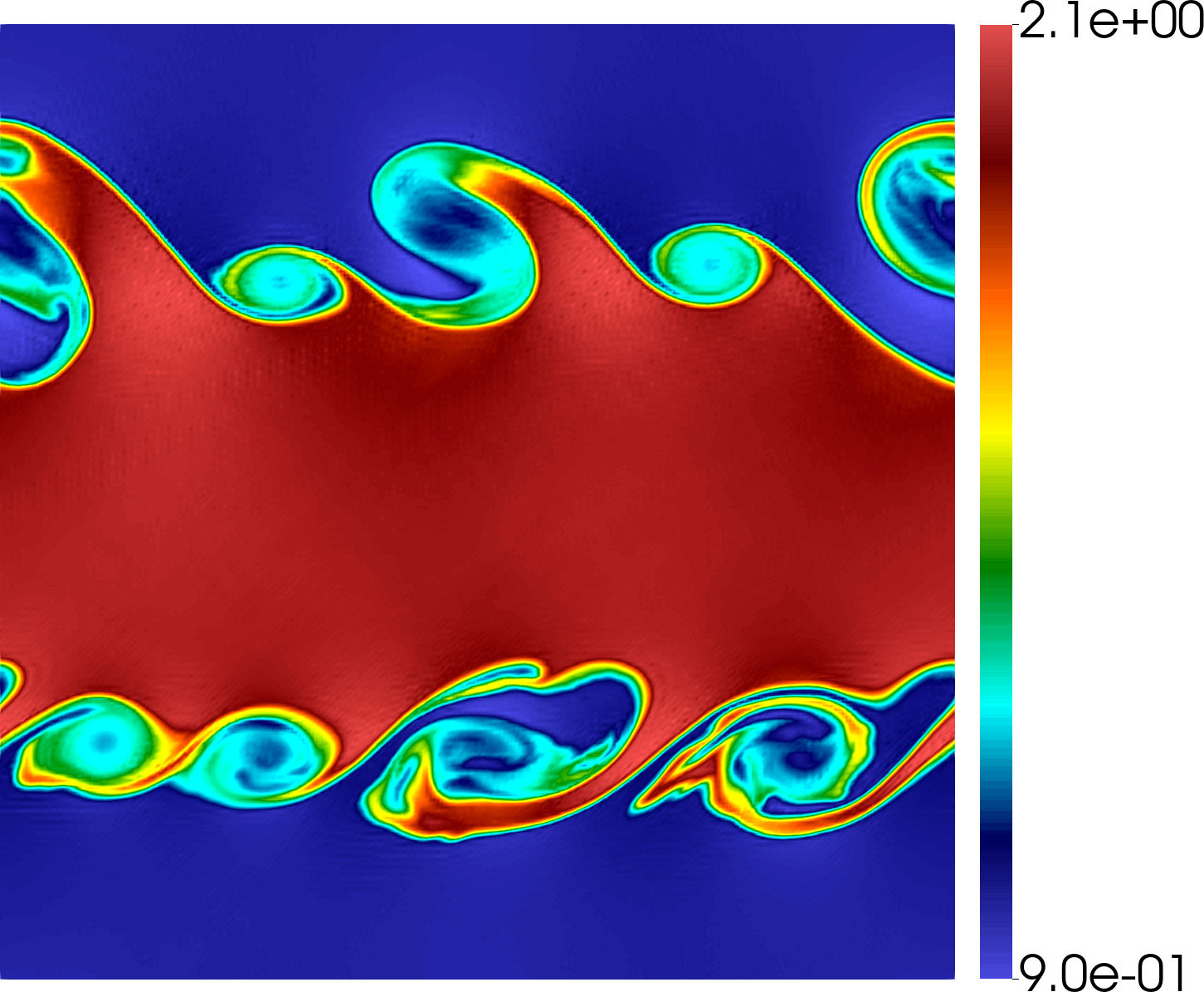}
  \end{subfigure}
  \begin{subfigure}{0.32\textwidth}
    \caption*{$t=3$}
    \includegraphics[width=\textwidth]{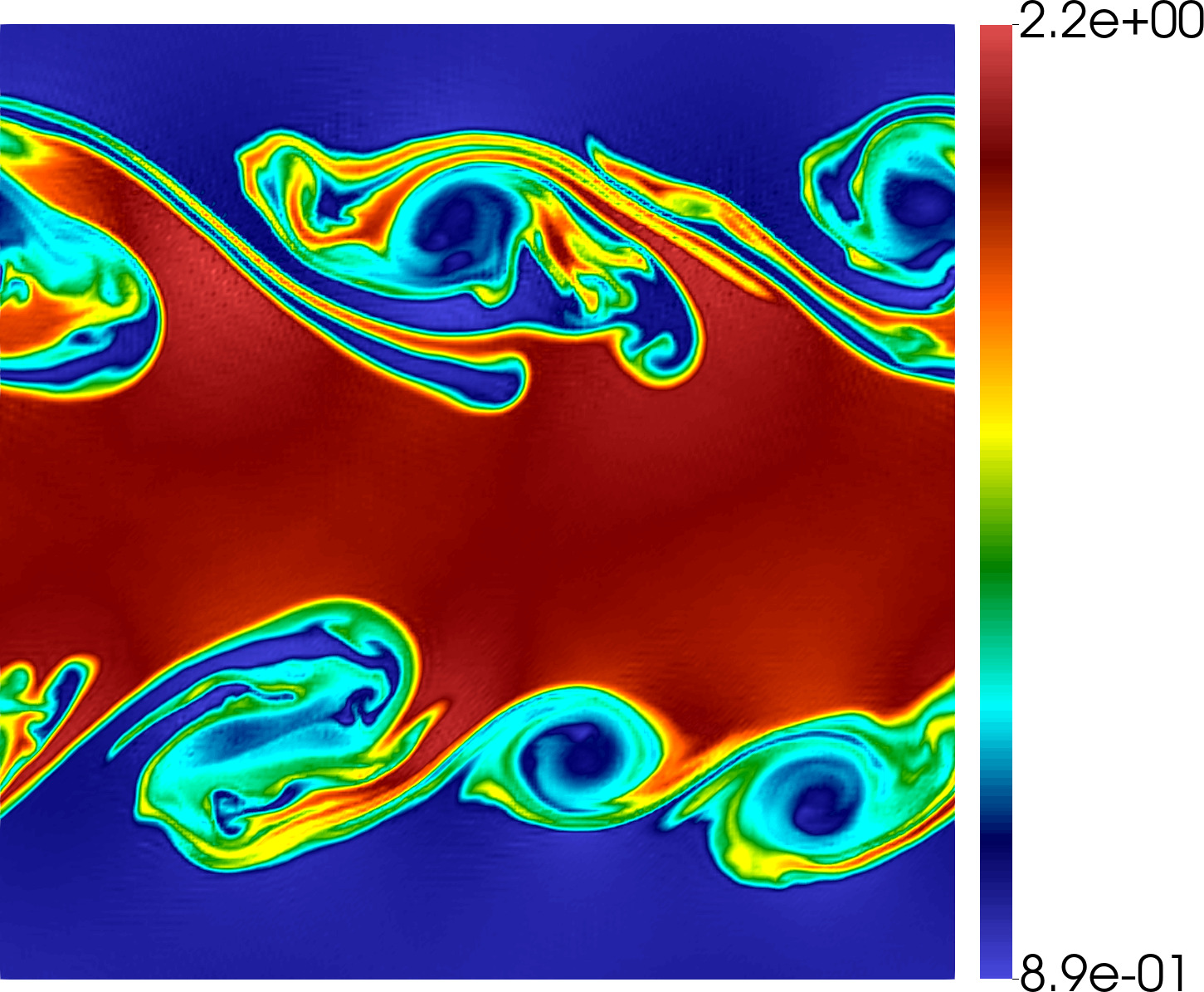}
  \end{subfigure}
  \\
  \vspace{0.05in}
  \begin{subfigure}{0.32\textwidth}
    \caption*{$t=4$}
    \includegraphics[width=\textwidth]{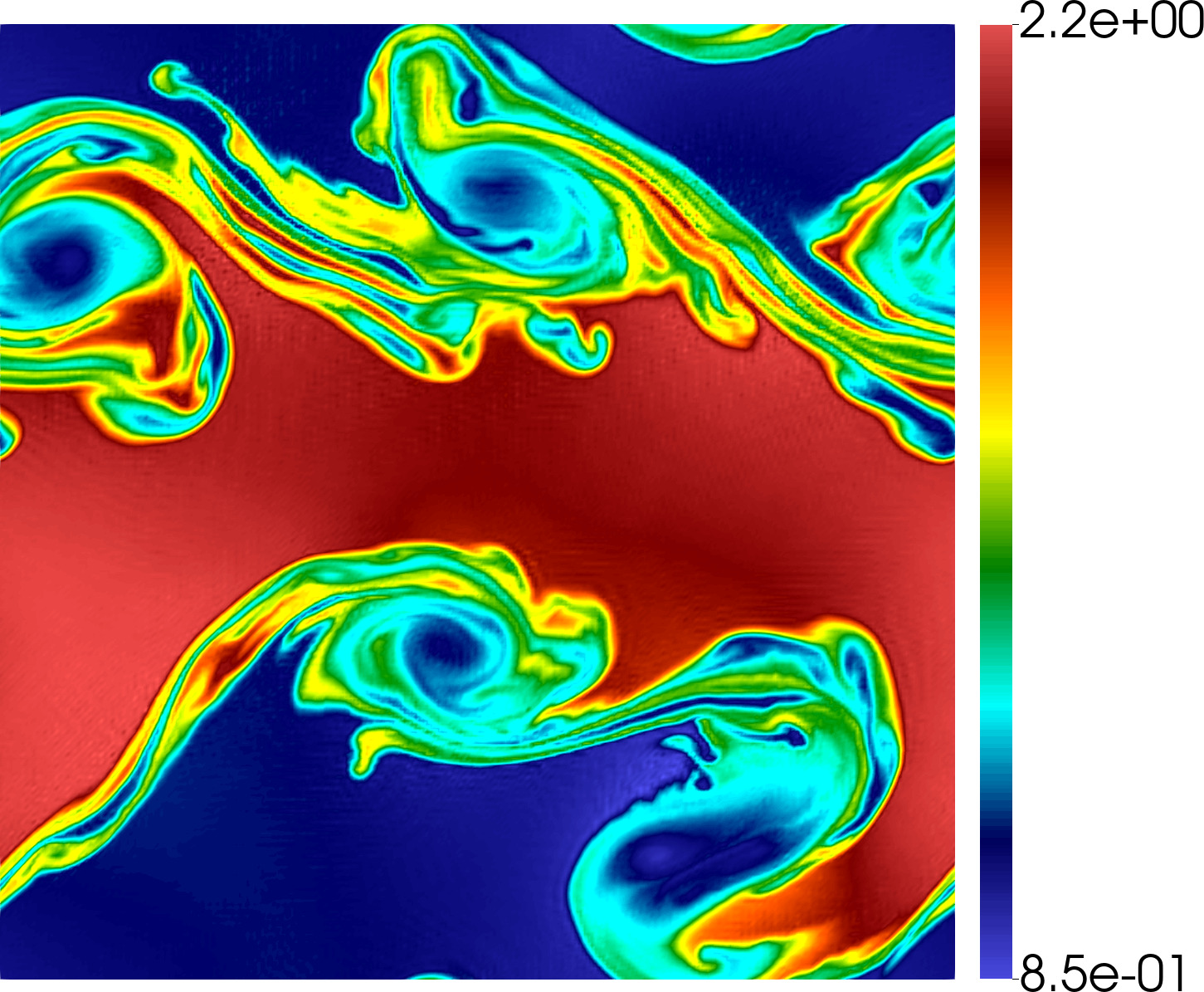}
  \end{subfigure}
  \begin{subfigure}{0.32\textwidth}
    \caption*{$t=5$}
    \includegraphics[width=\textwidth]{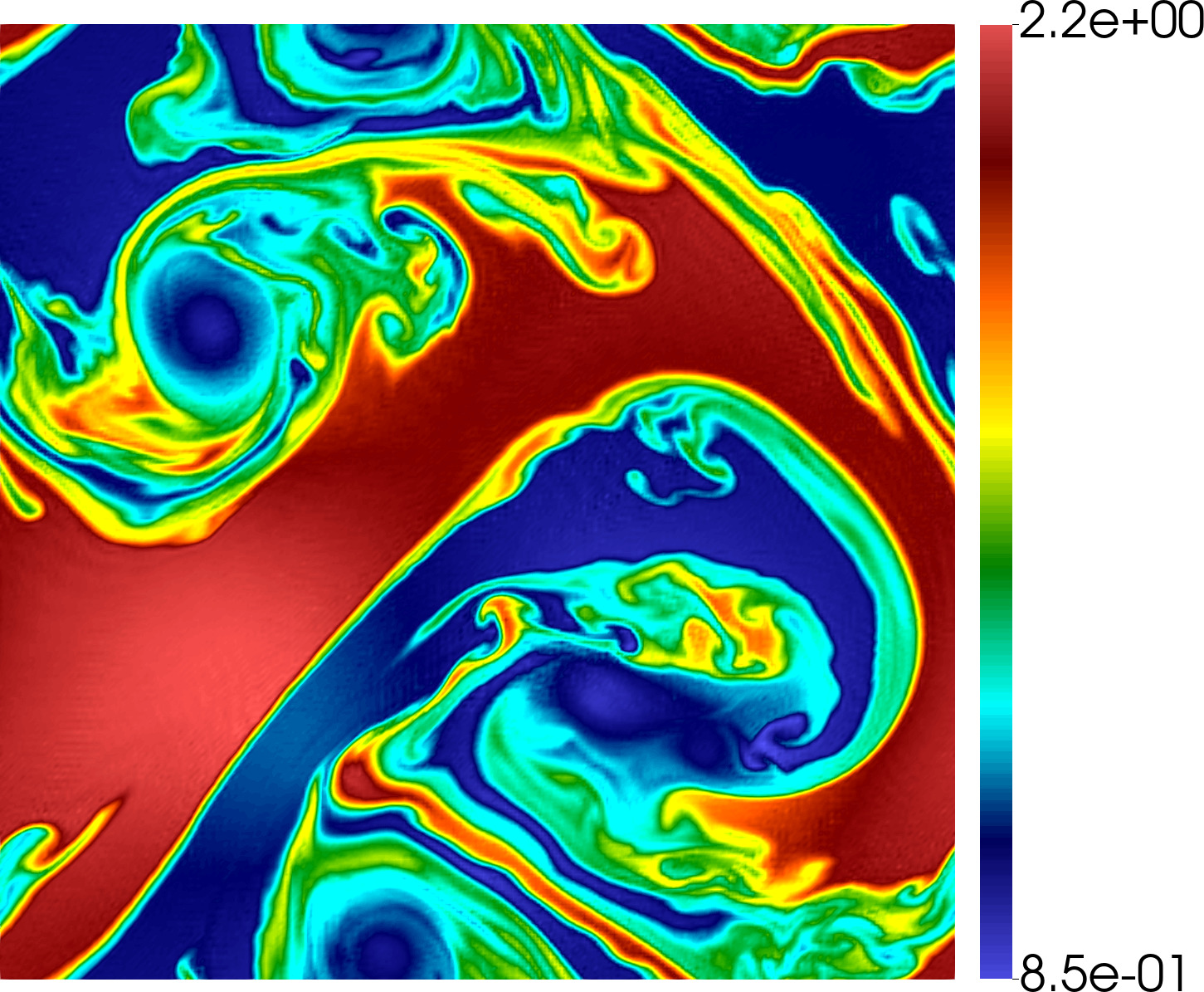}
  \end{subfigure}
  \begin{subfigure}{0.32\textwidth}
    \caption*{$t=6$}
    \includegraphics[width=\textwidth]{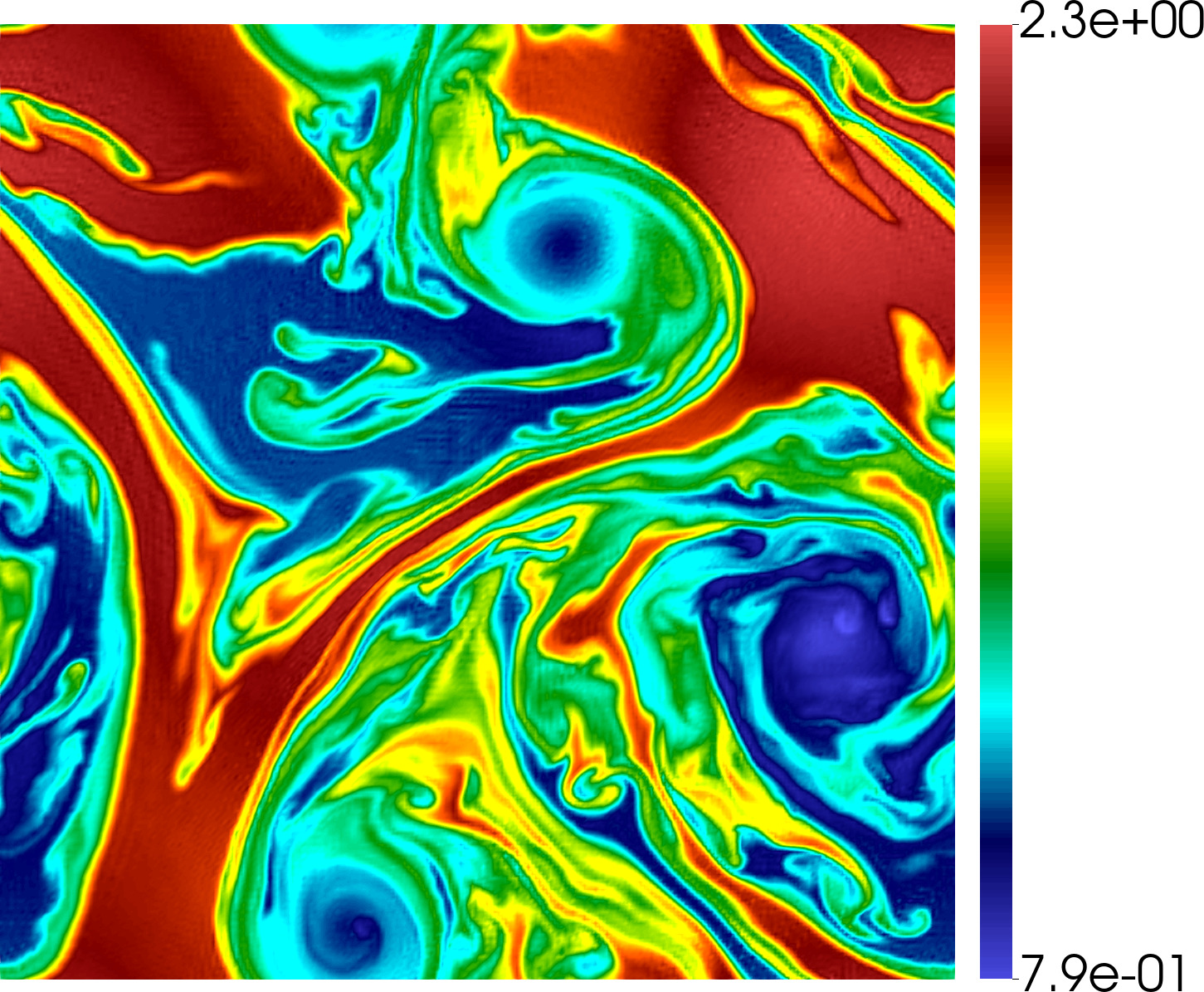}
  \end{subfigure}  
  \\
  \vspace{0.3in}
  $510 \times 510$ $\polP_3$ nodes  
  \\
    \begin{subfigure}{0.32\textwidth}
    \caption*{$t=1$}
    \includegraphics[width=\textwidth]{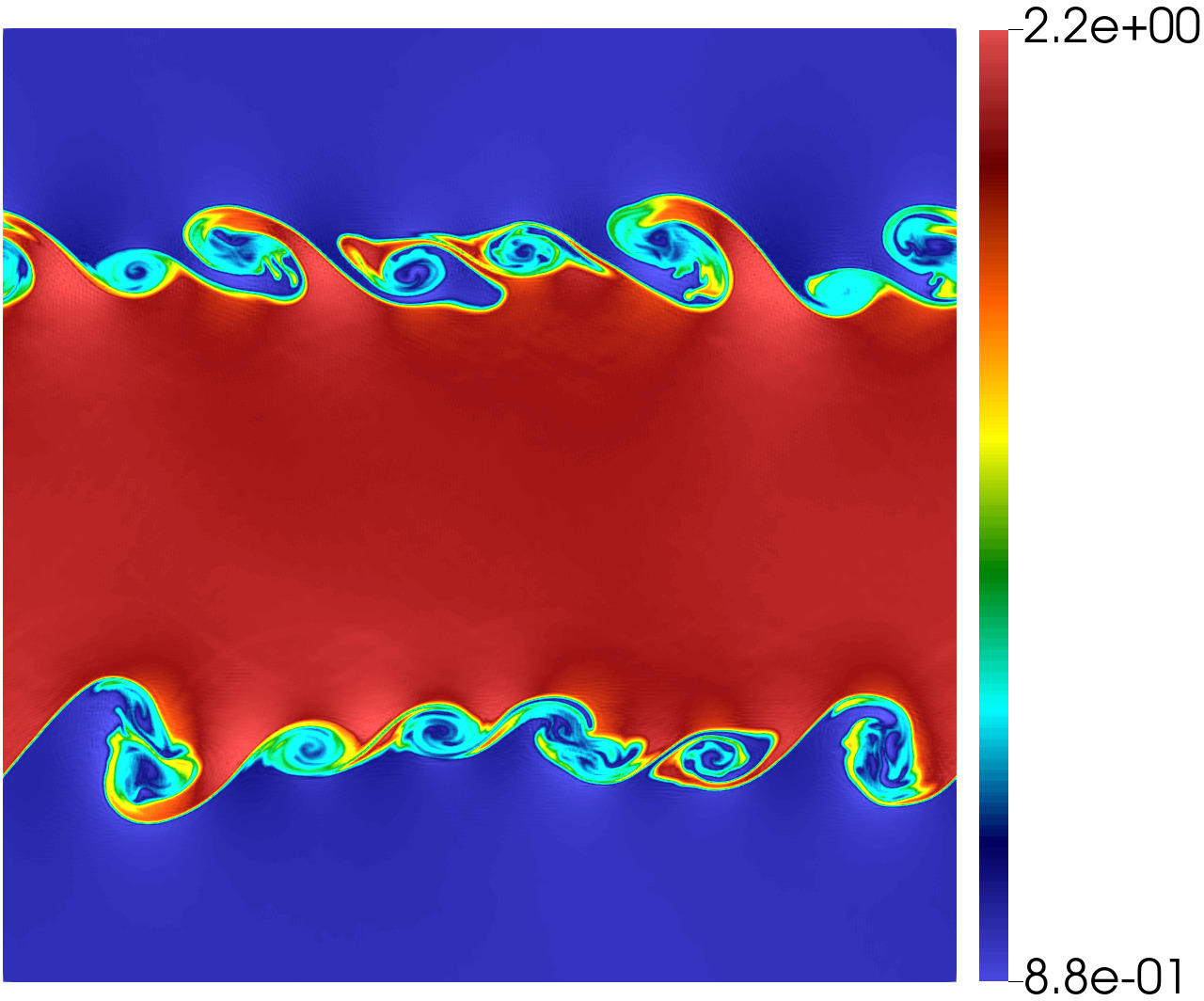}
  \end{subfigure}
  \begin{subfigure}{0.32\textwidth}
    \caption*{$t=2$}
    \includegraphics[width=\textwidth]{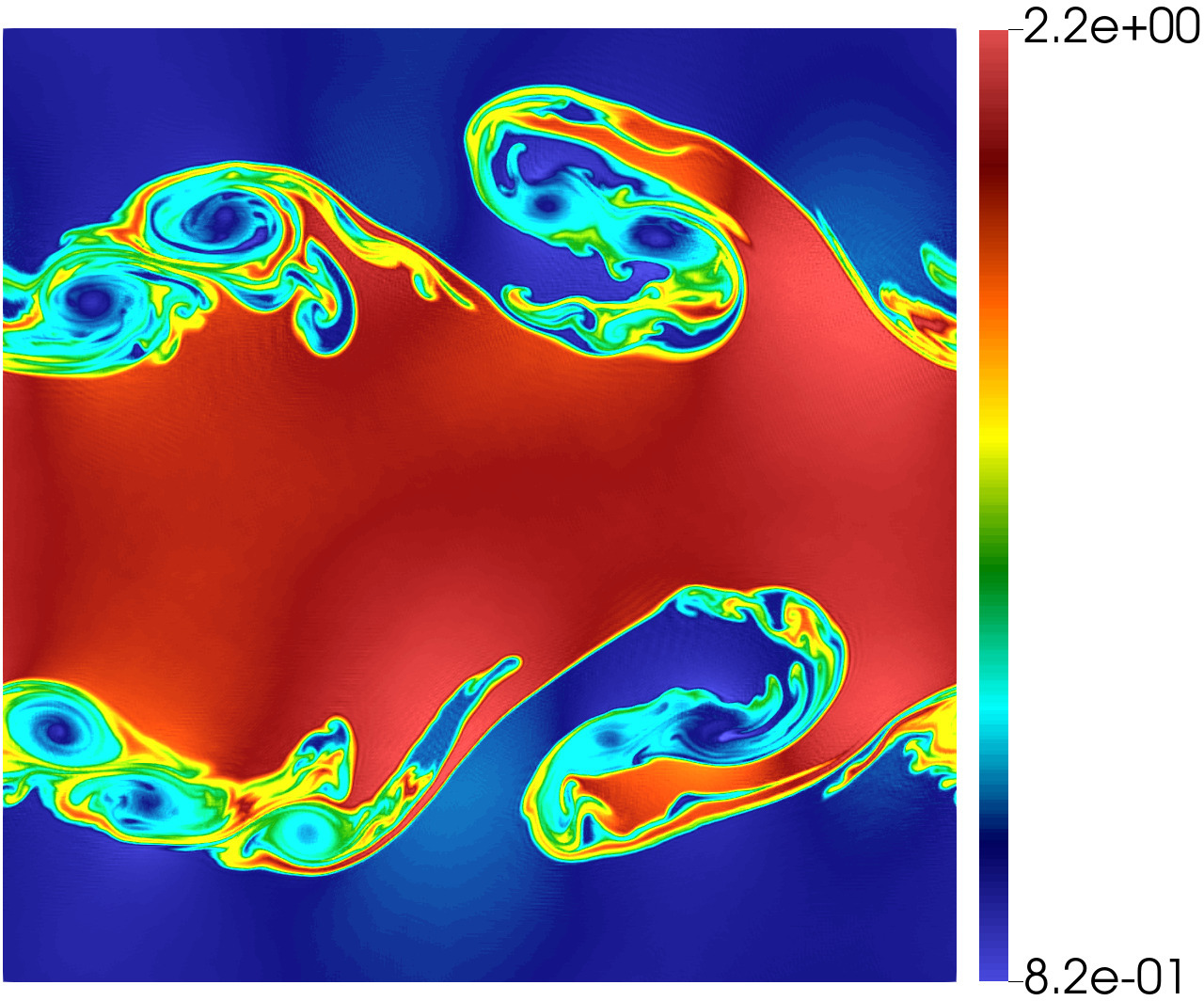}
  \end{subfigure}
  \begin{subfigure}{0.32\textwidth}
    \caption*{$t=3$}
    \includegraphics[width=\textwidth]{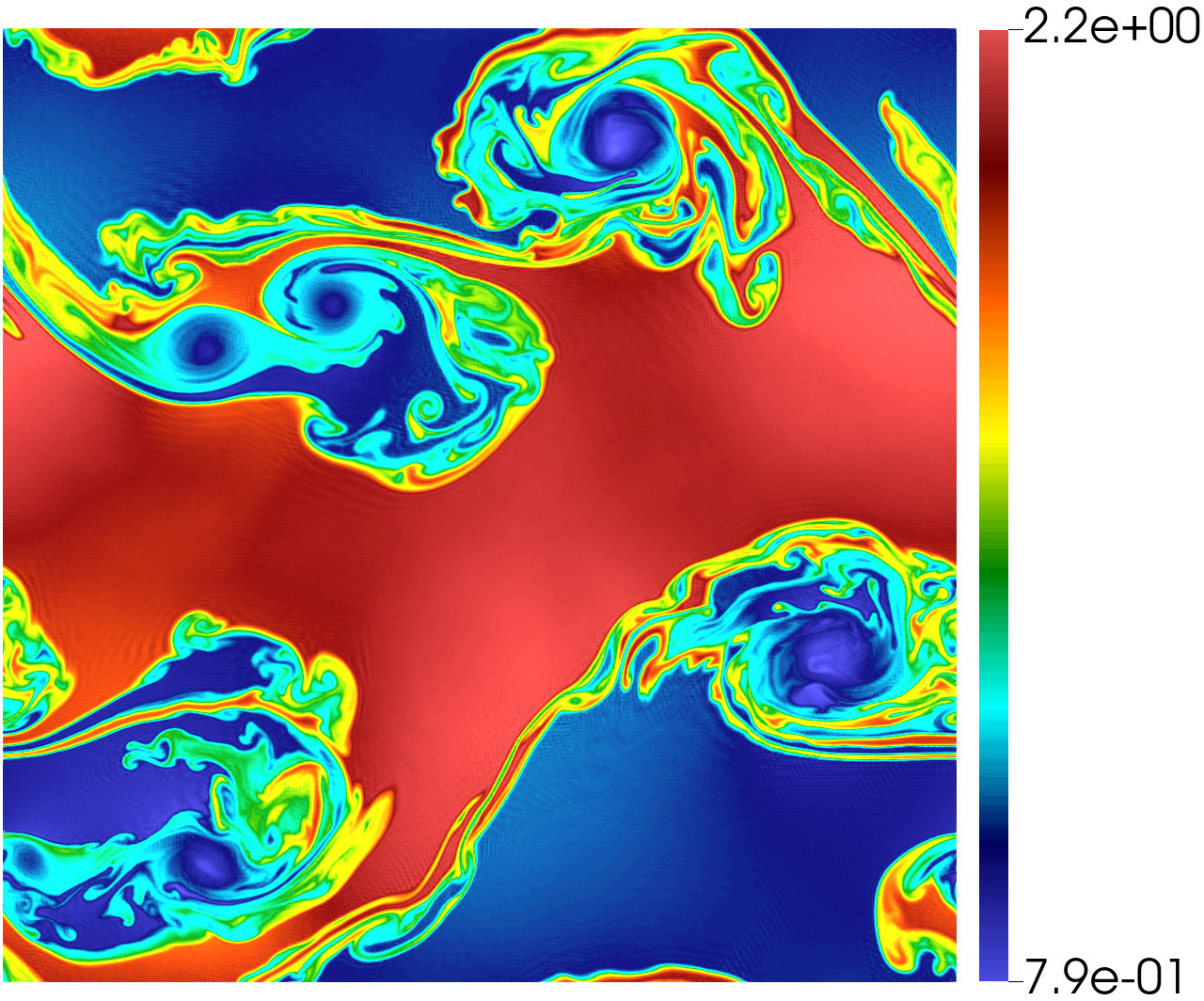}
  \end{subfigure}
  \\
  \vspace{0.05in}
  \begin{subfigure}{0.32\textwidth}
    \caption*{$t=4$}
    \includegraphics[width=\textwidth]{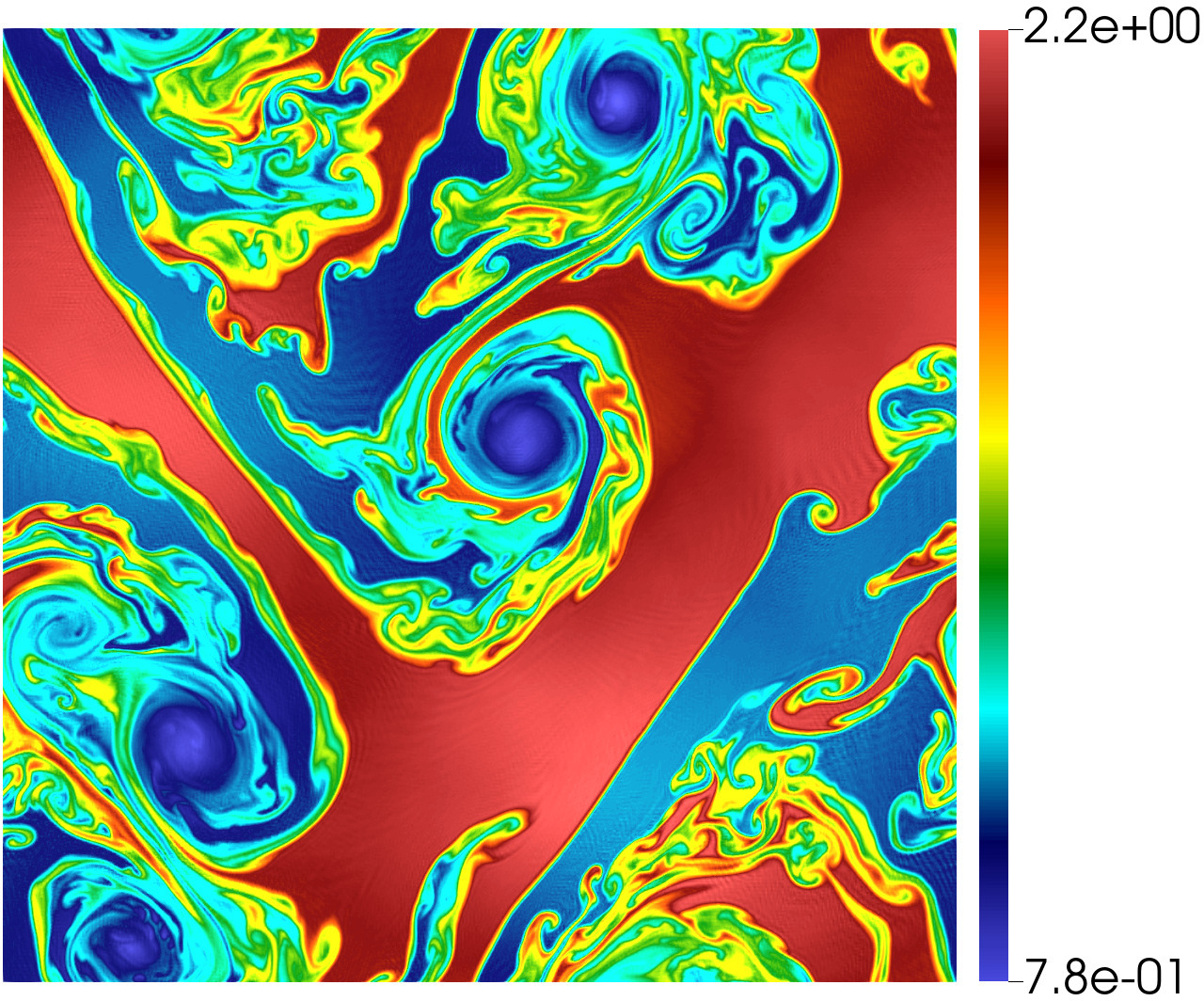}
  \end{subfigure}
  \begin{subfigure}{0.32\textwidth}
    \caption*{$t=5$}
    \includegraphics[width=\textwidth]{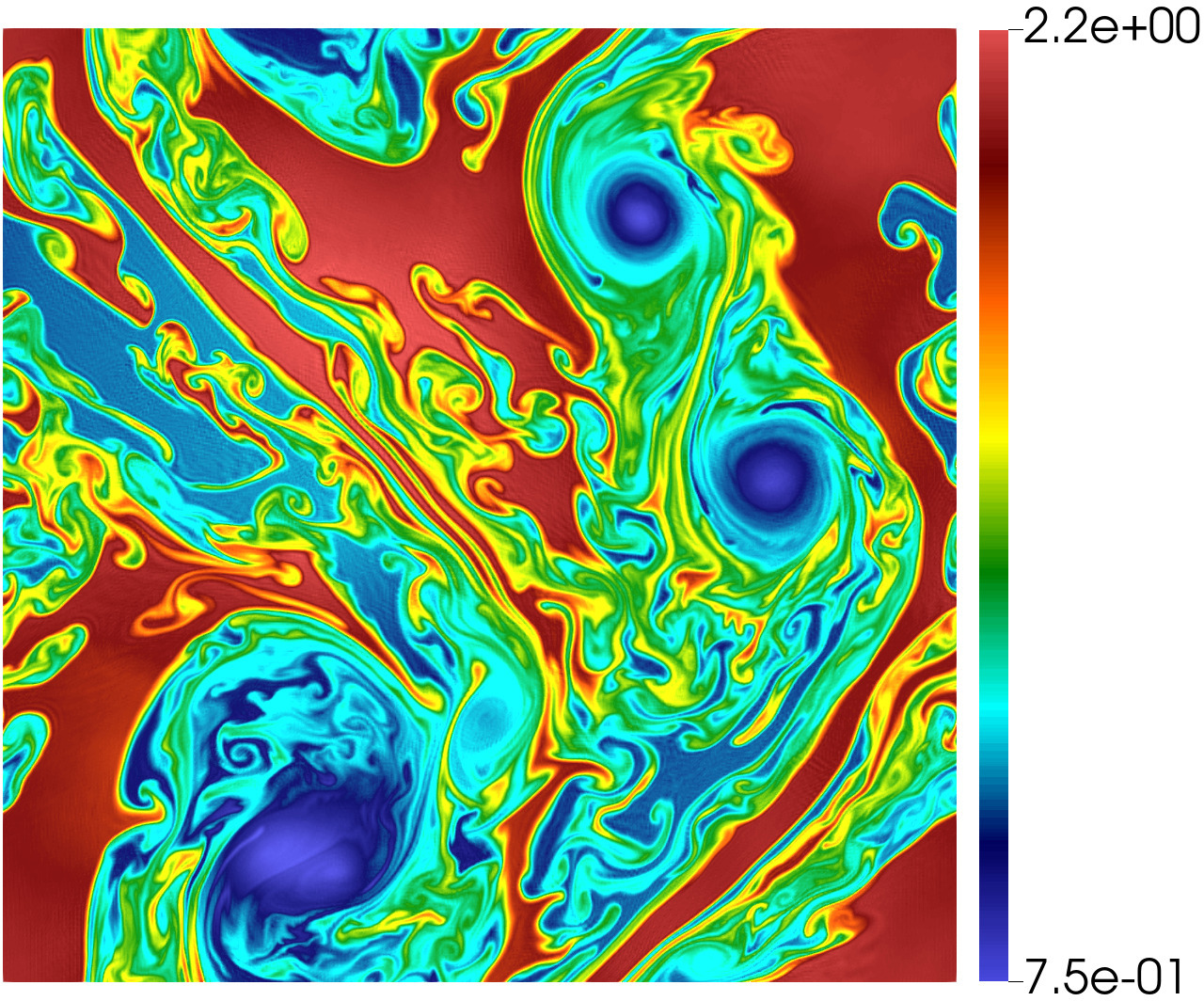}
  \end{subfigure}
  \begin{subfigure}{0.32\textwidth}
    \caption*{$t=6$}
    \includegraphics[width=\textwidth]{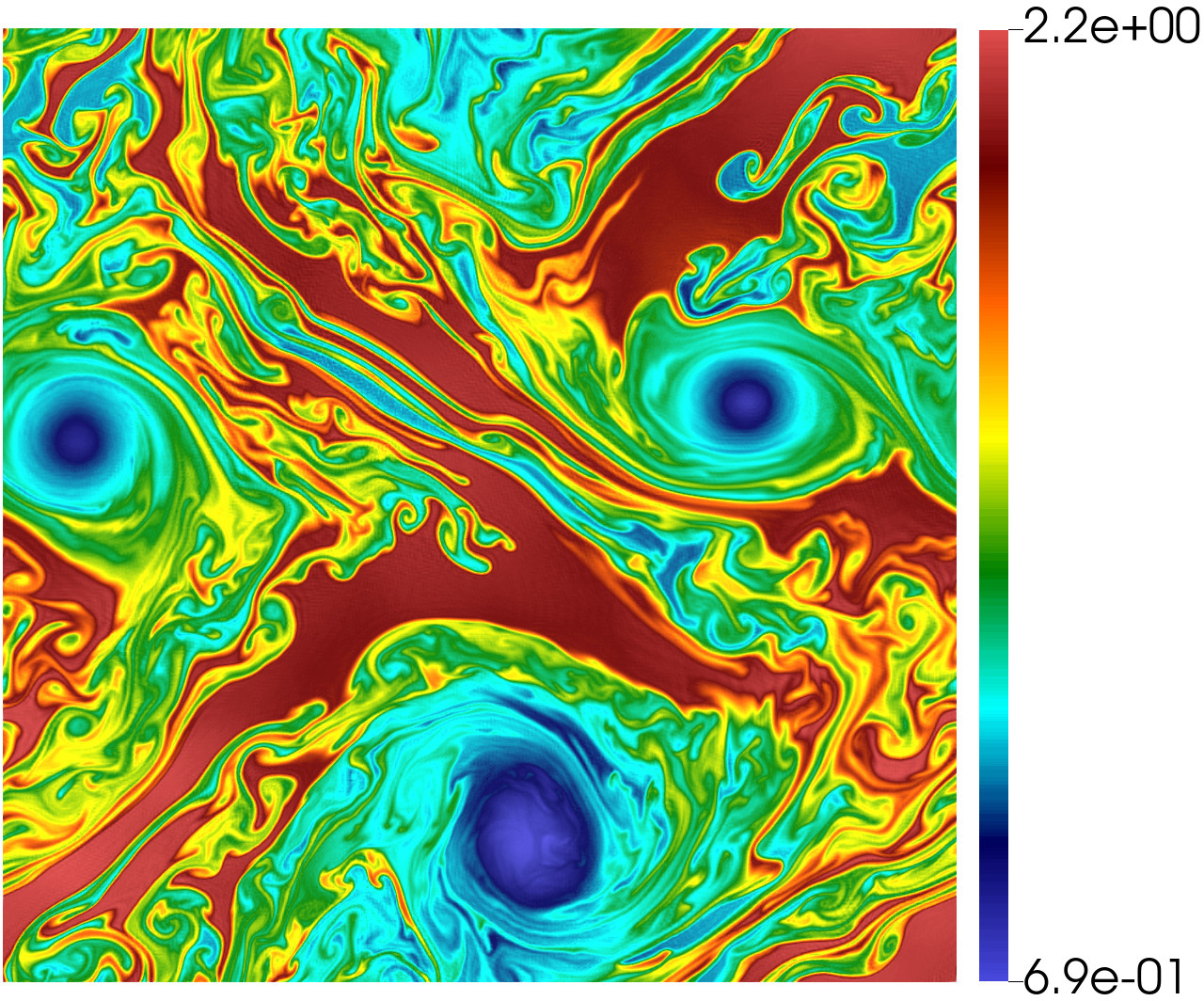}
  \end{subfigure} 
  \caption{\red{Kelvin-Helholtz instabilty. $\polP_3$ solution of fully Hydrodynamic regime, $b_x=0$. The density profile is plotted at different time levels for two mesh resolutions: $85 \times 85$ and  $170 \times 170$ vertices.}}
  \label{fig:KH:b0}
\end{figure}

\begin{figure}[h!]
  \centering
  $255 \times 255$ $\polP_3$ nodes  
  \\
  \begin{subfigure}{0.32\textwidth}
    \caption*{$t=1$}
    \includegraphics[width=\textwidth]{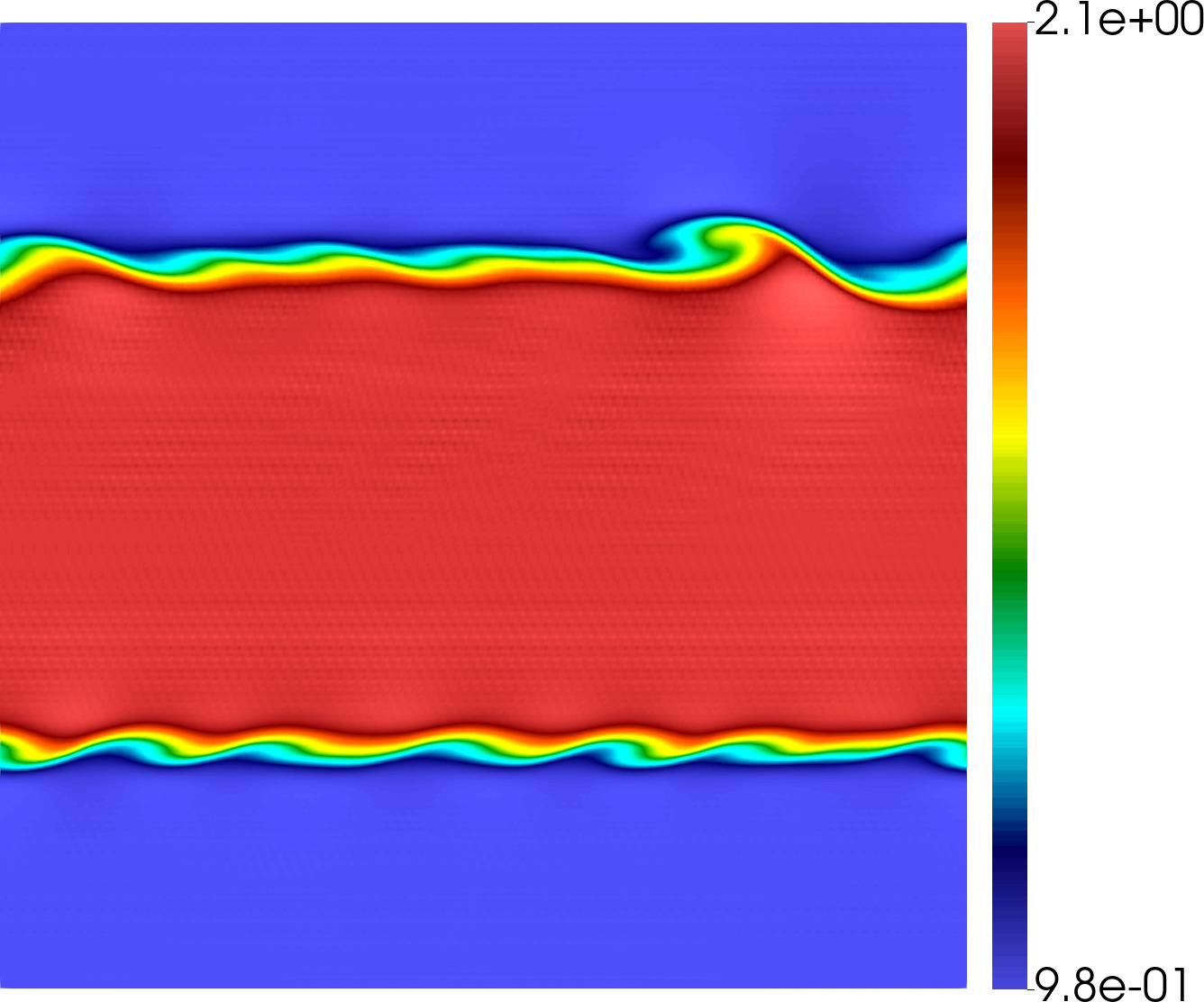}
  \end{subfigure}
  \begin{subfigure}{0.32\textwidth}
    \caption*{$t=2$}
    \includegraphics[width=\textwidth]{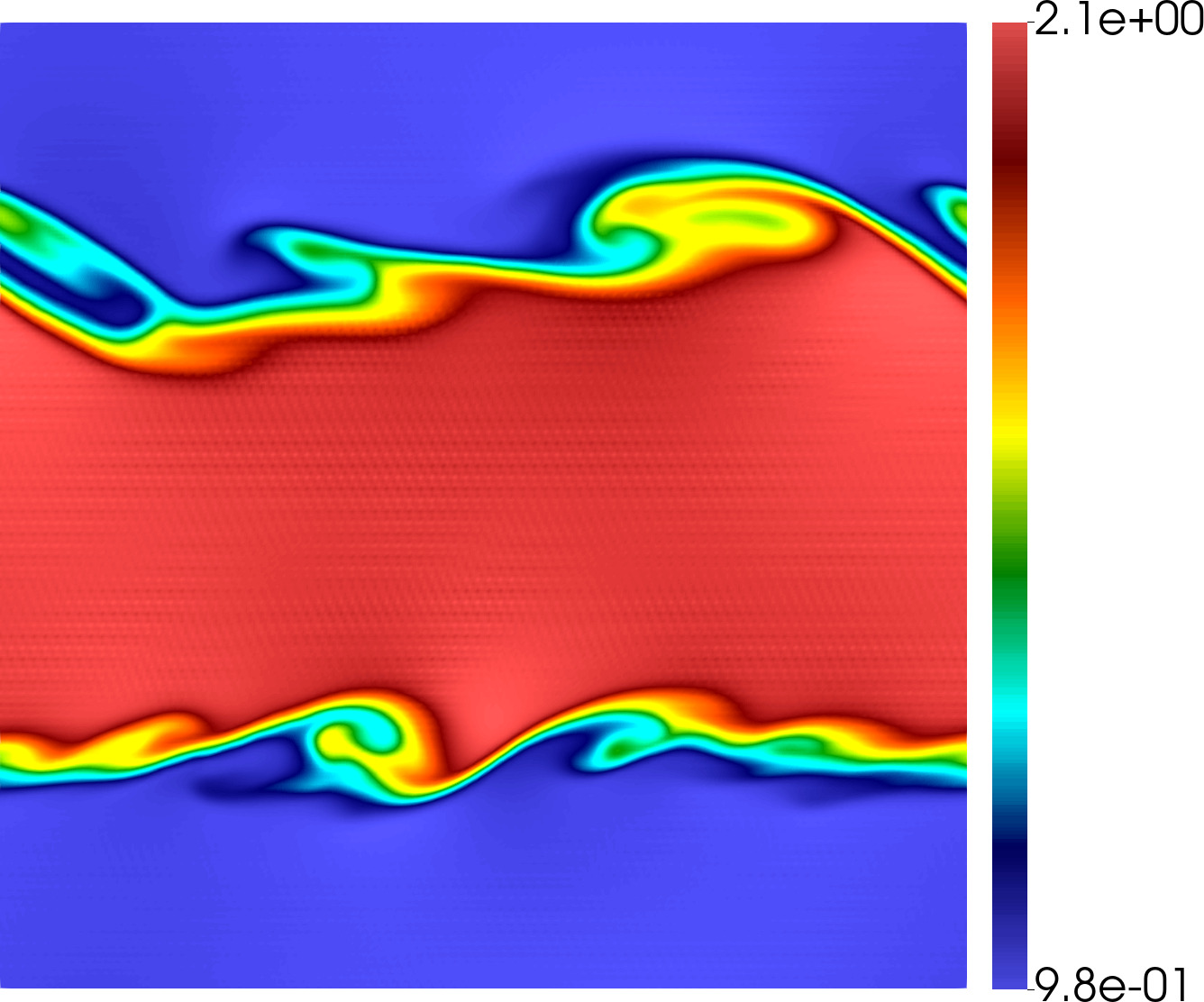}
  \end{subfigure}
  \begin{subfigure}{0.32\textwidth}
    \caption*{$t=3$}
    \includegraphics[width=\textwidth]{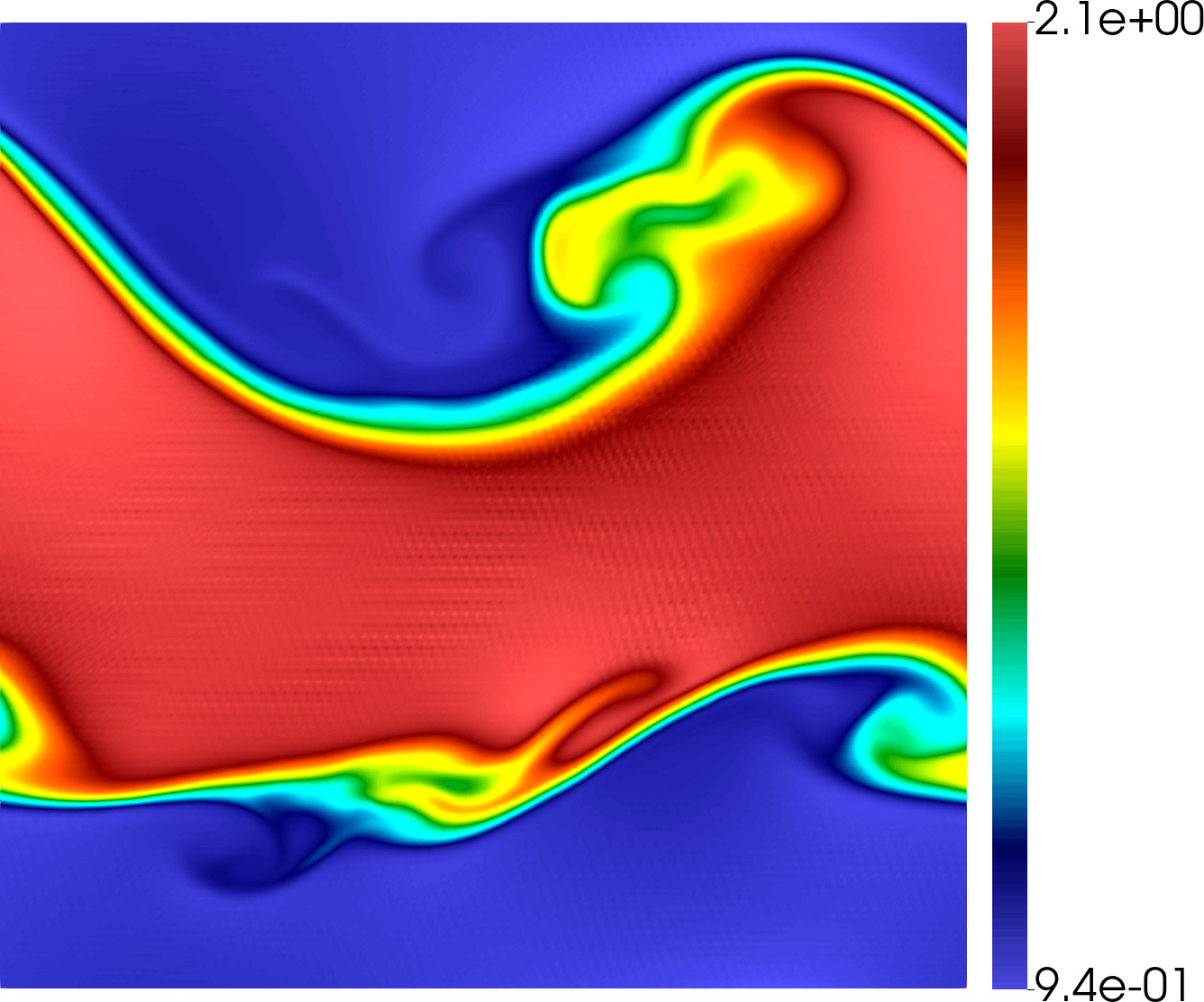}
  \end{subfigure}
  \\
  \vspace{0.05in}
  \begin{subfigure}{0.32\textwidth}
    \caption*{$t=4$}
    \includegraphics[width=\textwidth]{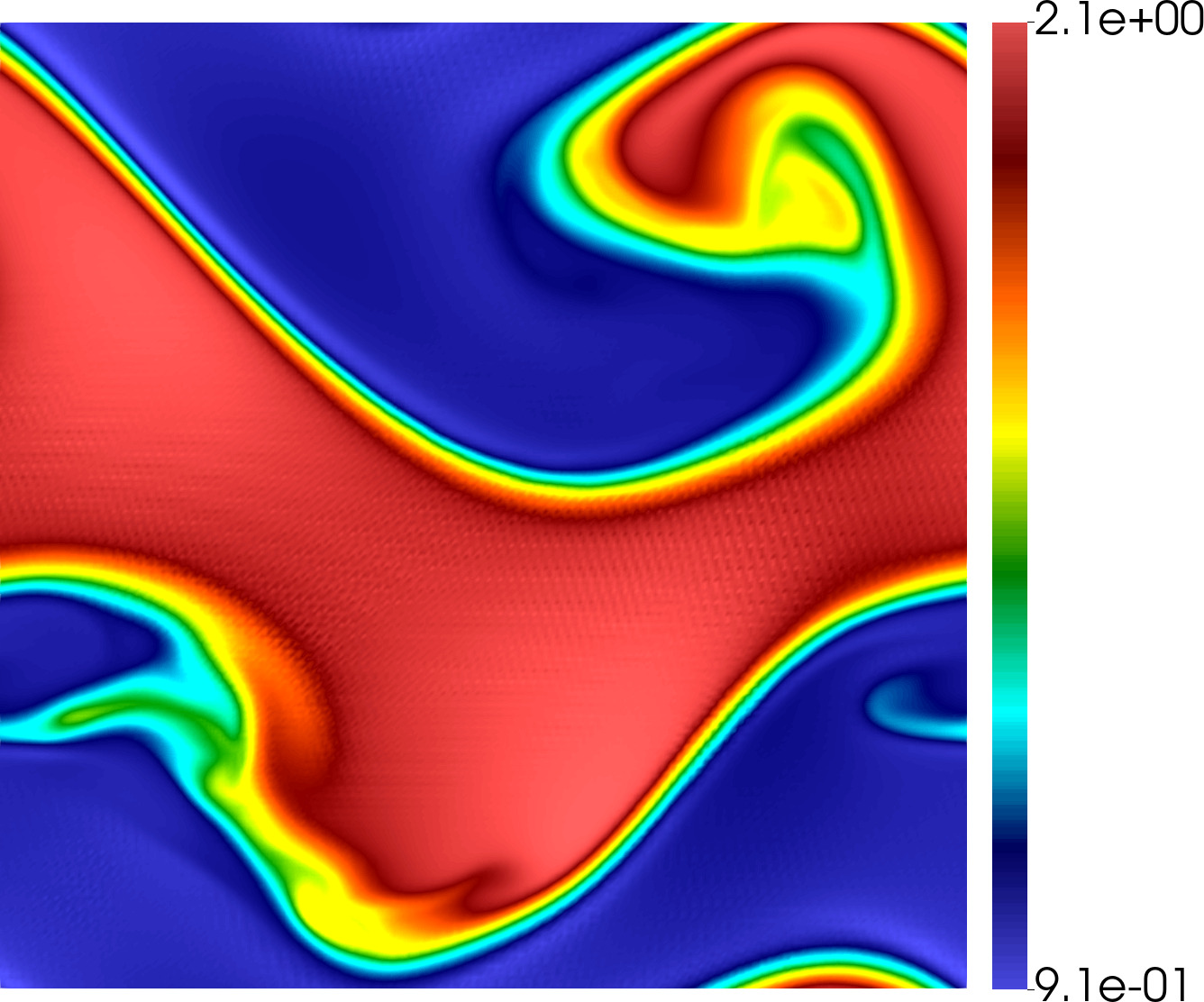}
  \end{subfigure}
  \begin{subfigure}{0.32\textwidth}
    \caption*{$t=5$}
    \includegraphics[width=\textwidth]{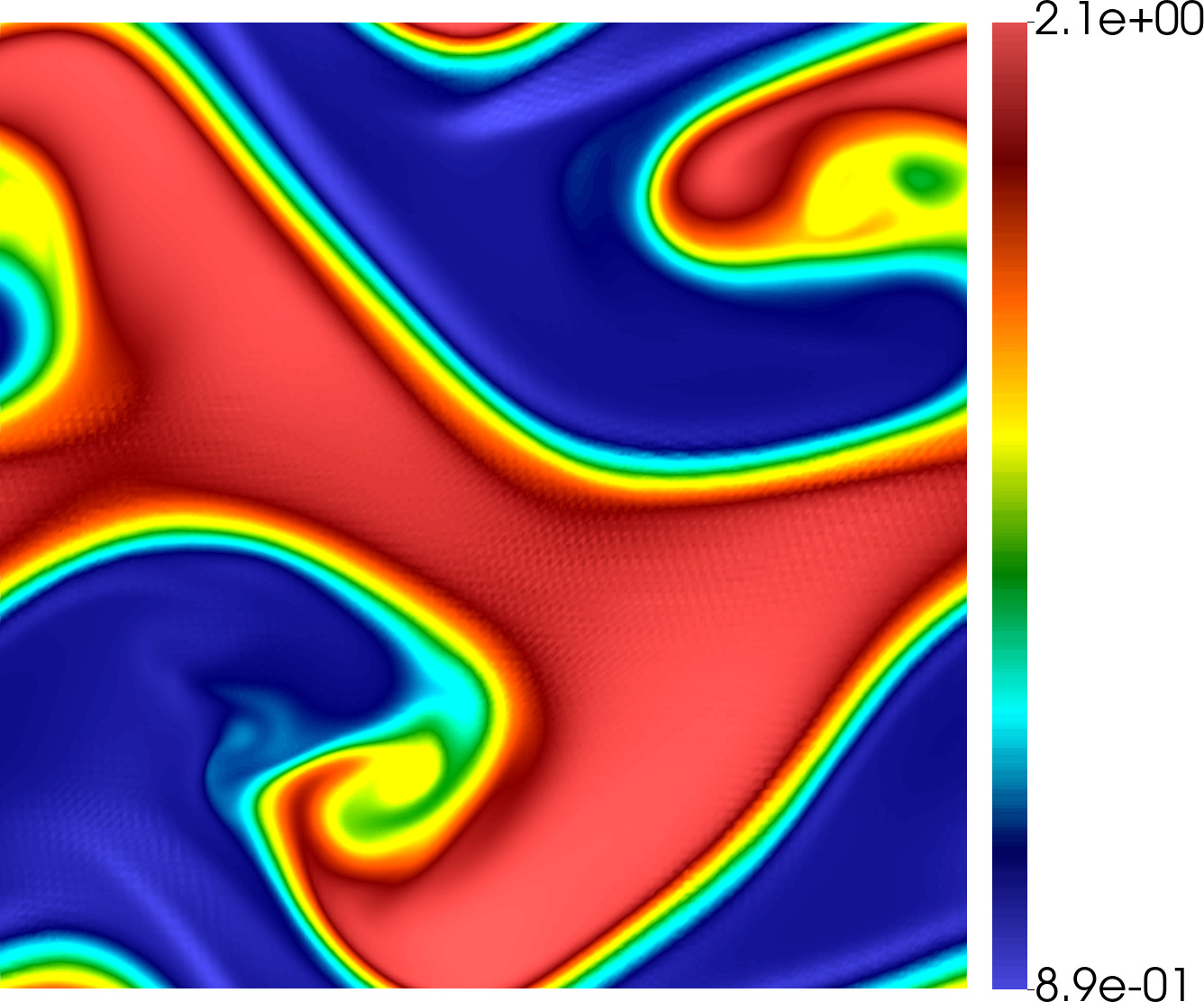}
  \end{subfigure}
  \begin{subfigure}{0.32\textwidth}
    \caption*{$t=6$}
    \includegraphics[width=\textwidth]{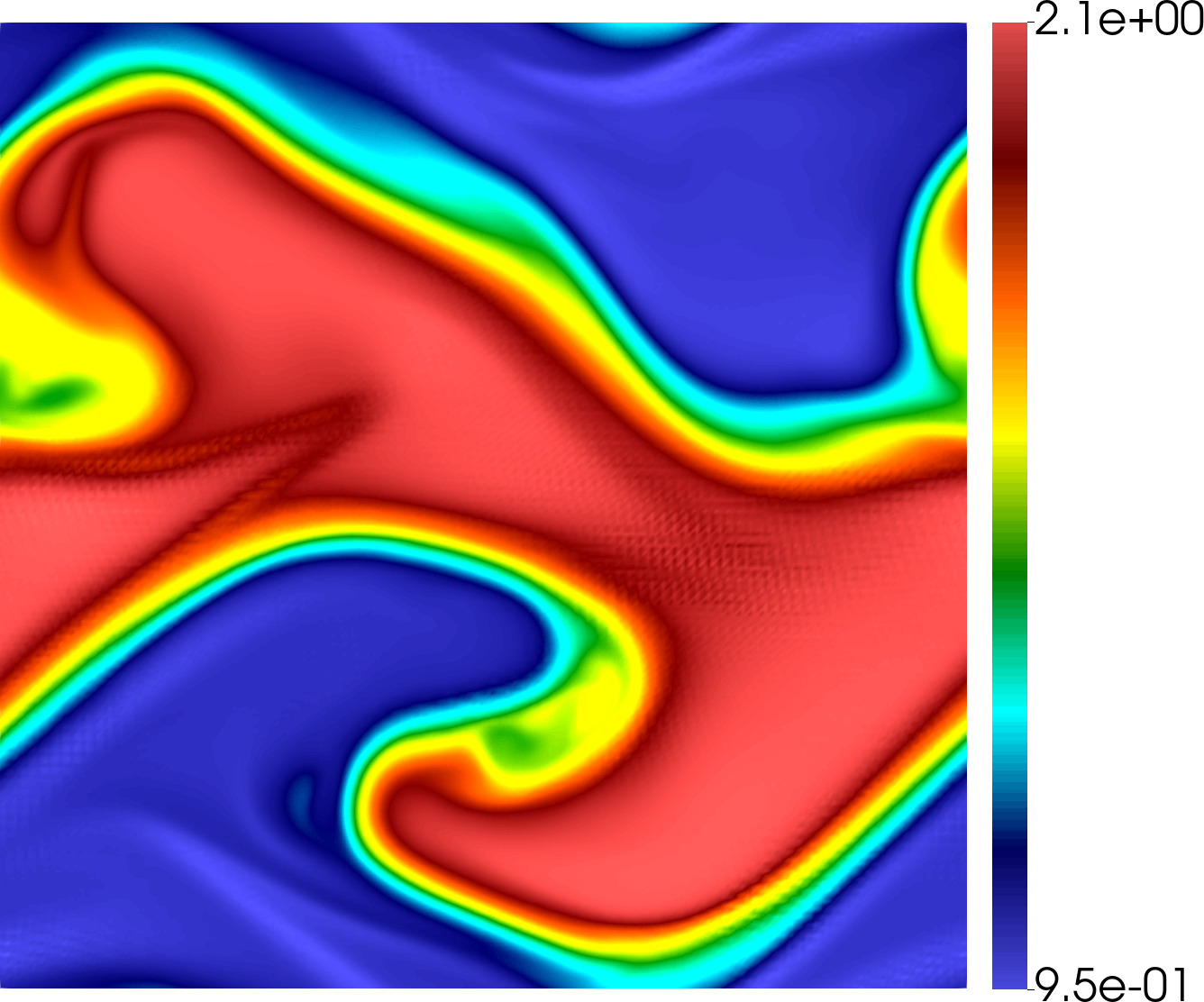}
  \end{subfigure}
  \\
  \vspace{0.3in}
  $510 \times 510$ $\polP_3$ nodes  
  \\
    \begin{subfigure}{0.32\textwidth}
    \caption*{$t=1$}
    \includegraphics[width=\textwidth]{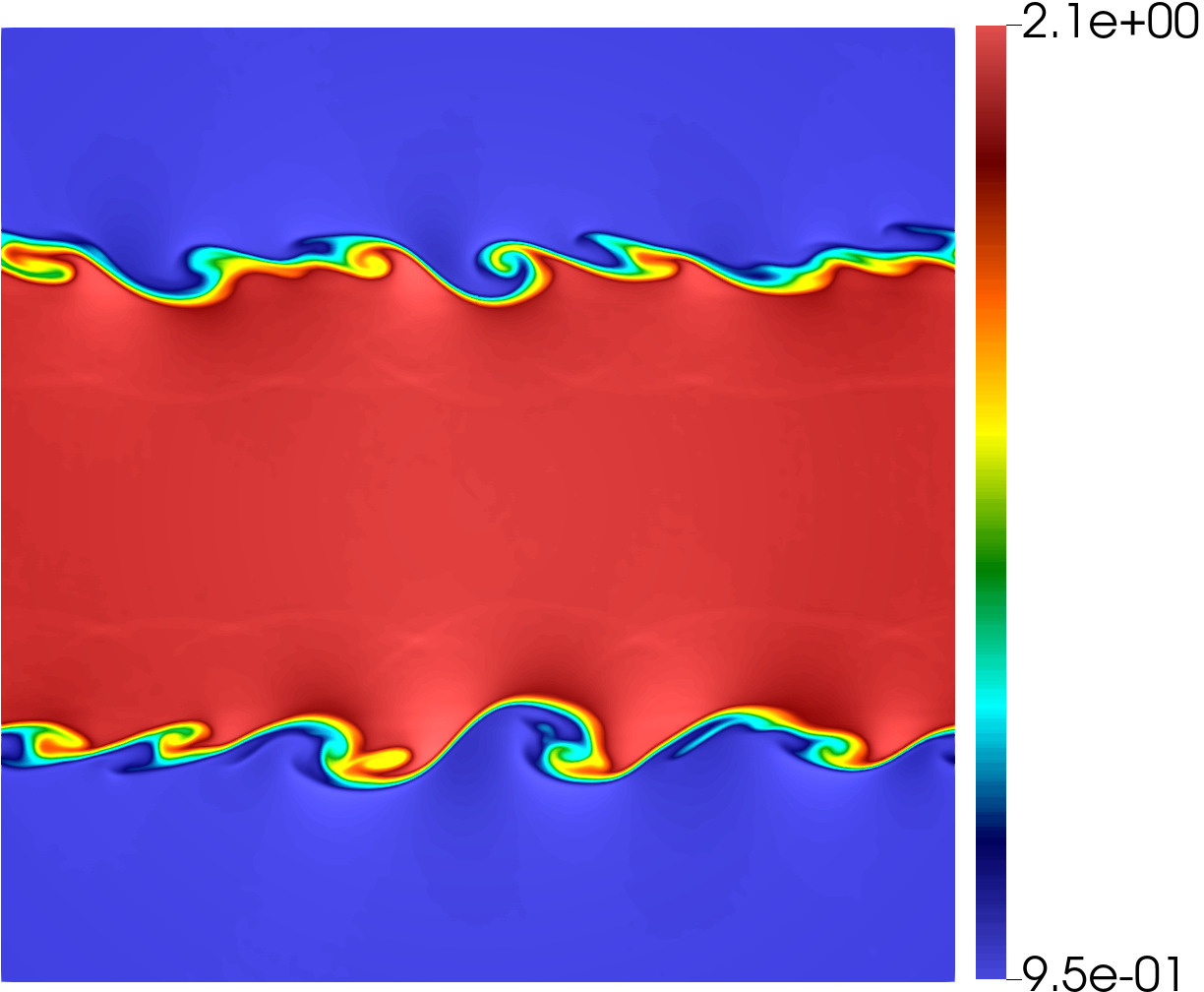}
  \end{subfigure}
  \begin{subfigure}{0.32\textwidth}
    \caption*{$t=2$}
    \includegraphics[width=\textwidth]{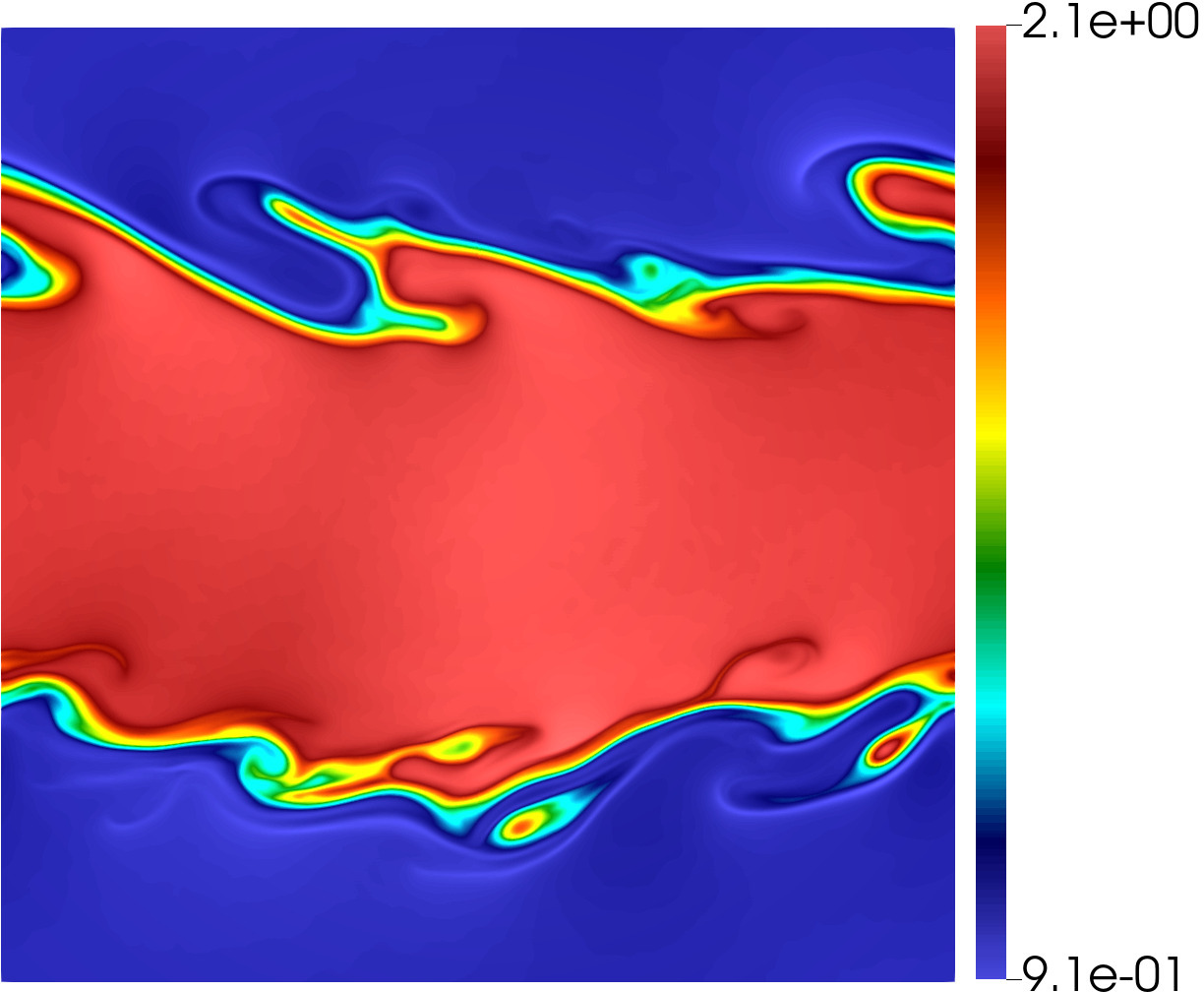}
  \end{subfigure}
  \begin{subfigure}{0.32\textwidth}
    \caption*{$t=3$}
    \includegraphics[width=\textwidth]{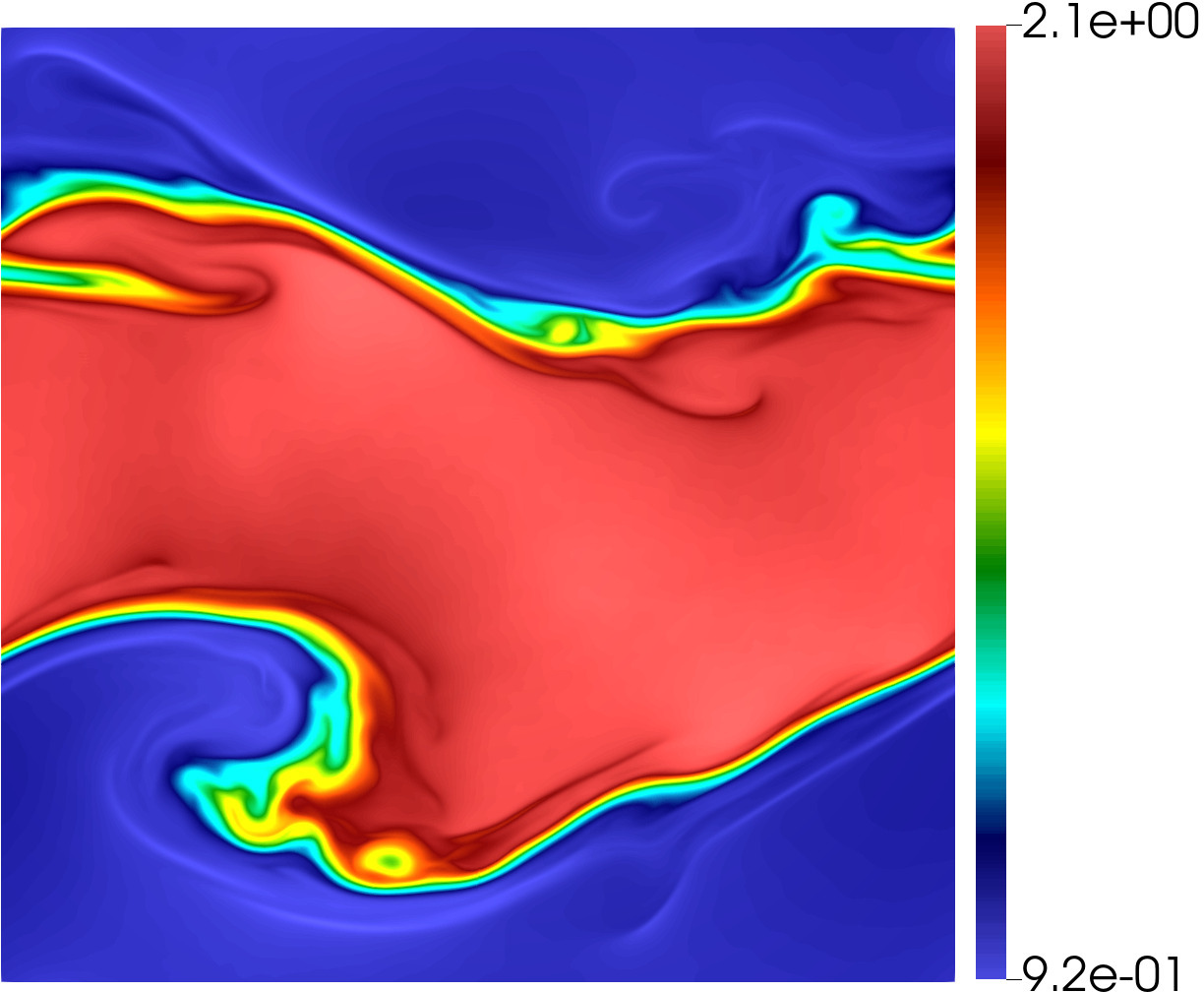}
  \end{subfigure}
  \\
  \vspace{0.05in}
  \begin{subfigure}{0.32\textwidth}
    \caption*{$t=4$}
    \includegraphics[width=\textwidth]{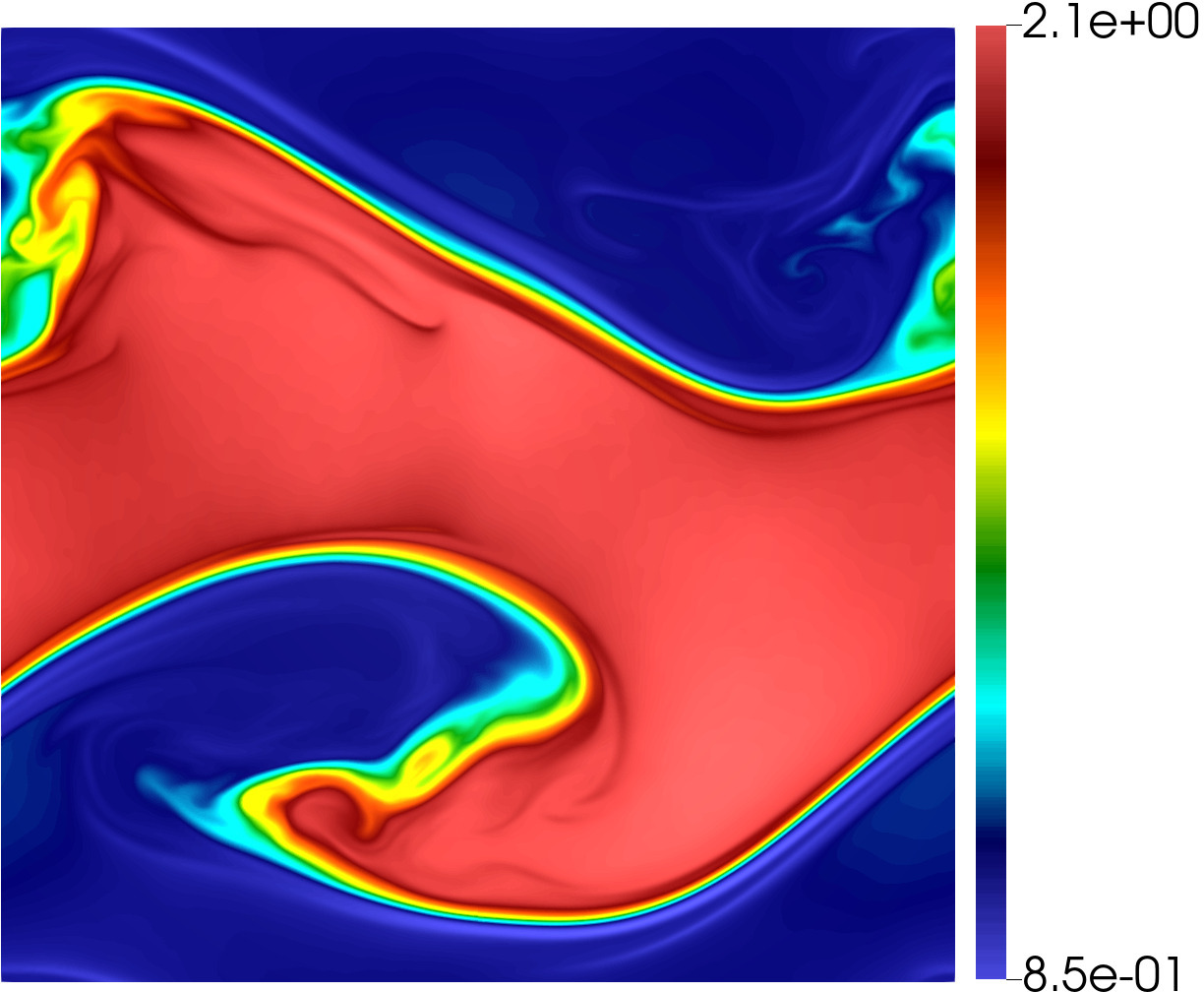}
  \end{subfigure}
  \begin{subfigure}{0.32\textwidth}
    \caption*{$t=5$}
    \includegraphics[width=\textwidth]{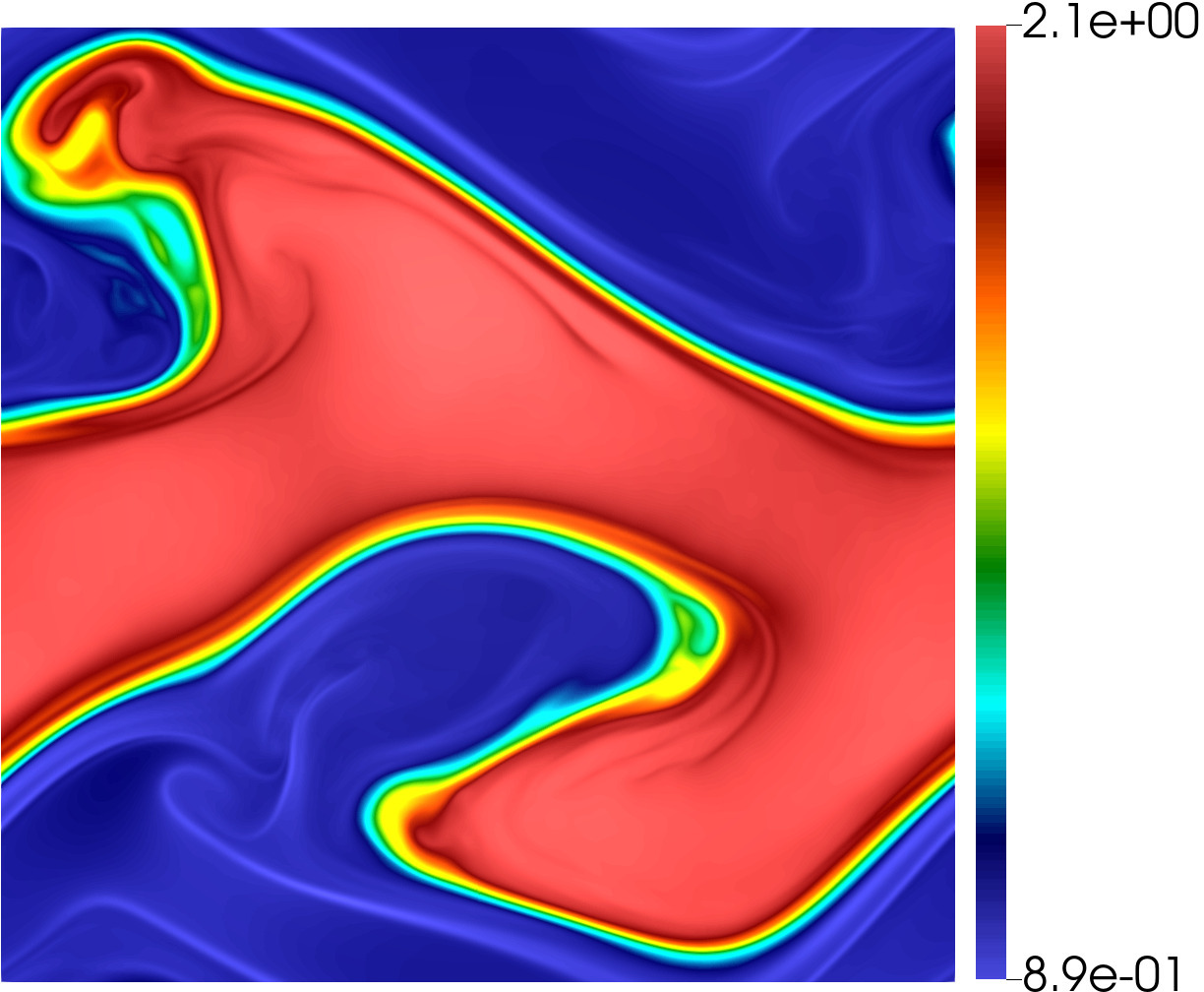}
  \end{subfigure}
  \begin{subfigure}{0.32\textwidth}
    \caption*{$t=6$}
    \includegraphics[width=\textwidth]{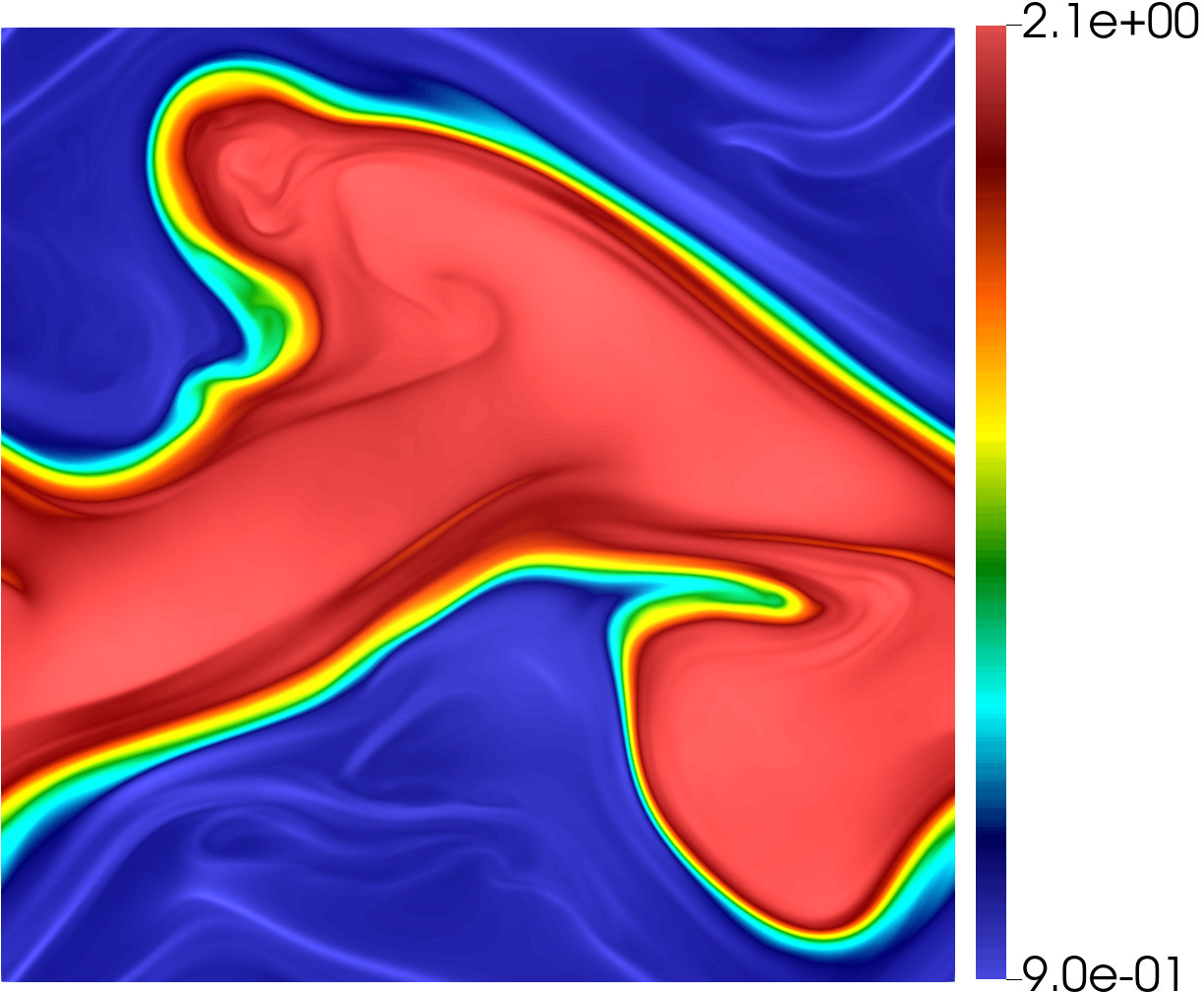}
  \end{subfigure}
  \caption{\red{Kelvin-Helholtz instability. $\polP_3$ solution of fully Magnetohydrodynamic regime, $b_x=0.2$. The density profile is plotted at different time levels for two mesh resolutions: $85 \times 85$ and  $170 \times 170$ vertices.}}
  \label{fig:KH:b02}
\end{figure}

}

%%%%%%%%%%%%%%%%%%%%%%%%%%%%%%%%%%%%%%%%%%%% 
\subsection{MHD Blast, \citep{Balsara_et_al_1999}}
The MHD Blast problem is a challenging benchmark because the solvers can easily crash due to negativity of pressure. The domain is a square $\Omega = [-0.5, 0.5] \times [-0.5, 0.5]$. The ambient solution is
\[
  (\rho_0, \bu_0, p_0, \bB_0) = \left(1, (0, 0), 0.1, \l(\frac{100}{\sqrt{4\pi}}, 0\r)\right).
\]
For $\bx\in B((0,0)^\top,R)$ the circle centered at origin with radius $R = 1$, the solution is initialized with a sharp jump in the pressure $p=1000$ which is 10,000 times bigger than the ambient pressure. The gas constant is $\gamma=1.4$.
\red{
The periodic boundary condition is used on all boundaries. We use fourth-order classical Runge-Kutta method with $\textrm{CFL}=0.2$ in time.

We solve the problem using two polynomial spaces: $\polP_1$ on 201,117 nodes and $\polP_3$ on 201,627 nodes, see Figures~\ref{fig:Blast:P1} and \ref{fig:Blast:P3}. Since the initial condition is interpolated into finite element space, a significant jump in pressure in the initial data leads to overshoots and undershoots on the $\polP_3$ solution at $t=0$. These overshoots and undershoots propagate over time, and small-scale oscillations can be observed near the expansion waves, see at the density profile of Figure~\ref{fig:Blast:P3}. For $\polP_1$, we observe less noise in the expansion region since the interpolated initial condition does not produce any overshoots or undershoots. However, the shocks are captured accurately and we observe more structures in the $\polP_3$ solution compared to the corresponding $\polP_1$ solution, especially at the center of the explosion. We plot $\bB_{h,x}$ along the lines $x=0$ and $y=0$ in Figure~\ref{fig:Blast:slices}. One can see that the $\polP_3$ solution is more accurate for the same degrees of freedom. 

Another important observation in this test is the magnitude of the artificial viscosity for both polynomial spaces. Since the degrees of freedom are the same, the values of the viscosity coefficients at the impact sections are relatively similar. Since the $\polP_3$ solution exhibits small-scale oscillations, a small viscosity is added in the expansion region. Note that the residual normalization discussed in section~\ref{sec:rv} is robust when adding enough viscosity to solve this difficult problem for both polynomial spaces.

\begin{figure}[h!]
  \centering
  \begin{subfigure}{0.4\textwidth}
    \centering
    \includegraphics[width=\textwidth]{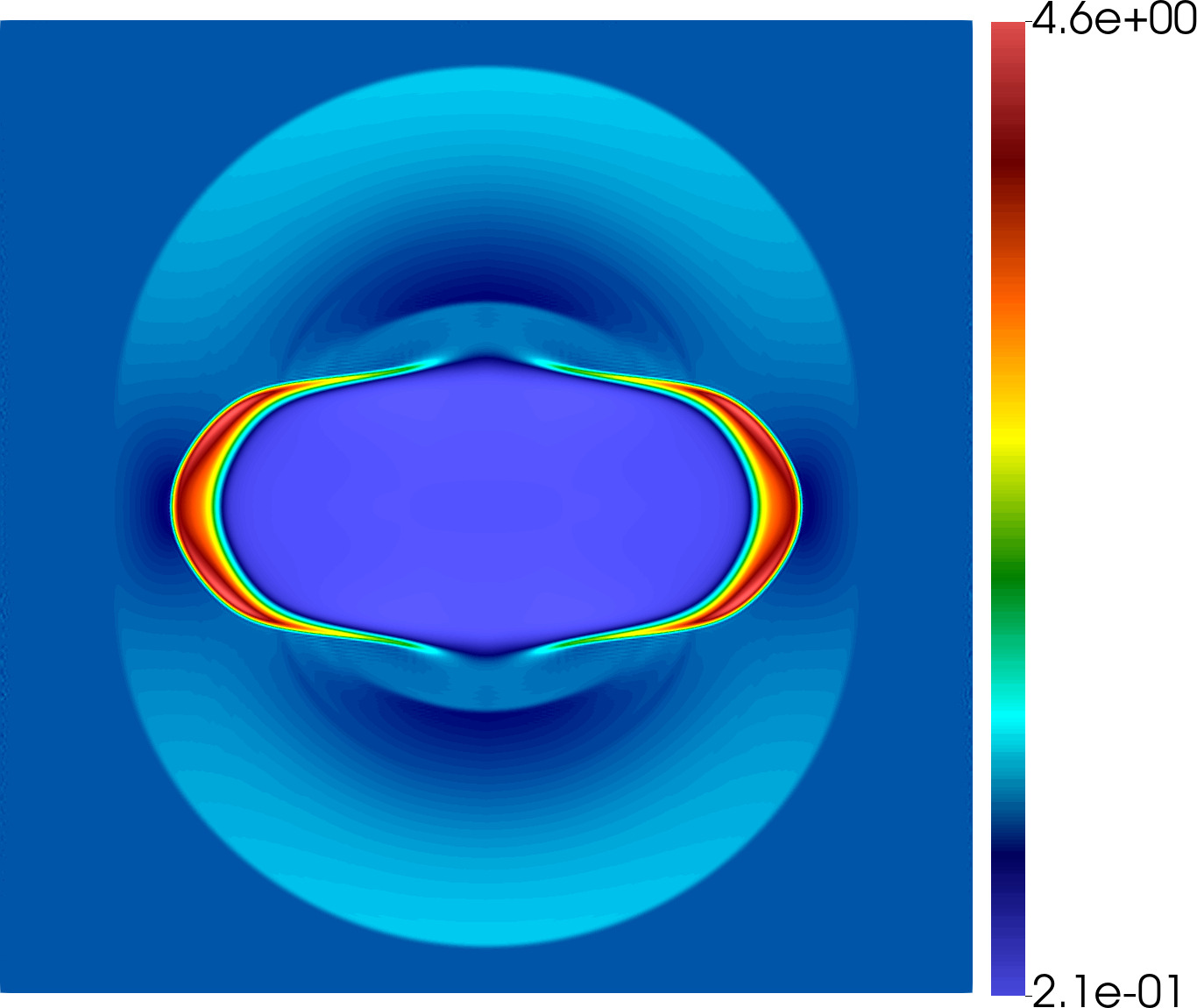}
    \caption*{Density $\rho_h$}
  \end{subfigure}
  \hspace{0.2in}
  \begin{subfigure}{0.4\textwidth}
    \centering
    \includegraphics[width=\textwidth]{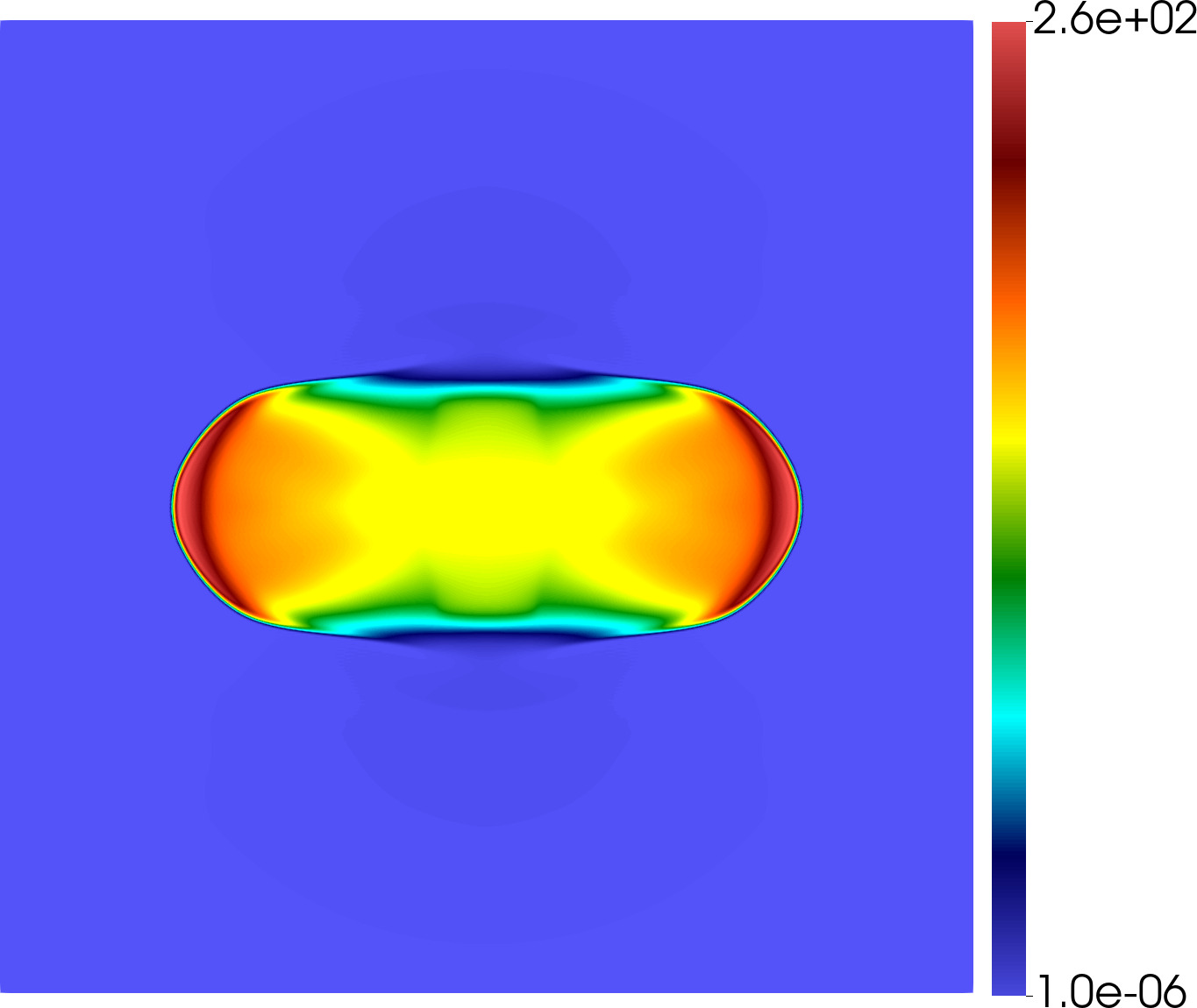}
    \caption*{Hydrodynamic pressure $P_h$}
  \end{subfigure}
  \\
  \vspace{0.1in}
  \begin{subfigure}{0.4\textwidth}
    \centering
    \includegraphics[width=\textwidth]{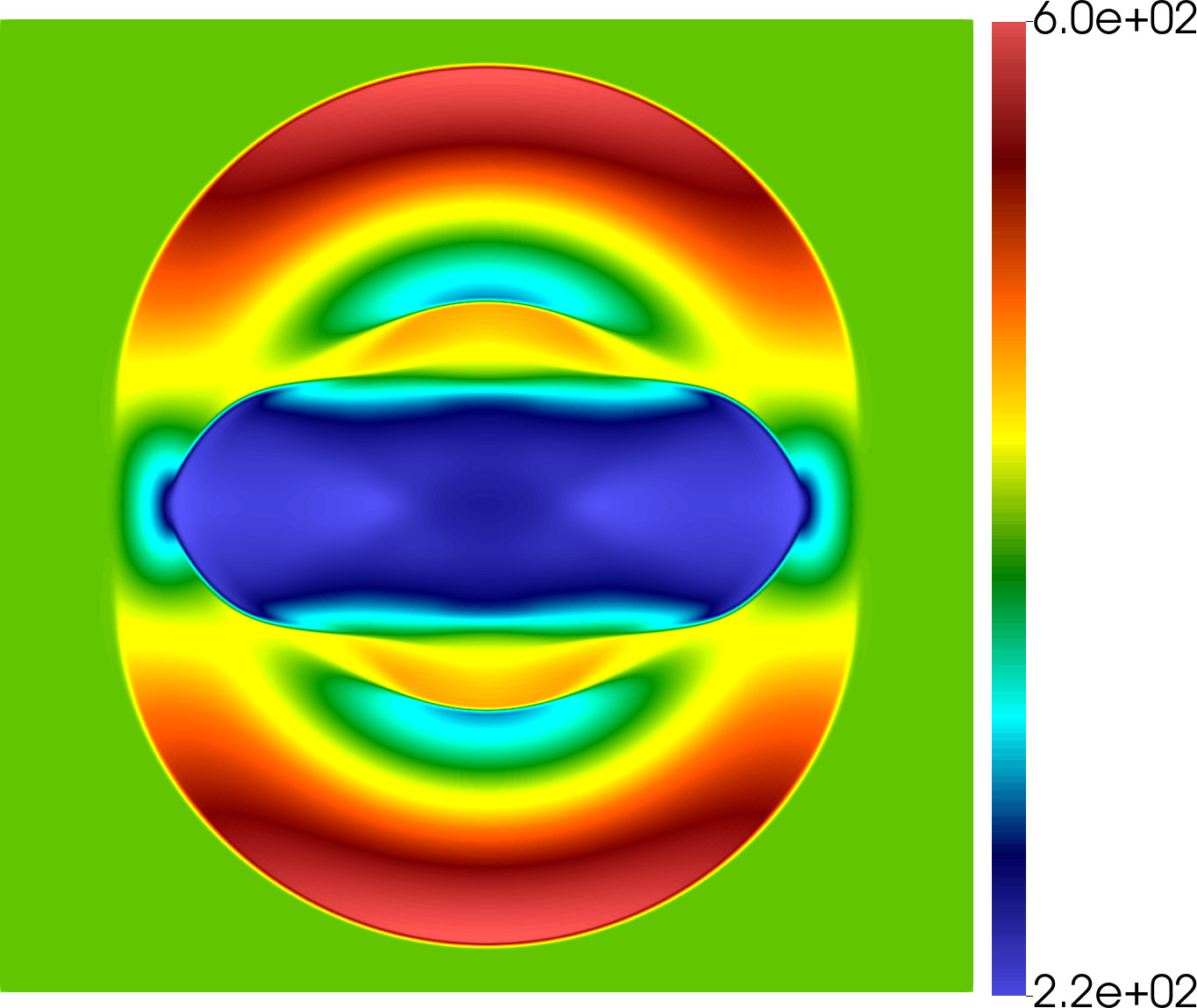}
    \caption*{Magnetic pressure $|\bB_h|^2/2$}
  \end{subfigure}
  \hspace{0.2in}
  \begin{subfigure}{0.4\textwidth}
    \centering
    \includegraphics[width=\textwidth]{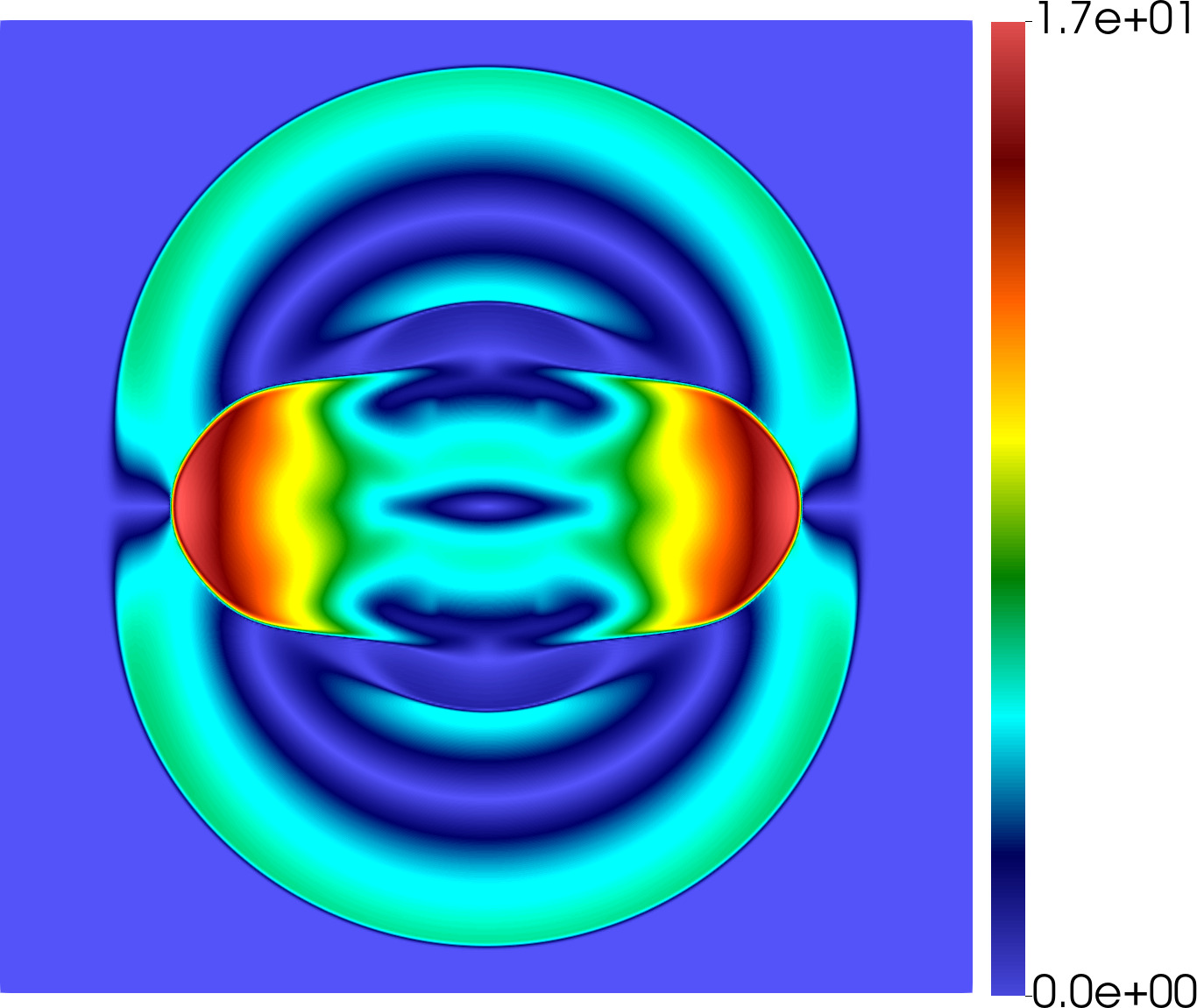}
    \caption*{Velocity magnitude $|\bu_h|$}
  \end{subfigure}
  \caption{Solution to Blast problem at time $t=0.01$ on 240359 $\polP_1$ nodes.}
  \label{fig:Blast:P1}
\end{figure}

\begin{figure}[h!]
  \centering
  \begin{subfigure}{0.4\textwidth}
    \centering
    \caption*{$\bB_{h,x}$ on $x=0$}
    \includegraphics[width=\textwidth]{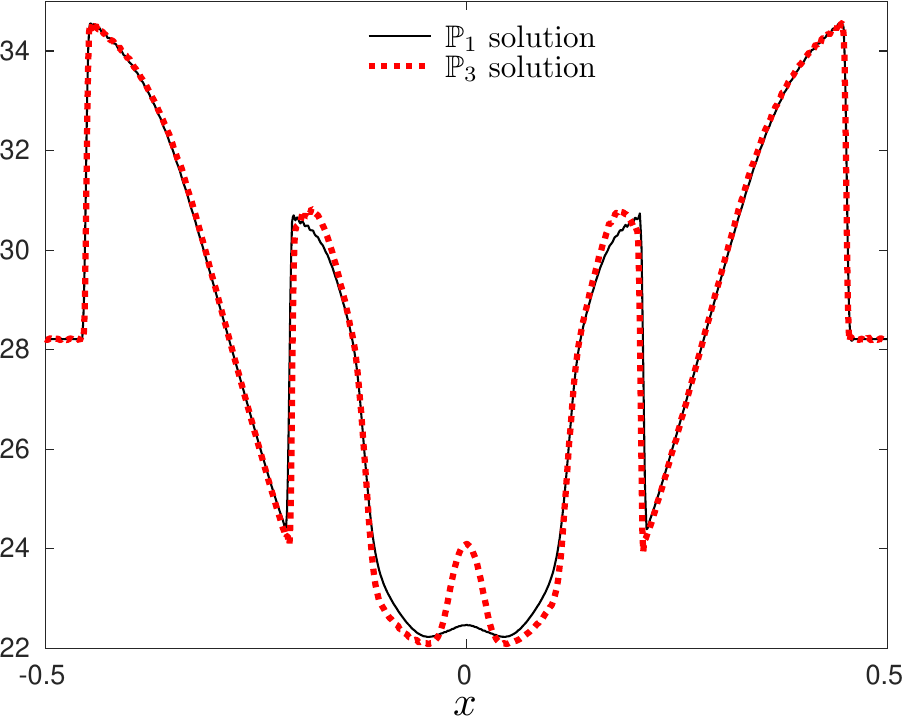}
  \end{subfigure}
  \hspace{0.1in}
  \begin{subfigure}{0.4\textwidth}
    \centering
    \caption*{$\bB_{h,x}$ on $y=0$}
    \includegraphics[width=\textwidth]{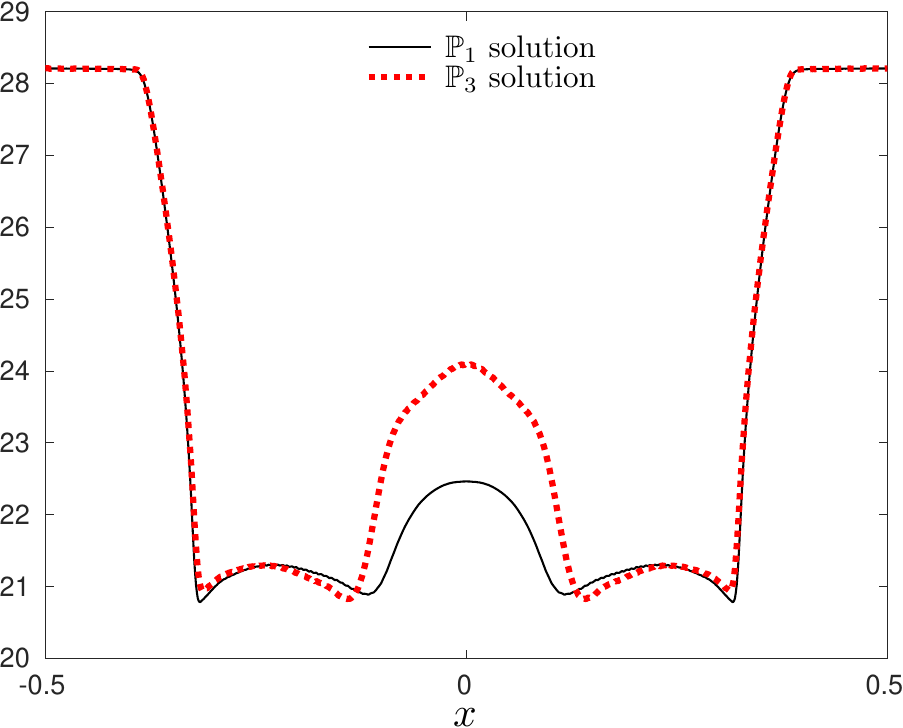}
  \end{subfigure}
  \caption{Slices of the numerical solution to the Blast problem at the final time $t=0.01$.}
  \label{fig:Blast:slices}
\end{figure}

\begin{figure}[h!]
  \centering
  \begin{subfigure}{0.4\textwidth}
    \centering
    \includegraphics[width=\textwidth]{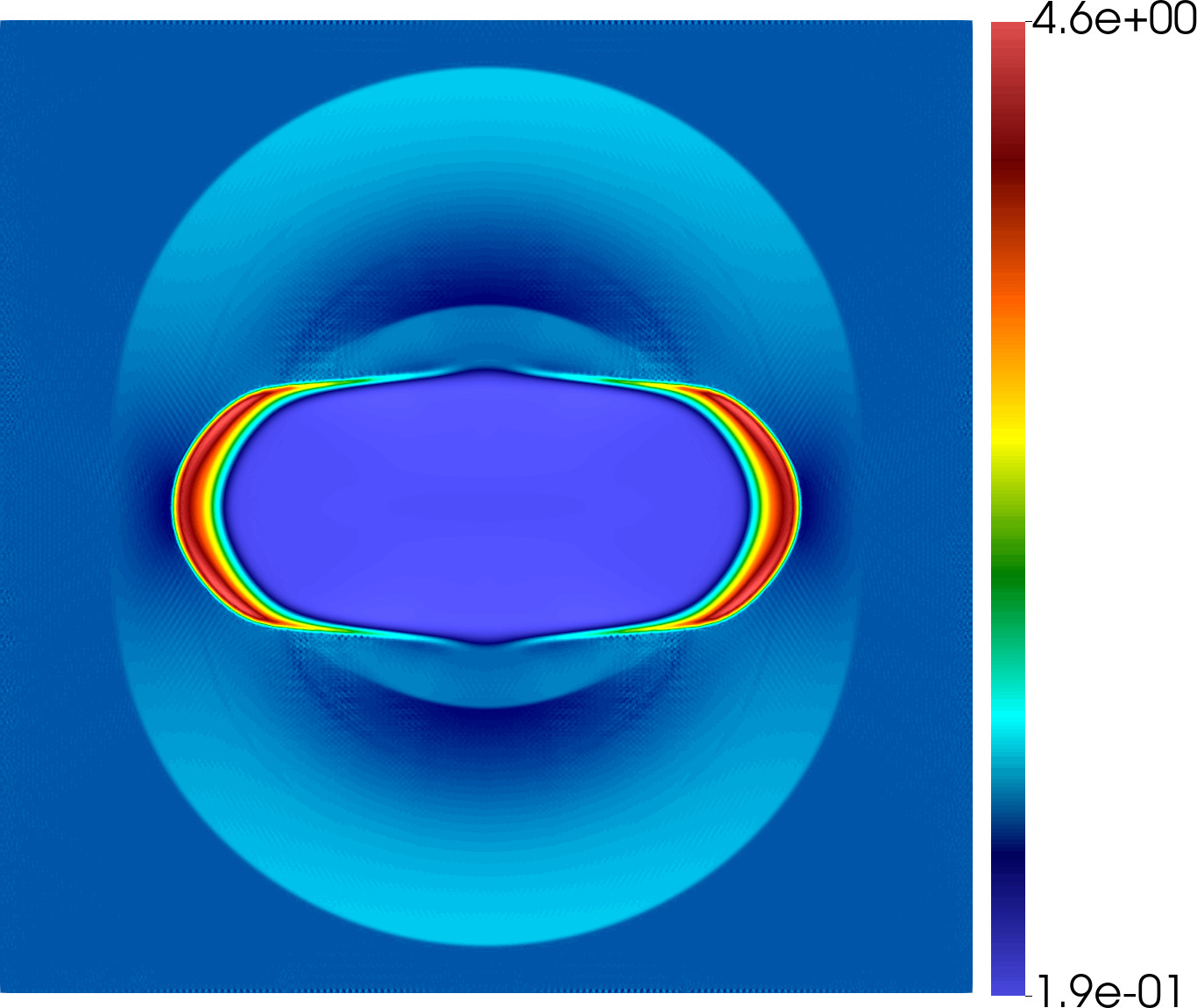}
    \caption*{Density $\rho_h$}
  \end{subfigure}
  \hspace{0.2in}
%  \hfill
  \begin{subfigure}{0.4\textwidth}
    \centering
    \includegraphics[width=\textwidth]{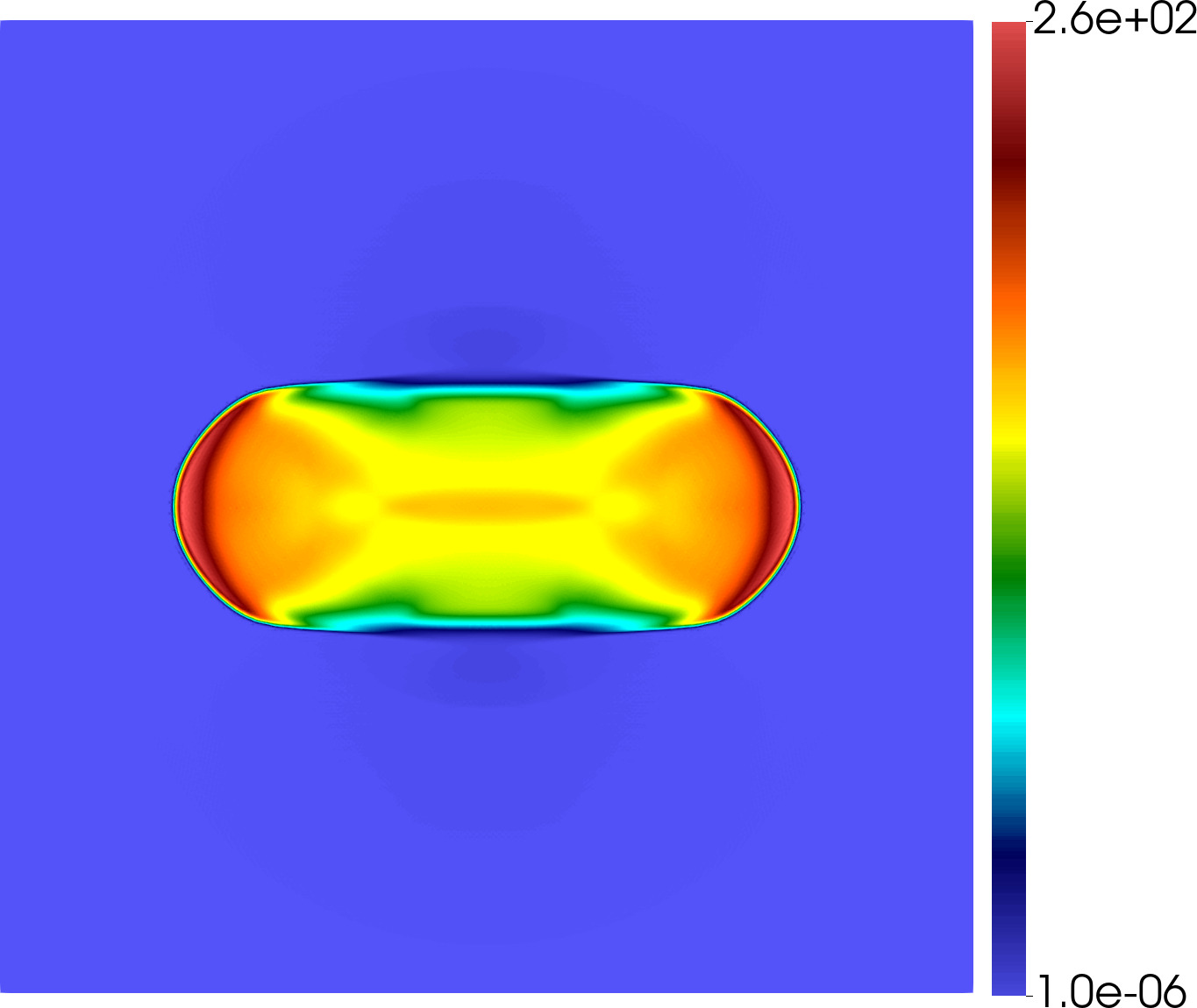}
    \caption*{Hydrodynamic pressure $P_h$}
  \end{subfigure}
  \\
  \vspace{0.1in}
%  \hfill
  \begin{subfigure}{0.4\textwidth}
    \centering
    \includegraphics[width=\textwidth]{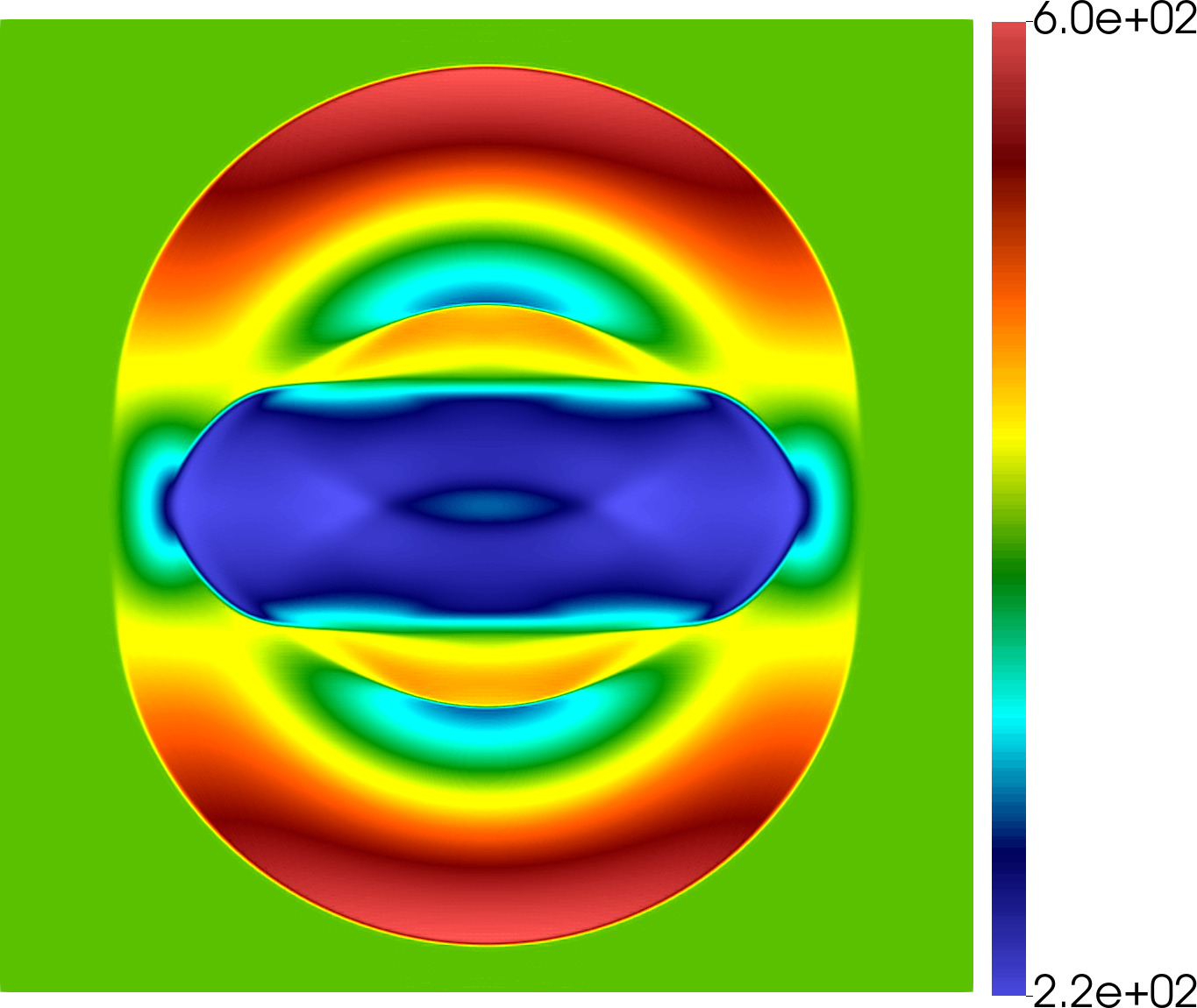}
    \caption*{Magnetic pressure $|\bB_h|^2/2$}
  \end{subfigure}
  \hspace{0.2in}
%  \hfill
  \begin{subfigure}{0.4\textwidth}
    \centering
    \includegraphics[width=\textwidth]{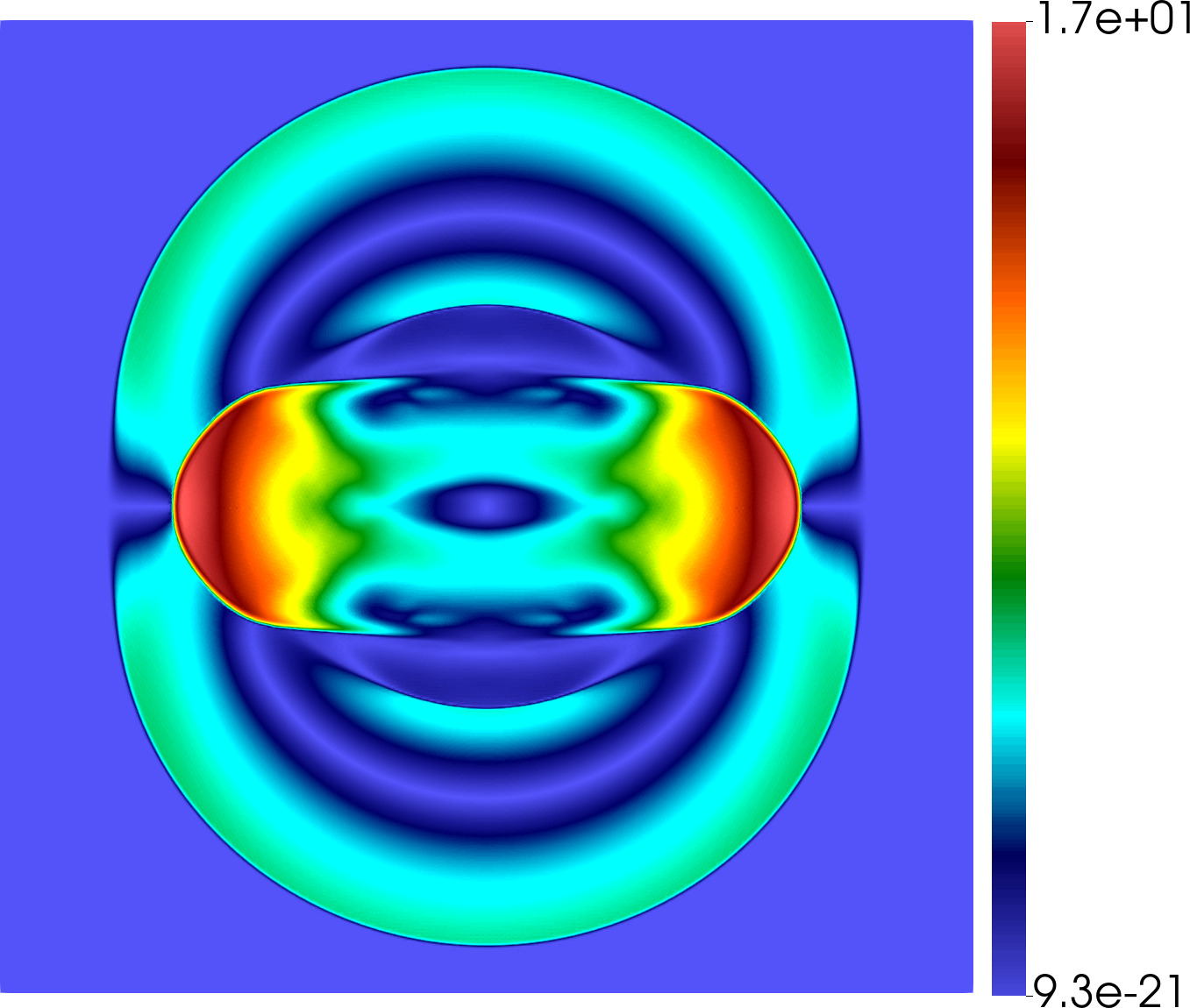}
    \caption*{Velocity magnitude $|\bu_h|$}
  \end{subfigure}
%  \hfill
  \caption{Solution to Blast problem at time $t=0.01$ on 242077 $\polP_3$ nodes.}
  \label{fig:Blast:P3}
\end{figure}

\begin{figure}[h!]
  \centering
  \begin{subfigure}{0.4\textwidth}
    \centering
    \includegraphics[width=\textwidth]{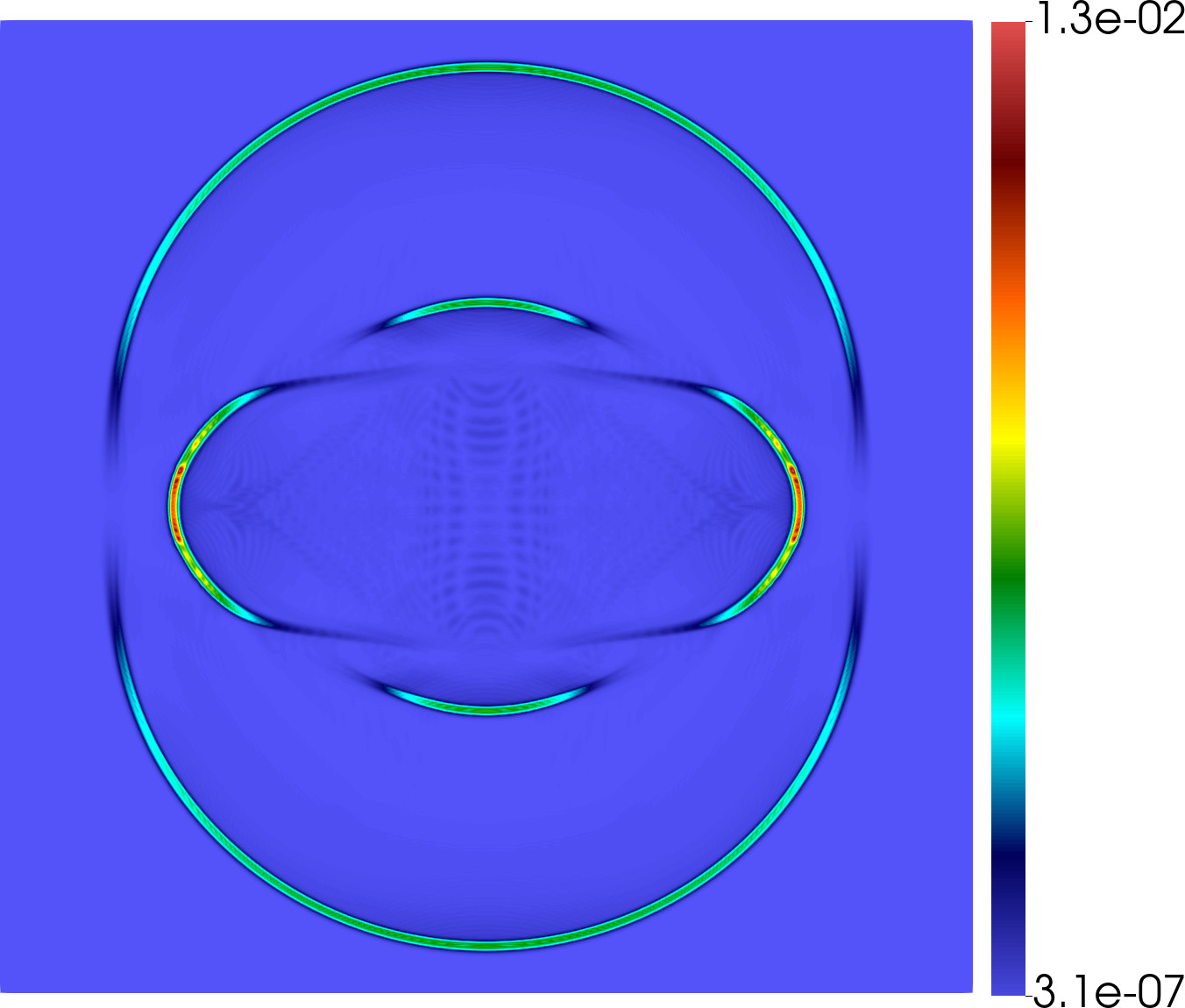}
    \caption*{$\e_h$ on 240359 $\polP_1$ nodes}
  \end{subfigure}
  \hspace{0.1in}
%  \hfill
  \begin{subfigure}{0.4\textwidth}
    \centering
    \includegraphics[width=\textwidth]{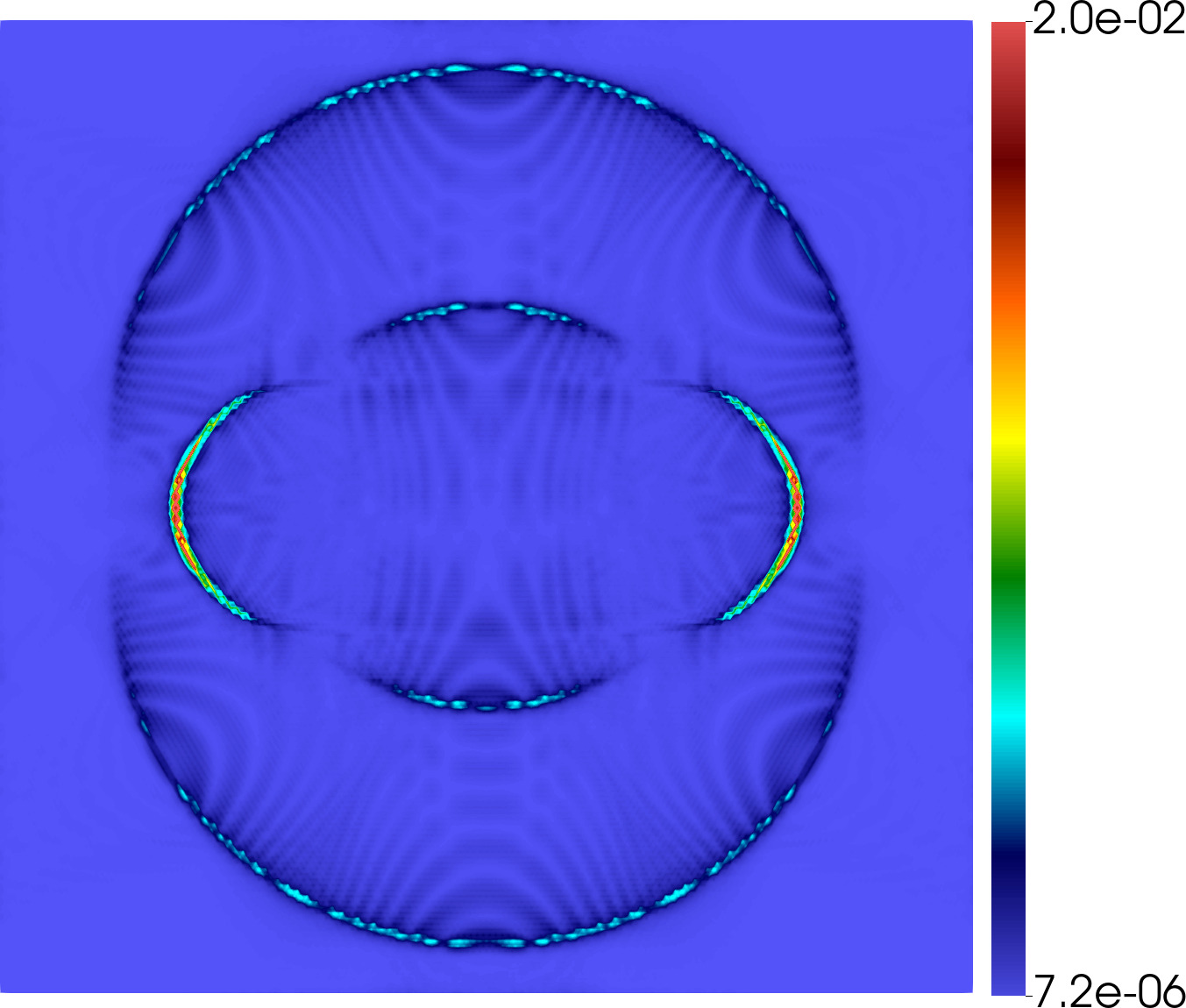}
    \caption*{$\e_h$ on 242077 $\polP_3$ nodes}
  \end{subfigure}
  \hfill
  \caption{Blast problem. The viscosity coefficients for $\polP_1$ and $\polP_3$ at time $t=0.01$.}
  \label{}
\end{figure}

}

\red{
%%%%%%%%%%%%%%%%%%%%%%%%%%%%%%%%%%%%%%%%%%%% 
\subsection{Supersonic plasma flow past 2D circular cylinder}
In this example, we will calculate supersonic plasma flow around an obstacle as a final example. The domain is $\Omega=[0,10]\times[-5,-5]$ and has a circular cylinder with radius $r=0.1$ and center $(x_c,y_c)=(1.2,0)$. The gas constant is chosen to be $\gamma=1.4$ and the initial condition is
\[
    (\rho_0, \bu_0, p_0, \bB_0) = \left(\gamma, (\gamma, 0), 1, (0, b_y) \right).
\]
Here $b_y$ is the magnetic field in the direction $y$, which is yet to be determined. The boundary condition is specified as follows: for density, velocity, and pressure we set the Dirichlet condition on the left, \ie $x=0$  with the same values as the initial condition; we set the slip boundary condition on the boundary on above and below boundaries $y=-0.5$ and $y=0.5$, and on the cylinder; on the right boundary $x=10$ the characteristic boundary condition is used. The magnetic field is set to $\bB_h=(0,b_y)$ at all outer boundaries, and we set $\bB_h\SCAL \bn=0$ on the cylinder. 

We run the simulation until $t=5$. At this time the shock waves reach the upper and lower boundaries, and their reflections do not interfere the backstream of the flow. In addition, the flow field does not reach the right boundary. We present the results of the simulation in the domain $\Omega_0=[0.6]\times[-1.23, 1.23]$. For this domain, the computational domain has 242,450 $\polP_1$ nodes. And the  

The simulation results are collected in Figures~\ref{fig:cyl1}-\ref{fig:cyl2} for different values of $b_y$. We plot Schlieren gray-scale diagram of the density
\[
\sigma:= \exp 
\left( 
-\zeta \frac{|\GRAD \rho_h|} {\max_{\Omega}|\GRAD \rho_h|}
\right),
\]
with $\zeta=5$ in Figures~\ref{fig:cyl1} and \ref{fig:cyl2}. The first column of the Figure~\ref{fig:cyl1} corresponds to a completely hydrodynamic regime, \ie for the value $b_y=0$. The method captures bow and trailing shocks very accurately. In addition, the so-called fishtail shock that develops downstream is well resolved and its development is accurately captured at all time levels. It can also be seen that the subsonic wake behind the cylinder is very well resolved and small eddies are captured.

After introducing a non-zero magnetic field, the structure of the hydrodynamic has noticeable effected. For example, we observe the formation of plasmoids in the downstream. This can also be seen in Figure~\ref{fig:cyl:isolines}, which shows magnetic field contours along with bow and trailing shock waves. Many plasmoids form over time and move downstream with the flow.

Next we consider the values $b_y=0.2$ and $b_y=0.3$. For these values, we see only one running plasmoid, the flow field is more stable and there is no subsonic wake of the hydrodynamic case. In addition, for higher values of $b_y$ the fishtail shock is not visible, but we observe very interesting shock formations. For values $b_y=0.2$ and $b_y=0.3$ we see the formation of additional trailing shocks close to the original one. The last graph in the Figure~\ref{fig:cyl2} indicates a trailing shock-like formation. Between the two visible trailing shocks we observe a crossed shock. This interesting solution structure needs to be analyzed in more detail.
 
It should be noted that we observe that at (only) two separation points at the downstream of the cylinder boundary the pressure value becomes very small. To avoid negative pressure values, we corrected the values of pressure at these two points. In our current research work, we extend the positivity preserving algorithm presented in our recent work \citep{Dao2023} to solve this problem. 

}

\begin{figure}[h!]
  \centering
  \includegraphics[width=0.9\textwidth]{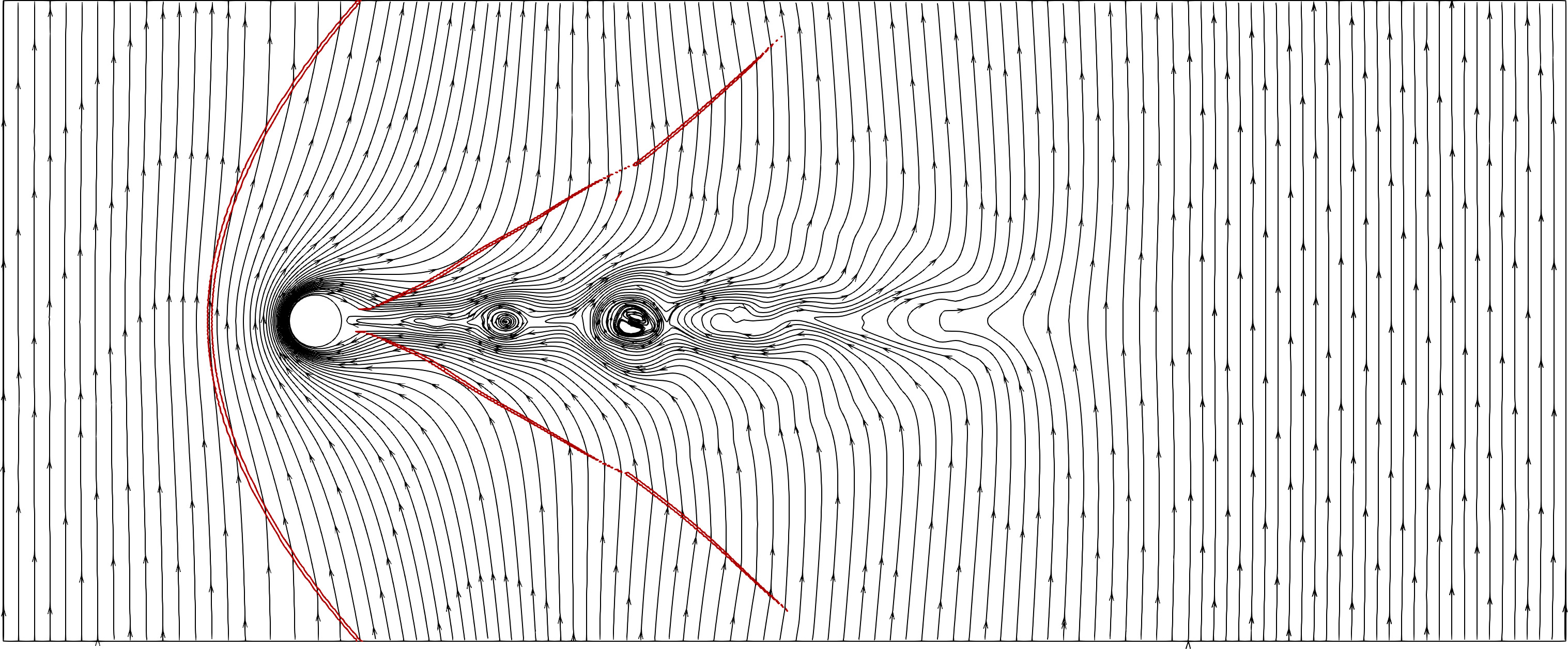}
  \caption{Supersonic plasma flow past a circular cylinder. The streamlines of the magnetic fields and bow and trailing shocks for $b_y=0.1$ and at time $t=2$.}
  \label{fig:cyl:isolines}
\end{figure}

\begin{figure}[h!]
  \centering
  \begin{subfigure}[b]{0.48\textwidth}
  \centering
  $b_y = 0.0$
  \\
  \includegraphics[width=\textwidth]{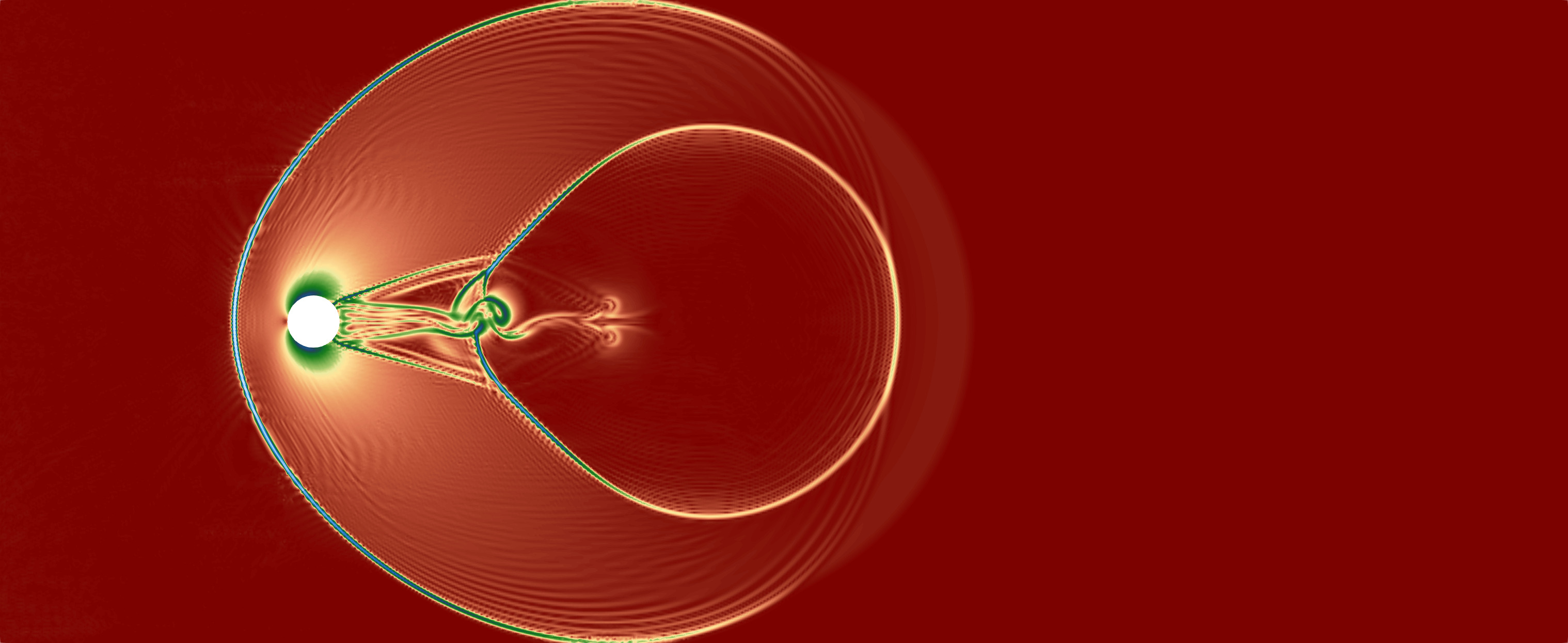}
  \includegraphics[width=\textwidth]{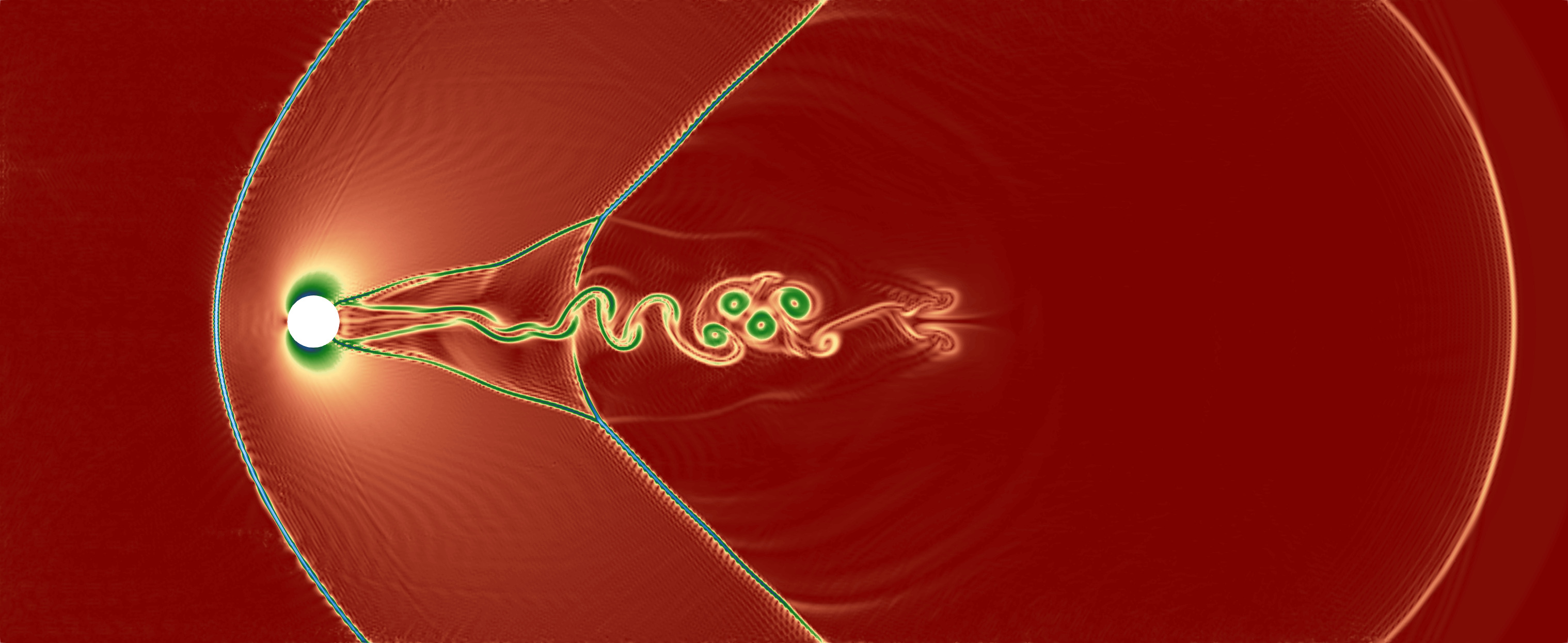}
  \includegraphics[width=\textwidth]{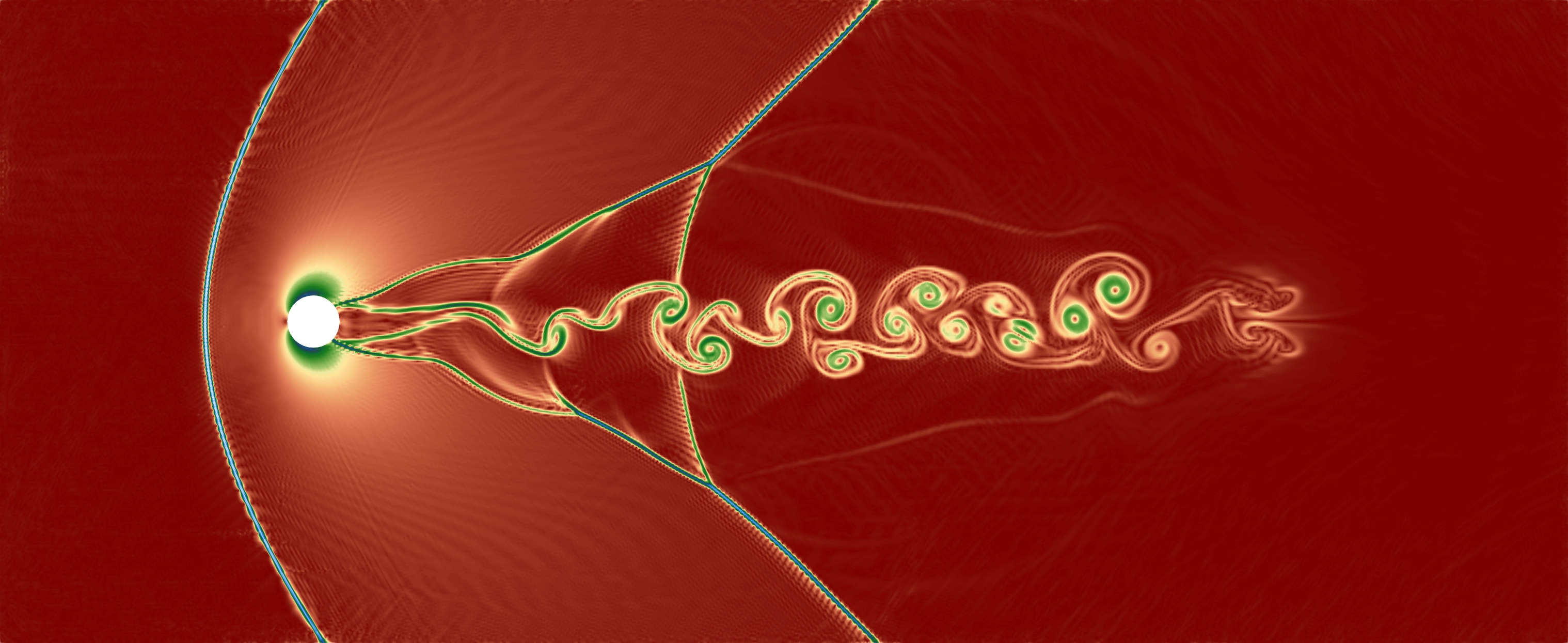}
  \includegraphics[width=\textwidth]{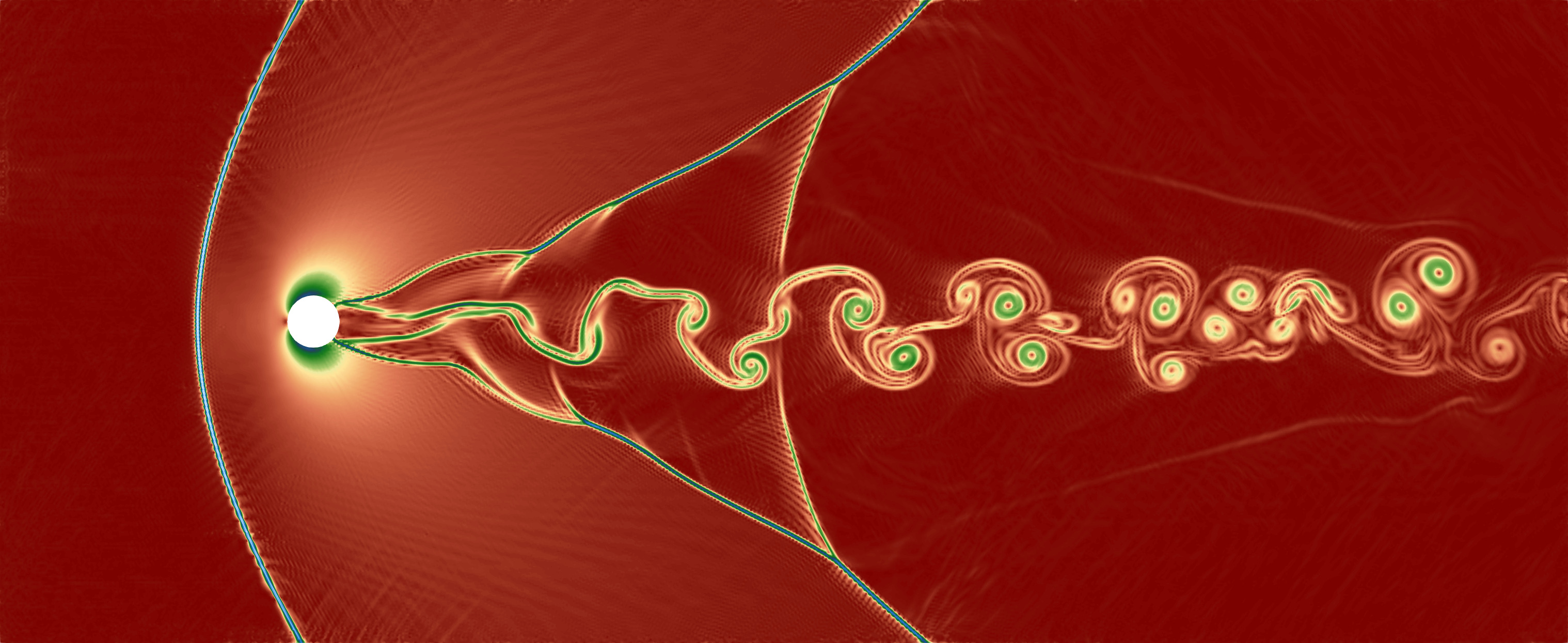}
  \includegraphics[width=\textwidth]{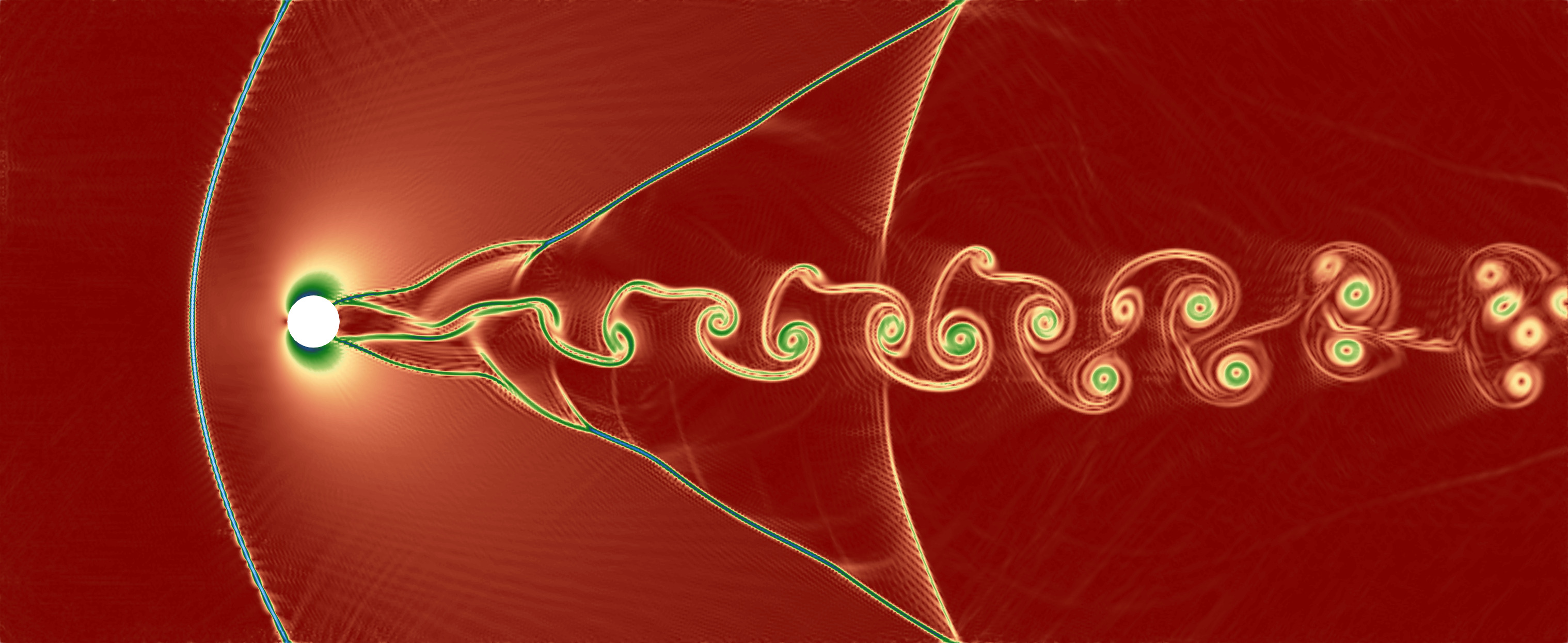}
  \end{subfigure}
  \begin{subfigure}[b]{0.48\textwidth}
  \centering
  $b_y = 0.1$
  \\
  \includegraphics[width=\textwidth]{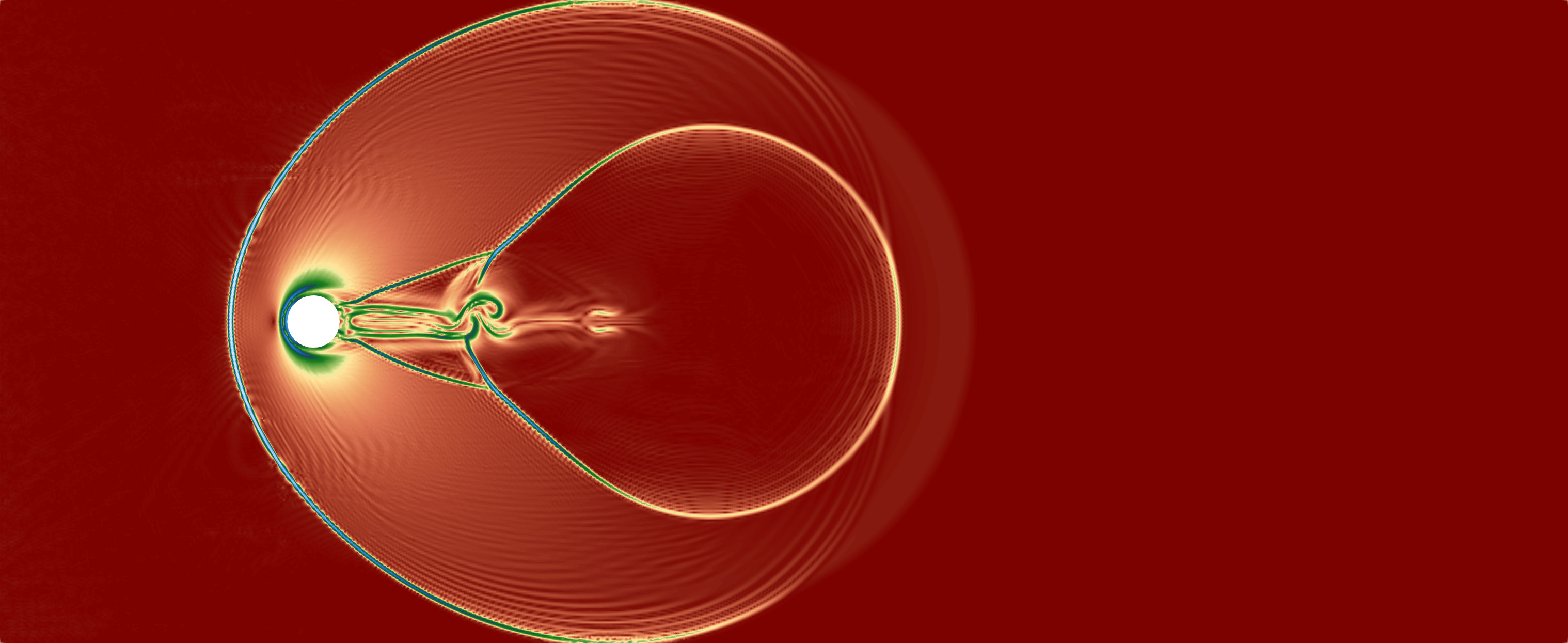}
  \includegraphics[width=\textwidth]{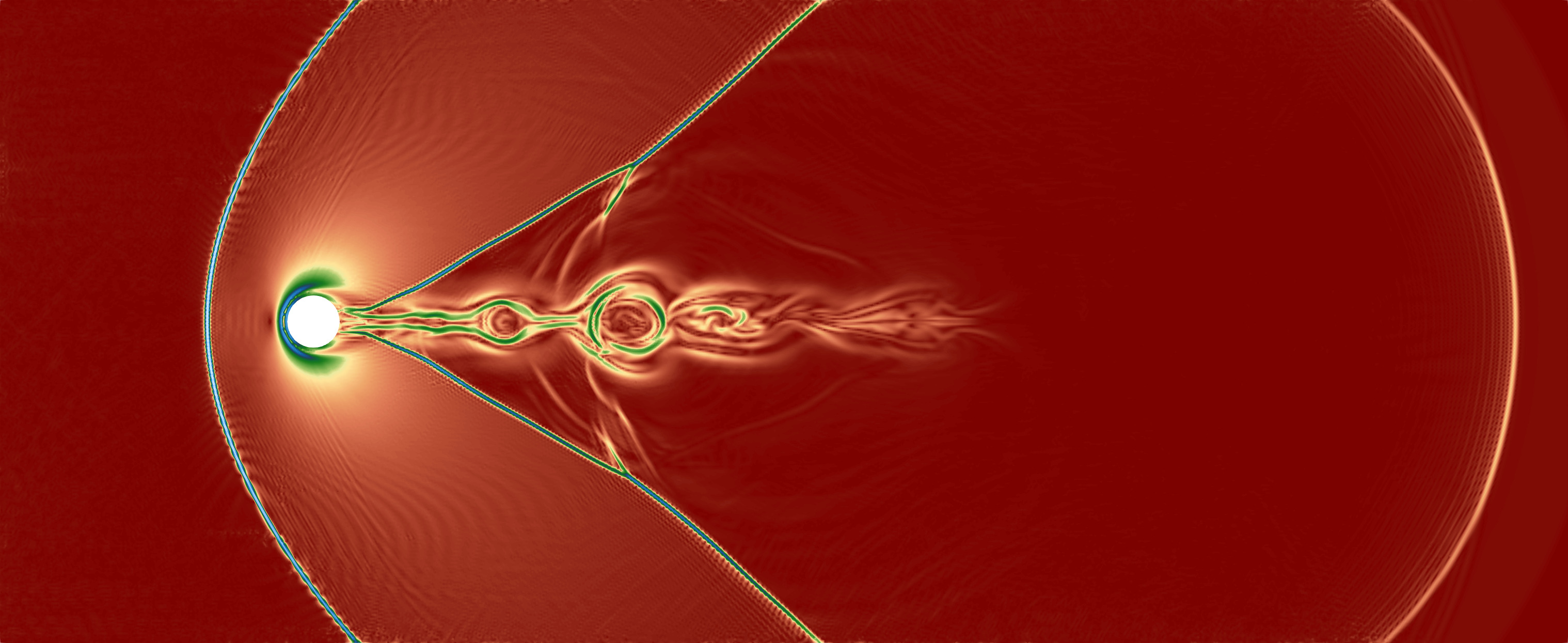}
  \includegraphics[width=\textwidth]{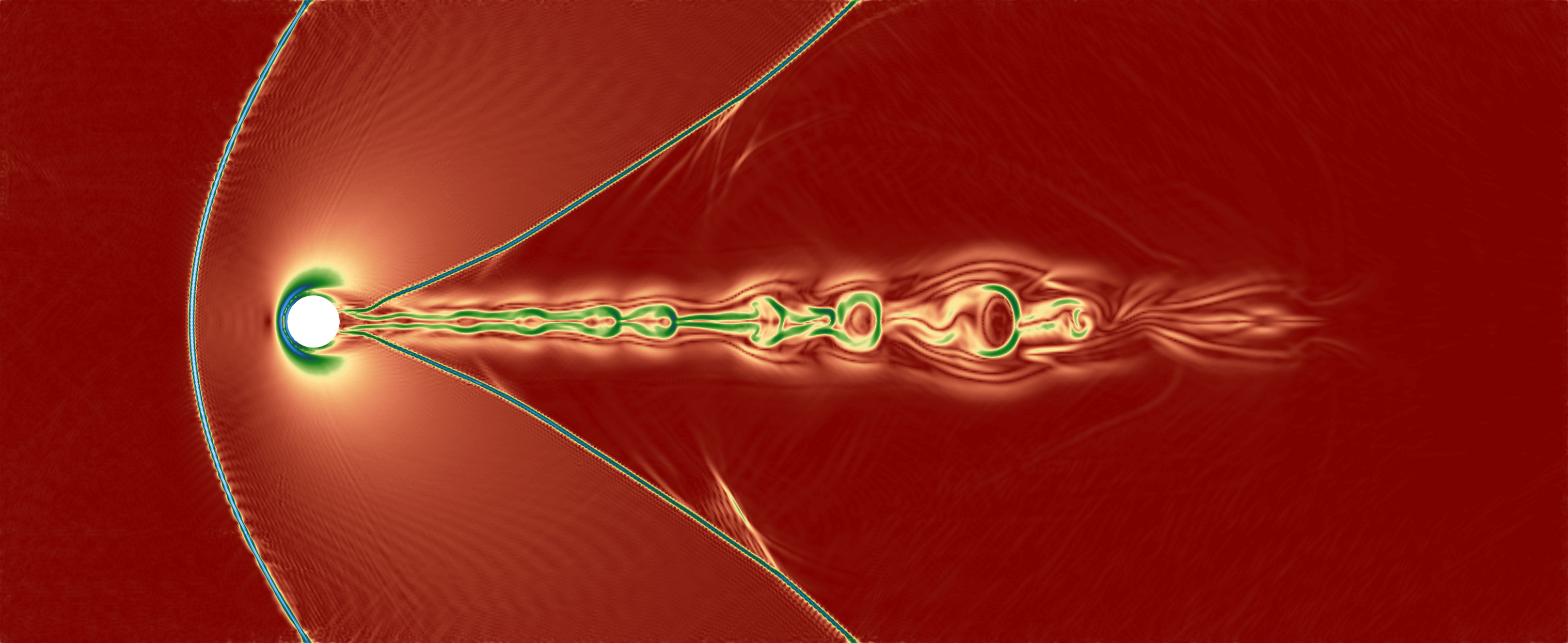}
  \includegraphics[width=\textwidth]{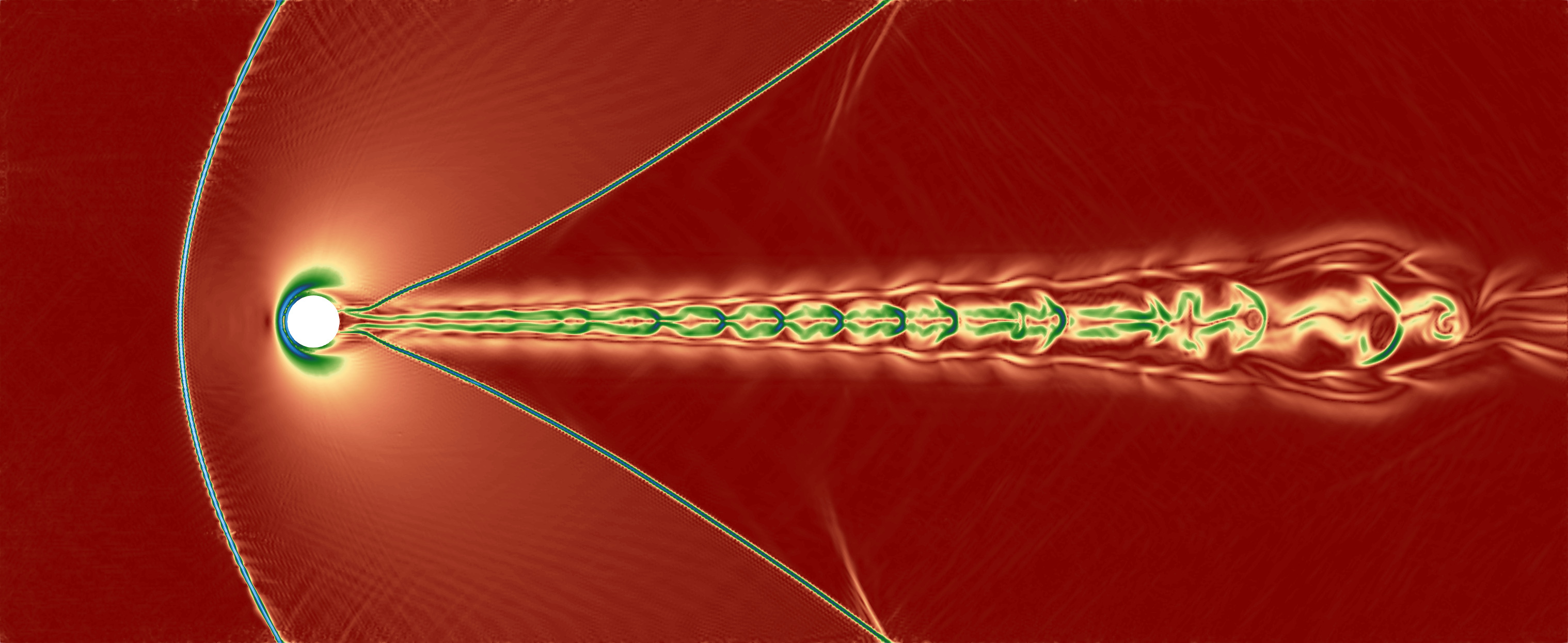}
  \includegraphics[width=\textwidth]{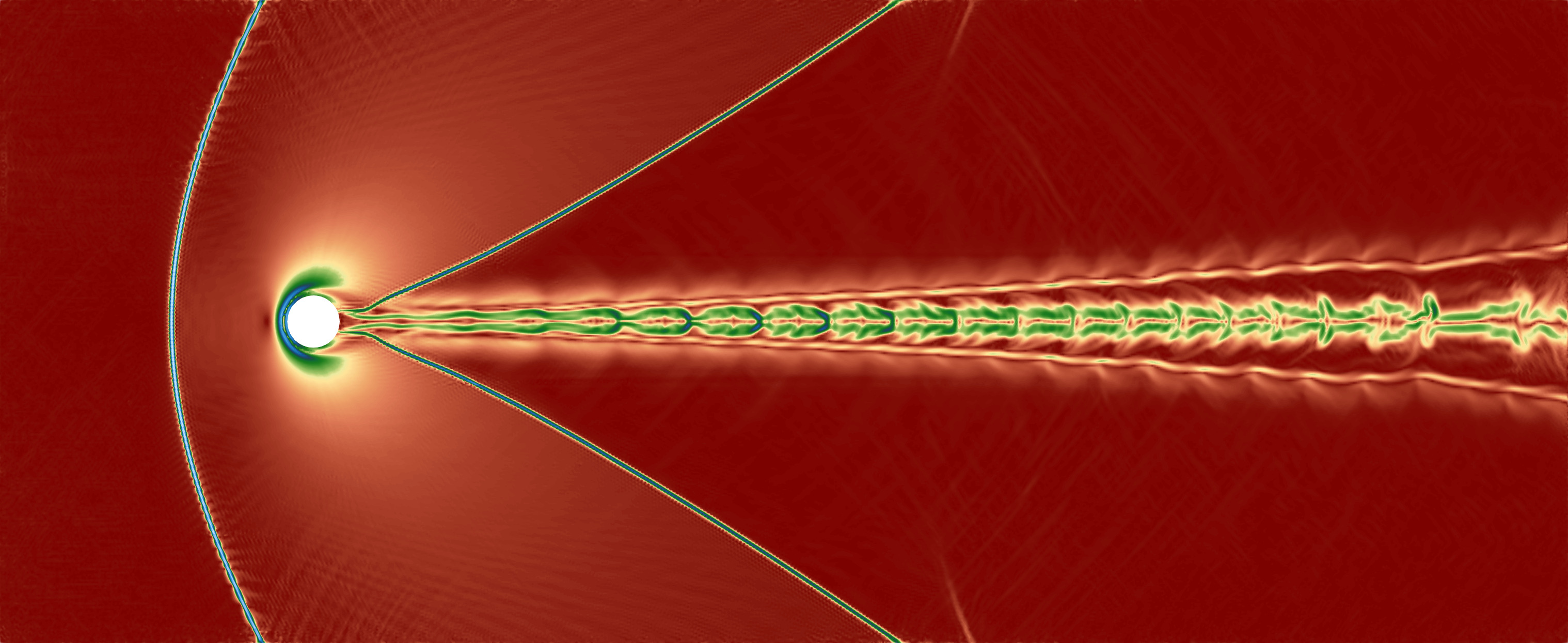}
  \end{subfigure}
  \\
  \vspace{0.1in}
  \begin{subfigure}[t]{0.7\textwidth}
  \centering
    \includegraphics[width=0.7\textwidth]{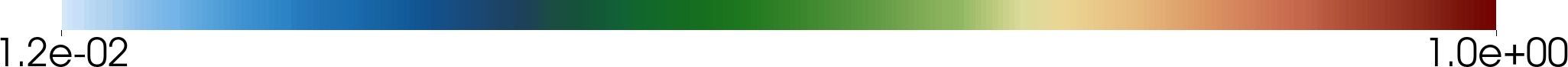}
  \end{subfigure}
  \caption{Supersonic plasma flow past a circular cylinder. Solution for the hydrodynamic regime at the left column and slightly magnetized fluid at the right column. The rows are corresponding to time levels: $t=1,2,3,4,5$. The mesh consists of 242450 $\polP_1$ nodes.}
  \label{fig:cyl1}
\end{figure}

\begin{figure}[h!]
  \centering
  \begin{subfigure}[b]{0.48\textwidth}
  \centering
  $b_y = 0.2$
  \\
  \includegraphics[width=\textwidth]{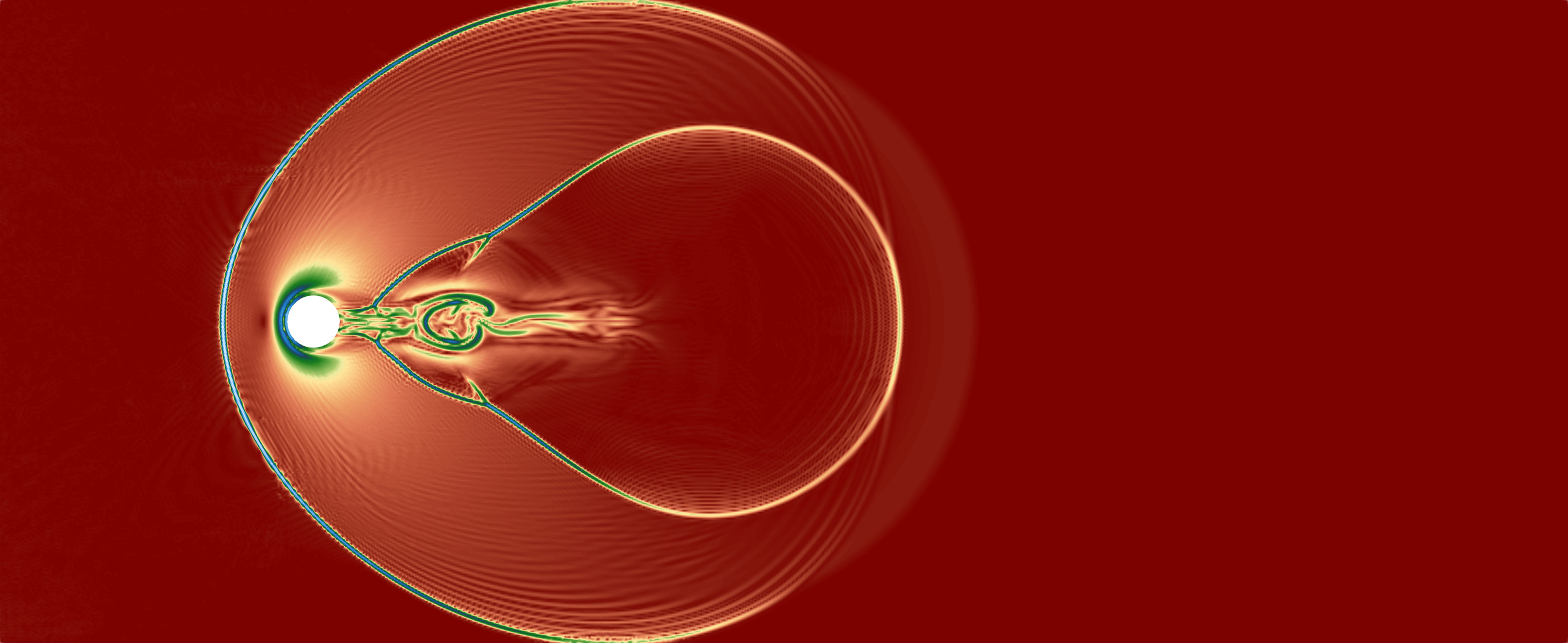}
  \includegraphics[width=\textwidth]{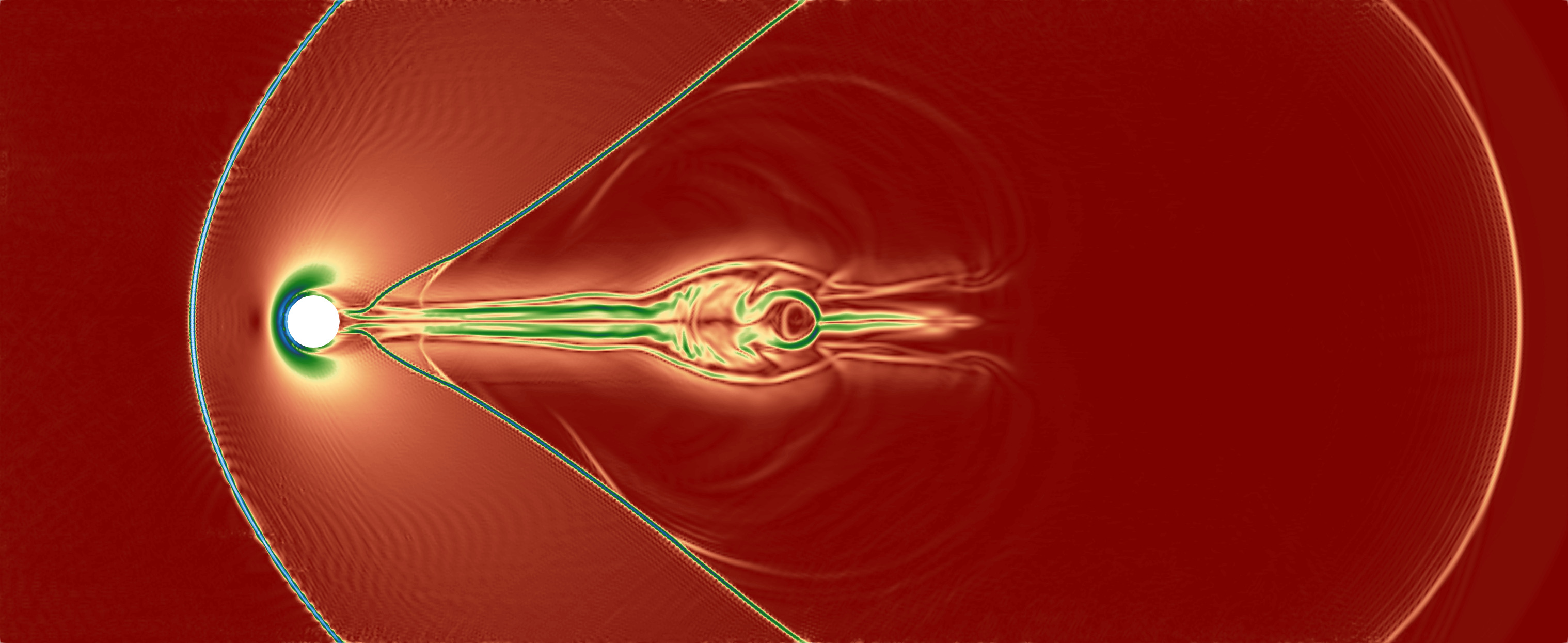}
  \includegraphics[width=\textwidth]{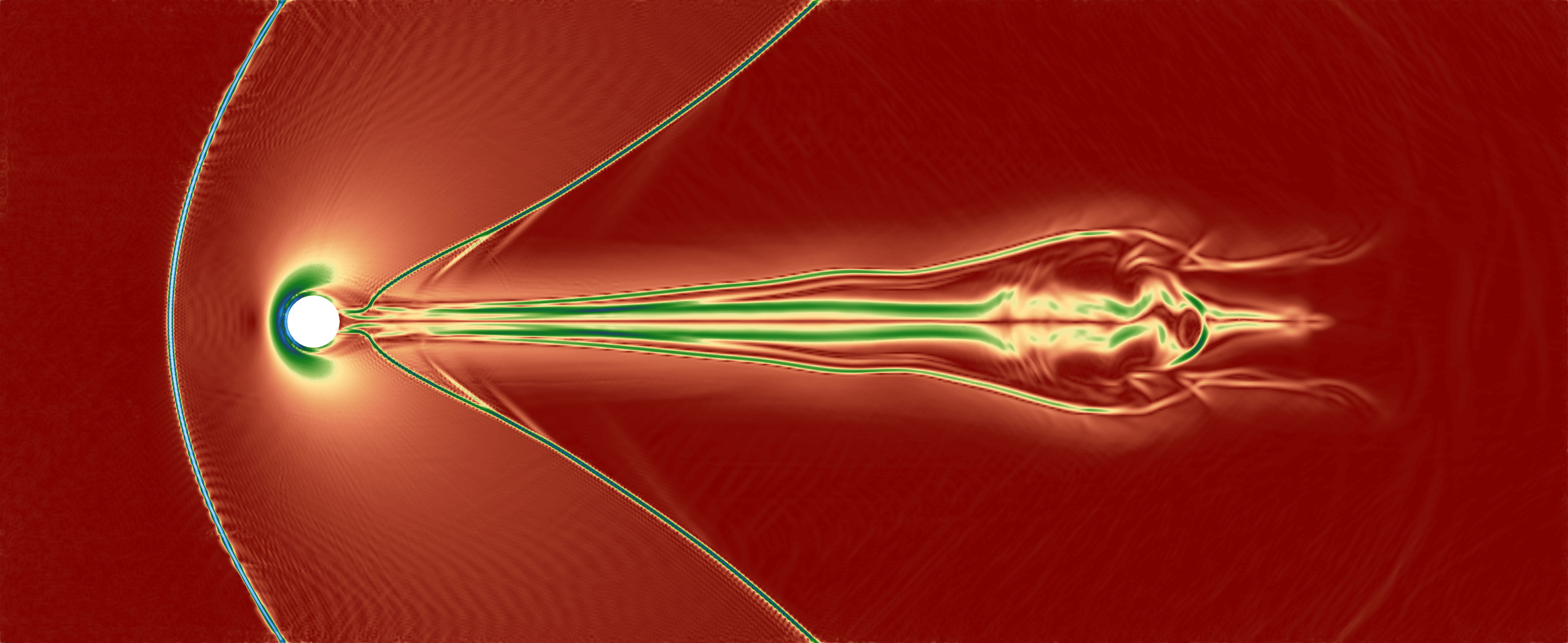}
  \includegraphics[width=\textwidth]{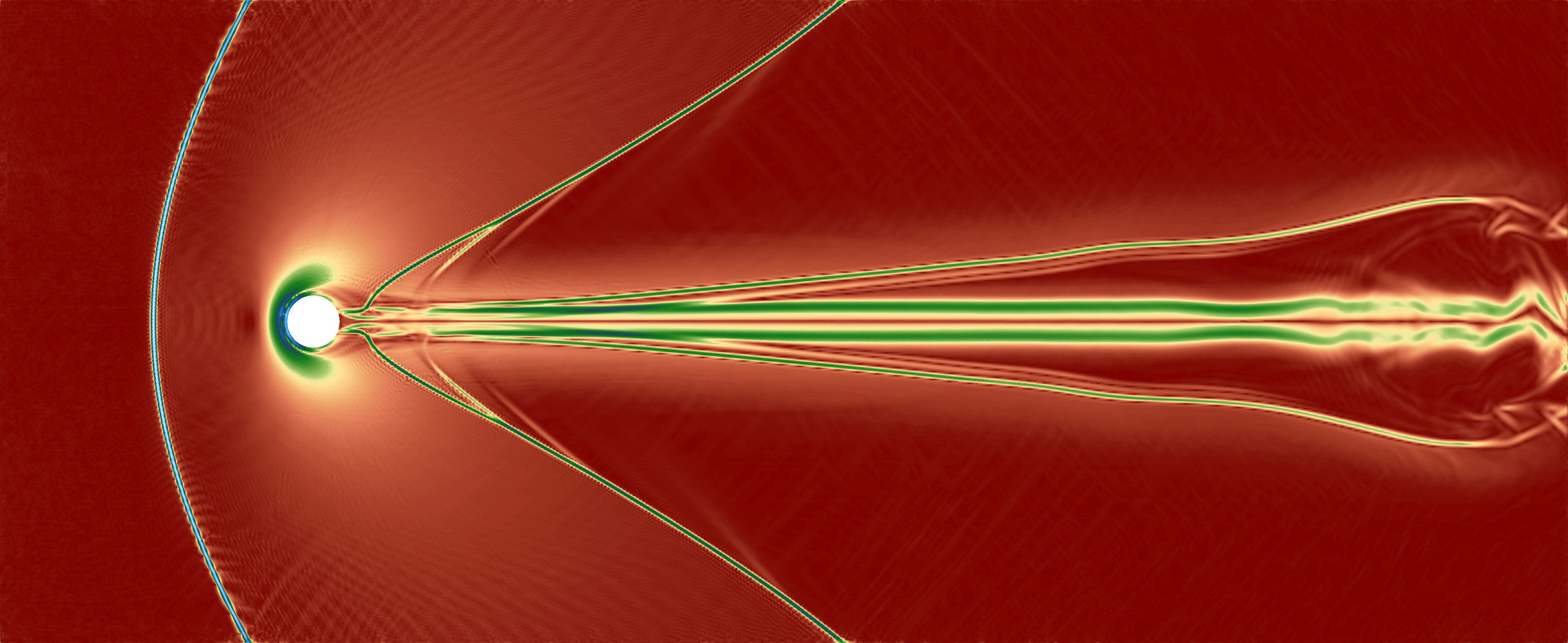}
  \includegraphics[width=\textwidth]{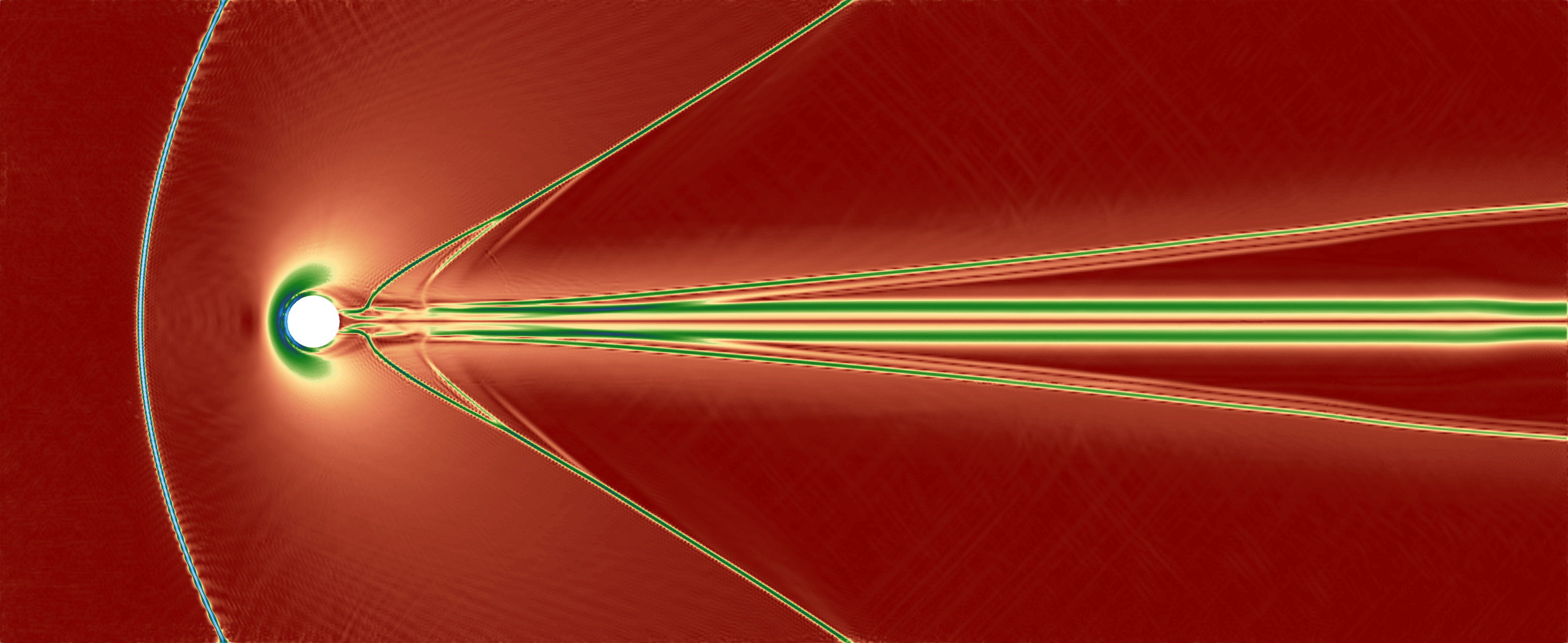}
  \end{subfigure}
  \begin{subfigure}[b]{0.48\textwidth}
  \centering
  $b_y = 0.3$
  \\
  \includegraphics[width=\textwidth]{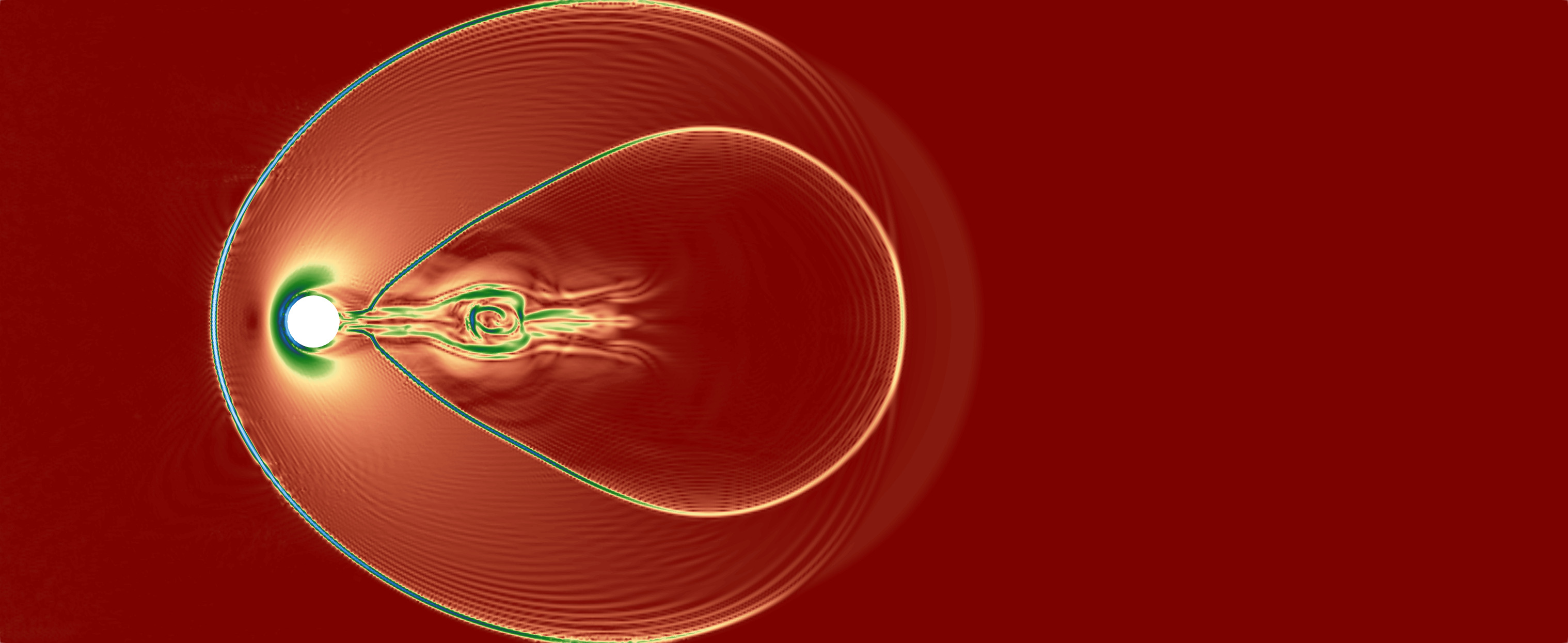}
  \includegraphics[width=\textwidth]{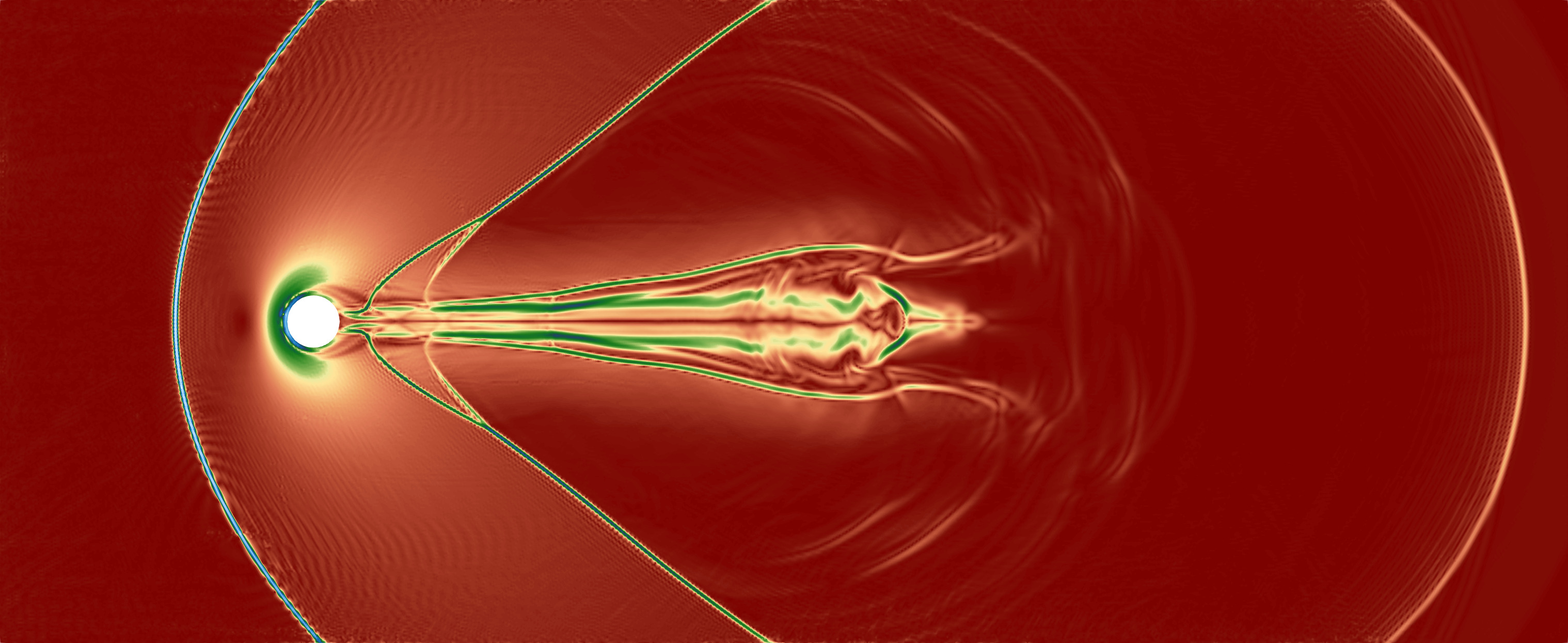}
  \includegraphics[width=\textwidth]{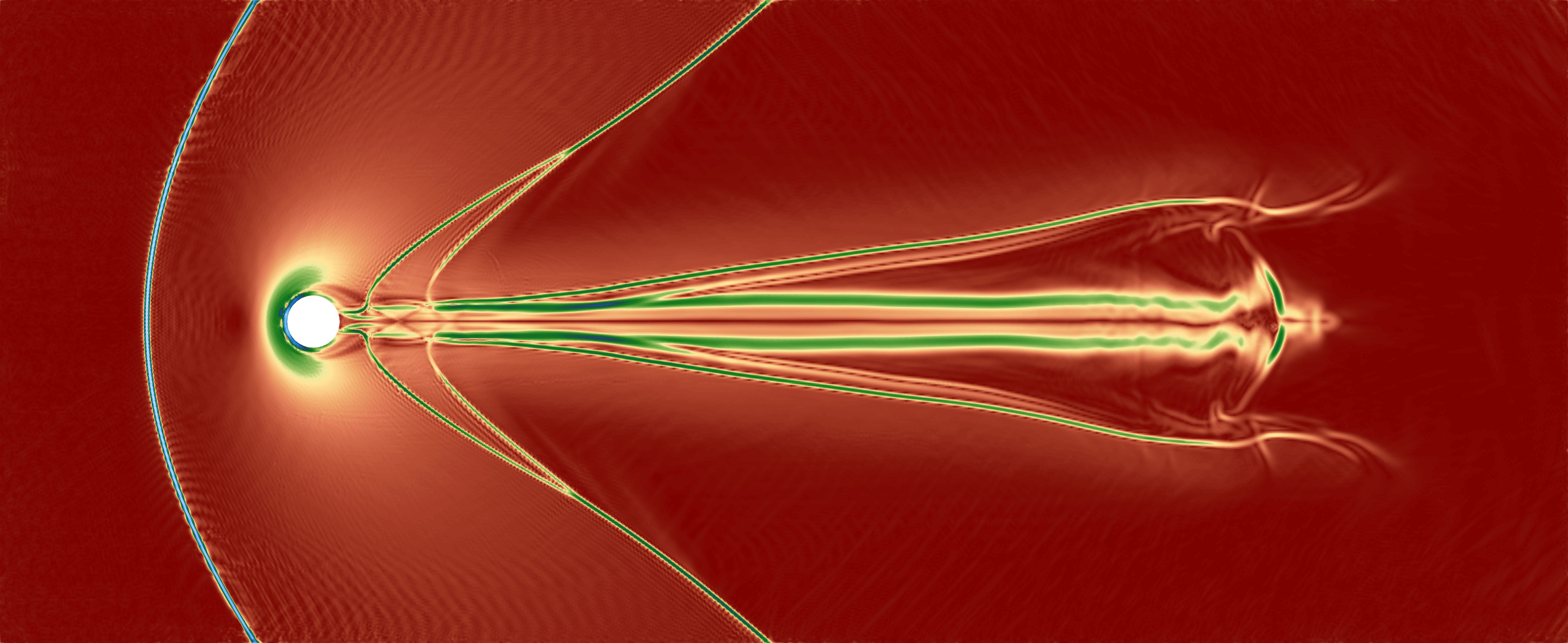}
  \includegraphics[width=\textwidth]{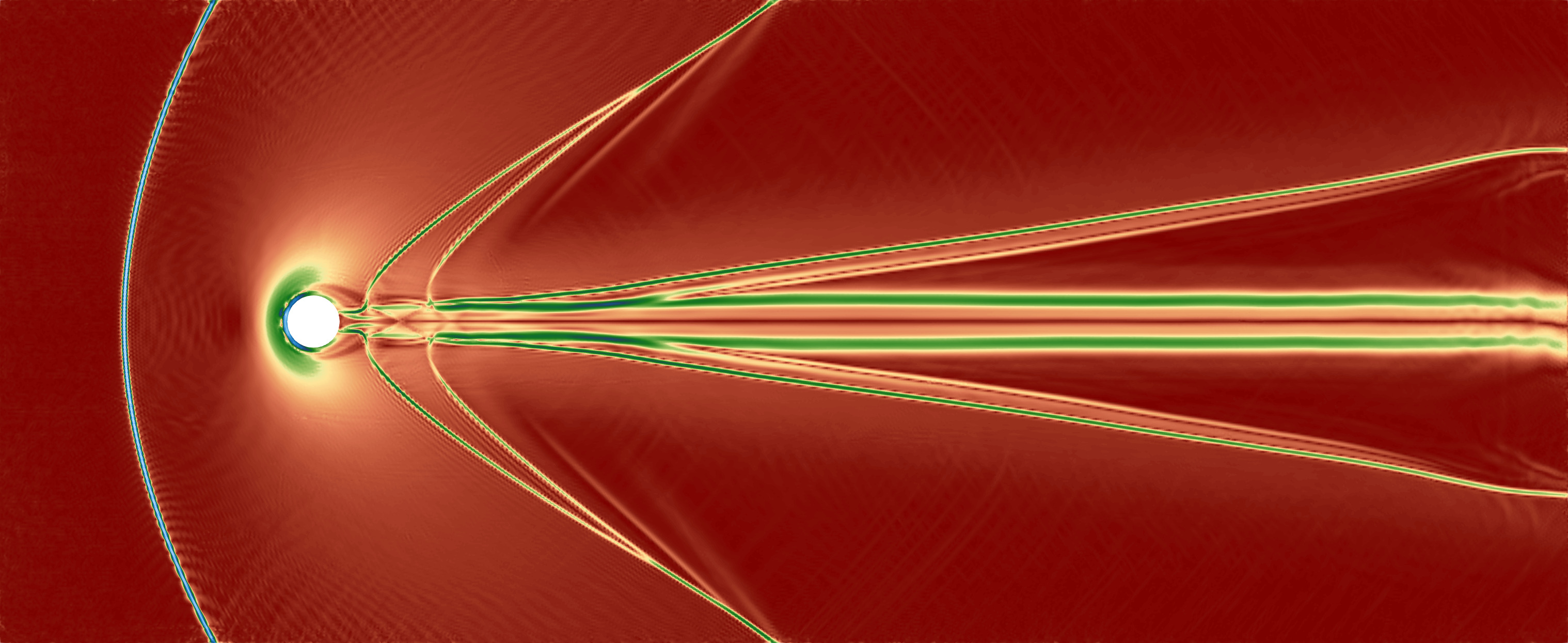}
  \includegraphics[width=\textwidth]{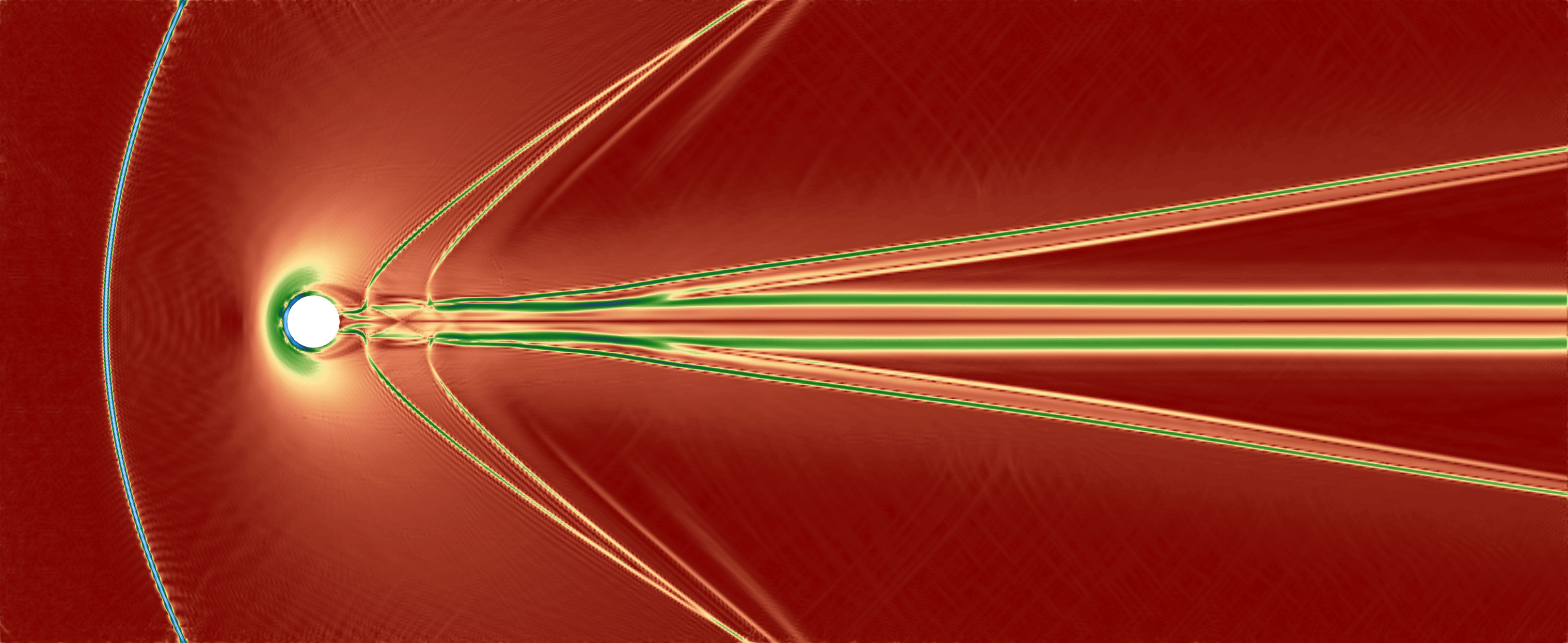}
  \end{subfigure}
  \\
  \vspace{0.1in}
  \begin{subfigure}[t]{\textwidth}
  \centering
    \includegraphics[width=0.7\textwidth]{figures_Cylinder_color_bar_flat}
  \end{subfigure}
  \caption{Supersonic plasma flow past a circular cylinder. Solution for two different magnetic regimes are plotted on each column. The rows are corresponding to time levels: $t=1,2,3,4,5$. The mesh consists of 242450 $\polP_1$ nodes.}
  \label{fig:cyl2}
\end{figure}

\red{
\section*{Conclusion}
In this paper, we presented a new high-order nodal artificial viscosity method for solving ideal MHD equations. The basis of the method is a first-order method, which does not include any special stabilization parameters and does not require explicit determination of the mesh size of the corresponding element. For higher degrees of the polynomial, we construct an additional mesh corresponding to the nodal values of the finite element space and use it when constructing a first-order scheme. Our numerical simulations show that the method adds sufficient viscosity to stabilize the scheme. In addition, we prove a discrete maximum principle for this new nodal viscosity for scalar conservation laws using linear finite elements.

The residual of MHD is then used to make the method a higher order in space. We tested the resulting scheme for several MHD benchmarks, ranging from smooth initial data to strong shock and discontunitie problems. We obtain optimal convergence rates for smooth problems of odd polynomial degrees, which is typical for Galerkin schemes.

Lastly, we remark that even though the scheme captures and resolves shocks and discontinuities, it is not positivity preserving. But it is one of the main ingredients for high-order positivity preserving schemes. A continuation of this work is to extend a limiting methodology similar to \citep{Dao2023}, which is the authors' ongoing research project. 

}

\appendix 

\section{Scalar conservation laws is revisited}
\label{Sec:scalar}
Let $\bef \in \calC^1(\mR; \mR^d)$ be the nonlinear flux and for $q_0(\bx) \in L^\infty(\Omega)$ be some given initial data that have compact support. Let us consider the following scalar conservation laws: 
\begin{equation}\label{eq:scalar}
  \p_t q + \DIV \bef(q) = 0, \quad q(\bx,t) = q_0(\bx), \quad (\bx, t) \in \Omega\times\mR^+,
\end{equation}
with appropriate boundary conditions. We apply the nodal viscosity proposed in this paper for solving the scalar equation. Specifically, we demonstrate that the first-order viscosity, applied to continuous piece-wise linear spaces, yields an approximation that preserves the maximum principle.

\subsubsection{First order method in the $\polP_1$ space}
\label{sec:visc}

The finite element approximation of \eqref{eq:scalar} reads: 
find $q_h(t) \in \calC^1(\mR^+; \calX_h)$ such that 
\begin{equation}\label{eq:fem:s}
  (\p_t q_h + \DIV \bef(q_h), v) = 0, \quad \forall v\in \calX_h.
\end{equation} 
We split the time $[0, \widehat t \,\,]$ into $N$ intervals of variable length. Let for $n=0,\ldots,N$, $t^n$ be the current time and the next time $t^{n+1}$ is computed using the time step $\tau$, \ie $t^{n+1} = t^n + \tau$. Let us denote by $q^n_h := \sum_{j\in\calV} Q^n_j \varphi_j(\bx)$ the finite element approximation of the solution $q(t^n, \bx)$ at time $t^n>0$ with the nodal values $Q_j^n$. Let us assume that the solution at time $t^n$ preserves the discrete maximum principle, \ie
\begin{equation}
  q_{\min}:=\inf_{\bx\in \Omega} q_0(\bx) \le
  \min_{1\le i\le N} Q^n_i\le \max_{1\le i\le N}  Q^n_i  
  \le  \sup_{\bx\in \Omega} q_0(\bx):=q_{\max}. \label{def_umin_umax}
\end{equation}

Next, we approximate the time derivative using the forward Euler method. The finite element approximation \eqref{eq:fem:s} takes the following form:
\begin{equation}\label{eq:lumped}
  m_i \frac{Q^{n+1}_i - Q^n_i}{\tau} + 
  \int_{S_i} \DIV \bef(q_h^n) \varphi_i \ud \bx
  =0,
\end{equation} 
for every $i\in\calV$. Here we know that $m_i>0$, for every $i\in\calV$, for continuous piecewise linear function spaces.

\begin{lemma}[\citep{Guermond_Nazarov_2014}]
  Let $\Phi_K: \widetilde K \mapsto K$ be the affine mapping to transform from a reference equilateral triangle (or thetraderon) $\widetilde K$ to $K$, and let $\polJ_K$ be the Jacobian matrix of this transformation. Then, for any element $K\subset \calT_h$ and shape functions $\varphi_i, \varphi_j \in \calX_h$ there exist constants $\alpha:=\frac{2}{d+1}$, $\gamma:= \frac{2d}{d+1}$, where $d$ is the space dimension, such that
  \begin{equation}\label{eq:intK}
    \int_K (\polJ_K^\top \GRAD \varphi_j) \SCAL (\polJ_K^\top \GRAD \varphi_i) \ud \bx
    = 
    \begin{cases}
      -\alpha \ |K|, \quad i\not=j,\\
      \ \ \, \gamma \ |K|, \quad i=j,
    \end{cases}
  \end{equation}
  $\forall i,j \in \calI(K)$.
\end{lemma}

The viscous bilinear form described by integral \eqref{eq:intK} possesses the precise structure required to establish the positivity, since it yields negative values for all $i \neq j$. This property is commonly referred to as the {\em acute angle condition} in the literature, as discussed in detail, for example, in \citep[Sec.~3.]{Guermond_Nazarov_2014}. We point out that, in the work \citep{Guermond_Nazarov_2014}, the viscosity coefficient is defined at the cell level. In contrast, we will proceed to define the viscosity coefficient at each finite element node.

Upon defining the maximum wave speed defined as 
$
\lambda_{\max,i}(q_h) := \|\bef'(q_h^n)\|_{L^\infty(S_i)},
$
in Definition~\ref{def:nodal}, we obtain the first order viscosity coefficient for the scalar conservation laws:
\begin{equation}\label{eq:eps:scal}
  \e_{i}(q_h(t)) := C_i m_i \|\bef'(q_h^n)\|_{L^\infty(S_i)} \, \max_{i\not= j\in\calI(S_i)} |\GRAD\varphi_j|.
\end{equation}

We next define the global bilinear viscous form as
\begin{equation}\label{eq:bf}
  b(u_h, v_h) := \sum_{K\subset \calT_h}\int_{K} \e_h^n\, \polJ_K\polJ_K^\top \GRAD u_h \SCAL \GRAD v_h \ud \bx, \quad \forall u_h,v_h \in \calX_h,
\end{equation}
and regularize the Galerkin formulation \eqref{eq:lumped} as
\begin{equation}\label{eq:lumped3}
  m_i \frac{Q^{n+1}_i - Q^n_i}{\tau} + 
  \int_{S_i} \DIV \bef(q_h^n) \varphi_i \ud \bx + 
  b(q^{n}_h, \varphi_i)
  =0,
\end{equation} 
for $i\in\calV$ and $n=0,1,\ldots$.

We denote by $K_{\max,i}$ and $K_{\min,i}$ the cell with the largest and smallest volumes in the patch $S_i$ respectively, \ie
\[
  |K_{\max,i}| = \max_{K\in S_i}{|K|}, \mbox{ and } |K_{\min,i}| = \min_{K\in S_i}{|K|}.
\]
where $|K|$ denotes the area or volume of $K$ in 2D and 3D, respectively. 
Then, for the node $\bN_i$, we define a nodal valued mesh quality as
\begin{equation}\label{eq:quality}
  \varkappa_i := \frac{|K_{\max,i}|}{|K_{\min,i}|}.
\end{equation}

\begin{theorem}
  Under the CFL condition 
  \begin{equation}\label{eq:cfl:burger}
    \tau \le
    \Big(
    \Big(1+ \frac{\gamma}{\alpha}
    \max_{j\in \calV} \varkappa_j
    \Big)
    \|\bef' \|_{L^\infty(\Omega)} \max_{j\in \calV}|\GRAD \varphi_j | \,
    \max_{j\in \calV} \varkappa_j
    % \Big(
    % (1+C_i\varkappa_i)
    % \|\bef' \|_{L^\infty(\Omega)} \max_{j\in \{1,\ldots,I\}}|\GRAD \varphi_j |
    \Big)^{-1},
  \end{equation}
  the finite element approximation \eqref{eq:lumped3} preserves the local discrete maximum principle, \ie 
  \[
    q_{\min} \le \min_{j\in \calI(S_i)} Q_j^n \le Q_i^{n+1}
    \le \max_{j\in \calI(S_i)} Q_j^n \le q_{\max}, 
  \]
  for any $n\ge0$ and $i\in \calV$.
\end{theorem}

\begin{proof} 
  We perform the proof for $d>1$. For the case of $d=1$ the proof is trivial. 
  After moving all known quantities to the right hand side of \eqref{eq:lumped3}, using the identity $\DIV \bef(q^n_h) = \bef'(q^n_h) \SCAL \GRAD q^n_h$ and multiplying by $\tau/m_i$, we obtain:
  \[
    Q^{n+1}_i = Q^n_i - \frac{\tau}{m_i}
    \Big(
    \int_{S_i} \bef'(q^n_h) \SCAL \GRAD q^n_h \, \varphi_i \ud \bx + 
    \sum_{K\subset S_i} \int_{K} \e_h(q^n_h)\, (\polJ_K^\top \GRAD q_h^n) \SCAL (\polJ_K^\top \GRAD \varphi_i) \ud \bx
    \Big),
  \]
  and 
  \begin{align*}
    Q^{n+1}_i 
    = &\ 
        Q^n_i\Big( 
        1 - \frac{\tau}{m_i} 	
        \sum_{K\subset S_i}\int_{K} 
        \Big(\bef'(q_h^n) \SCAL \GRAD \varphi_i \varphi_i + 
        \e_h(q^n_h)\, | \GRAD (\polJ_K^\top \GRAD \varphi_i)|^2 
        \Big)\ud \bx
        \Big) \\
      &\quad
        +
        \sum_{K\subset S_i} 
        \sum_{j\in \calI(K),\, j\not=i} Q^n_j \,\,
        \frac{\tau}{m_i}
        \int_K
        \Big(
        -
        \bef'(q_h^n) \SCAL \GRAD \varphi_j \varphi_i 
        - 
        \e_h(q^n_h)\, \GRAD (\polJ_K^\top \GRAD \varphi_j) \SCAL (\polJ_K^\top \GRAD \varphi_i)
        \Big) \ud \bx \\
    := &\
         a Q^n_i + \
         \sum_{K\subset S_i} 
         \sum_{j\in \calI(K),\, j\not=i} b_{j} Q^n_j \,\,
  \end{align*}
  \ie we obtained that the solution at time $t^{n+1}$ is constructed by a linear combination of the solution at time $t^n$. Here we observe that if 1) $a$ and $b_j$, for all $j\not=i$ and $i,j\in\calV$, are non-negative constants, and 2) $a+\sum_{K\subset S_i} 
  \sum_{j\in \calI(K),\, j\not=i} b_{j}  =1$, then the convex combination property gives that $Q^{n+1}_i$ preserves the discrete maximum principle. 

  We start by proving the following identity:
  \begin{equation}\label{eq:th:integration}
    \int_K \e_h(q^n_h)\, (\polJ_K^\top \GRAD \varphi_j) \SCAL (\polJ_K^\top \GRAD \varphi_i) \ud \bx = \frac{1}{d+1} \sum_{l\in \calI(K)} \e^n_h(\bN_l)\,
    \int_K (\polJ_K^\top \GRAD \varphi_j) \SCAL (\polJ_K^\top \GRAD \varphi_i)
    \ud \bx.
  \end{equation}
  In fact, since $ \GRAD \widetilde \varphi_j \SCAL \GRAD \widetilde \varphi_i$, where $\widetilde \varphi_k:= \Phi_K \circ \varphi_k$, for all $k\in \calV$,  is constant on the cell $\widetilde K$ and $\e^n_h$ is a linear function, $|\det (\polJ_K)| = |K|/|\widetilde K|$, using the Trapezoidal integration rule, we obtain:
  \[
    \begin{aligned}
      \int_K \e_h(q^n_h)\, (\polJ_K^\top \GRAD \varphi_j) \SCAL (\polJ_K^\top \GRAD \varphi_i) \ud \bx 
      = & \ 
          \int_{\widetilde K} \e_h(q^n_h(\widetilde\bx))\,  \GRAD \widetilde \varphi_j \SCAL \GRAD \widetilde \varphi_i \, |\det(\polJ_K)| \ud \widetilde\bx \\
      = & \ 
          \GRAD \widetilde \varphi_j \SCAL \GRAD \widetilde \varphi_i \int_{\widetilde K} \e_h(q^n_h(\widetilde\bx))\, |\det(\polJ_K)| \ud \widetilde\bx \\
      = & \ 
          \GRAD \widetilde \varphi_j \SCAL \GRAD \widetilde \varphi_i \int_{K} \e_h(q^n_h(\bx)) \ud \bx \\
      = & \
          \GRAD \widetilde \varphi_j \SCAL \GRAD \widetilde \varphi_i 
          \sum_{l \in \calI(K)} \e^n_h(\bN_l)\, \frac{|K|}{d+1} \\
      = & \
          \GRAD \widetilde \varphi_j \SCAL \GRAD \widetilde \varphi_i \, 
          \frac{1}{d+1}\sum_{l\in \calI(K)} \e^n_l \, \frac{|K|}{|\widetilde K|} |\widetilde K| \\
      = & \
          \frac{1}{d+1}\sum_{l\in \calI(K)} \e^n_l \
          \int_{\widetilde K}
          \GRAD \widetilde \varphi_j \SCAL \GRAD \widetilde \varphi_i \, 
          |\det (\polJ_K)| \ud \widetilde \bx 
      \\
      = & \
          \frac{1}{d+1} \sum_{l\in \calI(K)} \e^n_l\,
          \int_K (\polJ_K^\top \GRAD \varphi_j) \SCAL (\polJ_K^\top \GRAD \varphi_i)
          \ud \bx.
    \end{aligned}
  \]
  Also, note that $\e^n_l=\e^n_h(\bN_l)$, a nodal artificial viscosity defined in \eqref{eq:eps}. 

  Next, using the identity $\int_K \varphi_j \ud \bx = \frac{1}{d+1}|K|$, $j\in\calV$, we have the following estimate for the viscosity coefficient:
  \begin{equation}\label{eq:th:eps_est}
    \begin{aligned}
      \e^n_i 
      = &\
          \frac{1}{\alpha \Nel(S_i)} \frac{1}{\min_{K\subset S_i} |K|}\, m_i \, \|\bef'\|_{L^{\infty}(S_i)} \max_{i\not= j\in\calI(S_i)} |\GRAD\varphi_j| \\
      \le &\
            \frac{1}{\alpha \Nel(S_i)} \frac{1}{\min_{K\subset S_i}|K|} 
            \frac{1}{d+1} \Nel(S_i)\max_{K\subset S_i}|K| \,
            \|\bef'\|_{L^{\infty}(\Omega)} \max_{j\in\calV} |\GRAD\varphi_j| \\
      \le &\
            \frac{\varkappa_i}{\alpha} \frac{1}{d+1}
            \|\bef'\|_{L^{\infty}(\Omega)} \max_{j\in\calV} |\GRAD\varphi_j|
    \end{aligned}.
  \end{equation}

  Using \eqref{eq:intK}, \eqref{eq:th:eps_est}, \eqref{eq:th:integration} and the CFL confition \eqref{eq:cfl:burger} we obtain: 
  \begin{equation}
    \begin{aligned}
      a = &\,
            1 - \frac{\tau}{m_i}
            \Bigg[ 	
            \sum_{K\subset S_i}\int_{K} 
            \bef'(q_h^n) \SCAL \GRAD \varphi_i \varphi_i \ud \bx 
            + 
            \frac{1}{d+1}\sum_{K\subset S_i}
            \sum_{l\in \calI(K)} \e^n_l \int_K |\polJ_K^\top \GRAD \varphi_i |^2 \ud \bx
            \Bigg]
      \\
      \ge &\,
            1 - \|\bef'\|_{L^\infty(\Omega)} \max_{j\in\calV} |\GRAD\varphi_j| \,
            \frac{\tau}{m_i}
            \Bigg[
            m_i + \frac{1}{d+1} \sum_{K\subset S_i}\sum_{l\in \calI(K)}\frac{\varkappa_l}{\alpha} \frac{1}{d+1} \gamma |K|
            \Bigg] \\
      \ge &\,
            1 - \|\bef'\|_{L^\infty(\Omega)} \max_{j\in\calV} |\GRAD\varphi_j| \,
            \frac{\tau}{m_i}
            \Bigg[
            m_i + \frac{\gamma}{\alpha} \max_{j\in\calV}\varkappa_j
            \sum_{K\subset S_i} \frac{1}{d+1} |K|
            \Bigg] \\
      = &\,
          1 - \|\bef'\|_{L^\infty(\Omega)} \max_{j\in\calV} |\GRAD\varphi_j| \,
          \frac{\tau}{m_i} \,
          m_i
          \Bigg[
	  1 + \frac{\gamma}{\alpha} \max_{j\in\calV}\varkappa_j
	  \Bigg] \\
      \ge &\,
            0.
    \end{aligned}
  \end{equation}

  Now, we prove that $b_j \ge 0$ for any $j\in\calV$. For that, it is sufficient to prove that the integral is nonnegative. Using \eqref{eq:th:integration}, the definition of the artificial viscosity coefficient \eqref{eq:eps:scal} and the estimate
  $
  m_l \ge \Nel(S_l) \frac{1}{d+1}\min_{K\subset S_l} |K|,
  $
  for any $l\in \calV$, we obtain:
  \[
  \begin{aligned}
    -\int_{K} &\
                \Big(
                \bef'(q_h^n) \SCAL \GRAD \varphi_j \varphi_i + 
                \e_h(q^n_h)\, \GRAD (\polJ_K^\top \GRAD \varphi_j) \SCAL (\polJ_K^\top \GRAD \varphi_i)
                \Big) \ud \bx \\
    = &\,
        - \int_{K} 
        \Big(
        \bef'(q_h^n) \SCAL \GRAD \varphi_j \varphi_i \ud \bx 
        -
        \frac{1}{d+1} \sum_{l\in\calI(K)} 
        \e_l^n
        \int_K \GRAD (\polJ_K^\top \GRAD \varphi_j) \SCAL (\polJ_K^\top \GRAD \varphi_i) \ud \bx \\
    = &\,
        - \int_{K} 
        \Big(
        \bef'(q_h^n) \SCAL \GRAD \varphi_j \varphi_i \ud \bx \\
              &\, \hspace{0.5in} -
                \frac{1}{d+1} \sum_{l\in\calI(K)} 
                C_l m_l \|\bef'\|_{L^{\infty}(S_l)}\max_{j\in \calI(S_l)}|\GRAD \varphi_j|
                \int_K \GRAD (\polJ_K^\top \GRAD \varphi_j) \SCAL (\polJ_K^\top \GRAD \varphi_i) \ud \bx \\
    = &\,
        -\|\bef'\|_{L^{\infty}(K)} \max_{j\in\calI(K)}|\GRAD \varphi_j| 
        \int_K \varphi_i \ud \bx \\
              &\, \hspace{0.5in} -
                \frac{1}{d+1} \sum_{l\in\calI(K)} 
                C_l\,m_l\,
                \|\bef'\|_{L^{\infty}(S_l)}\max_{j\in \calI(S_l)}|\GRAD \varphi_j|
                (-\alpha |K|) \\ 
    \ge \,\, &
               \|\bef'\|_{L^{\infty}(K)} \max_{j\in\calI(K)}|\GRAD \varphi_j| \,\, 
               \Bigg(
               -\int_K \varphi_i \ud \bx \\
              &\, \hspace{0.5in} 
                +
                \frac{1}{d+1} \alpha |K| \sum_{l\in\calI(K)} 
                \frac{1}{\alpha \Nel(S_l)} \frac{1}{\min_{K\in S_l}|K|} \,
                \, \frac{1}{d+1} \Nel(S_l) \min_{K\in S_l}|K|
                \Bigg)
    \\ 
    = \,\,&
            \|\bef'\|_{L^{\infty}(K)} \max_{j\in\calI(K)}|\GRAD \varphi_j| \,\, 
            \Bigg(
            -\int_K \varphi_i \ud \bx 
            +
            \frac{1}{d+1} |K|       \Bigg)
    \\ 
    = \,\, & 
             0.
  \end{aligned}
  \]

  Finally, to show the second property, we use the partition of unity property:
\[
  \begin{aligned}
    a + \sum_{K\subset S_i} &\ 
    \sum_{j\in \calI(K),\, j\not=i} b_{j} \\
    =&\
       1 - 
       \frac{\tau}{m_i}
       \sum_{K\subset S_i} 
       \sum_{j\in \calI(K)}
       \int_{K} 
       \Big(
       \bef'(q^n_h) \SCAL \GRAD \varphi_j \, \varphi_i \ud \bx + 
       \e_h(q^n_h)\, (\polJ_K^\top \GRAD \varphi_j) \SCAL (\polJ_K^\top \GRAD \varphi_i) 
       \Big) \ud \bx
    \\
    =&\
       1 - 
       \frac{\tau}{m_i}
       \sum_{j\in \calI(S_i)}
       \sum_{K\subset S_{ij}} 
       \int_{K} 
       \Big(
       \bef'(q^n_h) \SCAL \GRAD \varphi_j \, \varphi_i \ud \bx + 
       \e_h(q^n_h)\, (\polJ_K^\top \GRAD \varphi_j) \SCAL (\polJ_K^\top \GRAD \varphi_i)     
       \Big)
       \ud \bx
    \\
    = &\
	1 - 
	\frac{\tau}{m_i} 	
	\sum_{K\subset S_{ij}}
	\int_{K} 
	\Big(
	\bef'(q_h^n) \SCAL\GRAD \Big(\!\! \sum_{j\in \calI(S_i)} \varphi_j \Big) \varphi_i + 
	\e_h(q^n_h)\, \Big(\polJ_K^\top \GRAD \Big( \!\! \sum_{j\in \calI(S_i)}\varphi_j \Big) \Big) \SCAL (\polJ_K^\top \GRAD \varphi_j)
	\Big) \ud \bx \\
    = &\
	1.
  \end{aligned}
\]

  The proof is completed.
\end{proof}

\begin{remark}
  \label{remark:cfl}
  Note that $\frac{\gamma}{\alpha} = d$. Upon defining the global variables
  \begin{equation}
    \begin{aligned}
      h := \min_{j\in\calV} \big(|\GRAD \phi_j | \big)^{-1}, 
      \quad
      \varkappa := \max_{j\in\calV} \varkappa_j,
      \quad
      \beta: = \|\bef'(q_h^n) \|_{L^\infty(\Omega)}, \quad
      \textrm{CFL} :=  \frac{1}{(1 + d\varkappa)\varkappa},
    \end{aligned}
  \end{equation}
  the CFL condition \eqref{eq:cfl:scal} can be rewritten as 
  \begin{equation}\label{eq:cfl:scal}
    \tau \le \textrm{CFL} \frac{h}{\beta},
  \end{equation}
  which is a usual CFL condition to approximate for first order hyperbolic problems. Note that for quasi-uniform meshes $\varkappa\approx 1$ and the CFL condition is $\frac12$ in 1D, $\frac13$ in 2D, and $\frac14$ in 3D. 
\end{remark}

%\begin{remark}[Units]
%  The unit of the artificial viscosity is speed divided by length. However the unit of the Jacobian matrix $\polJ_K$ is proportional to length. As a result, the unit of $\e_h \polJ_K \polJ_K^\top$ is the product of speed and length, which aligns with physical kinematic viscosity.
%\end{remark}

\begin{remark}
  Taking a maximum over the patch in the definition of the artificial viscosity coefficient \eqref{eq:eps} makes the viscosity coefficient rather large. For instance, to prove the positivity of coefficient $b_j$ in the proof of the theorem, it is enough to have 
  \begin{equation}\label{eq:eps_ij}
    \e_{ij} = C_i\ 
    \Big|
    \int_{S_{ij}} \bef'(q_h^n) \SCAL \GRAD \varphi_j \varphi_i \ud \bx
    \Big|, 
  \end{equation}
  which is an edge based viscosity and is sharper than taking a maximum of the flux term over the patch. However, \eqref{eq:eps_ij} is a matrix and must be assembled at every time-step, whereas \eqref{eq:eps} is a vector, which is usually cheaper to assemble. Since in this work, we are interested in the construction of viscosity operator for higher-order polynomial spaces, we are interested in \eqref{eq:eps} since it is faster to compute.
\end{remark}

%\newpage
%\pagebreak
%\bibliographystyle{siam}
%\bibliographystyle{abbrvnat} 
\bibliographystyle{plainnat}
\bibliography{ref}
\end{document}